\documentclass[reqno]{amsart}
\usepackage{amssymb}
\usepackage{amsmath}
\usepackage[mathscr]{euscript}
\usepackage{pgf,tikz} 
\usetikzlibrary{arrows}
\usepackage[small]{caption}
\usepackage{psfrag}
\usepackage{hyperref}

\makeatletter
\@addtoreset{equation}{section}
\makeatother

\newtheorem{theorem}{Theorem}[section]
\newtheorem{lemma}[theorem]{Lemma}
\newtheorem{proposition}[theorem]{Proposition}
\newtheorem{corollary}[theorem]{Corollary}

\newtheorem{remark}[theorem]{Remark}

\newcommand{\Exp}[3]{\mathbf{E}_{#1}^{#2}\left[#3\right]}
\newcommand{\Prob}[3]{\mathbf{P}_{#1}^{#2}\left[#3\right]}
\newcommand{\mc}[1]{{\mathcal #1}}

\newcommand{\mb}[1]{{\mathbf #1}}
\newcommand{\bb}[1]{{\mathbb #1}}
\newcommand{\bs}[1]{{\boldsymbol #1}}

\begin{document}

\title[ABC model in the zero-temperature limit]{Evolution of the ABC model among the segregated configurations in the zero-temperature limit}

\author[R. Misturini]{Ricardo Misturini}

\begin{abstract}
We consider the ABC model on a ring in a strongly asymmetric regime. The main result asserts that the particles almost always form three pure domains (one of each species) and that this segregated shape evolves, in a proper time scale, as a Brownian motion on the circle, which may have a drift. This is, to our knowledge, the first proof of a zero-temperature limit for a non-reversible dynamics whose invariant measure is not explicitly known.
\end{abstract}

\address{\noindent IMPA, Estrada Dona Castorina 110, CEP 22460 Rio de
   Janeiro, Brasil.  \newline e-mail: \rm
   \texttt{misturini@impa.br} }



\maketitle

\section{Introduction} The \emph{ABC model}, introduced by Evans et al. \cite{ekkm1,ekkm2}, is a stochastic conservative dynamics consisting of three species of particles, labeled $A$, $B$, $C$, on a discrete ring $\{-N,\ldots,N\}$ (one particle per site). The system evolves by nearest neighbor transpositions: $AB\to BA$, $BC \to CB$, $CA \to AC$ with rate $q$ and $BA\to AB$, $CB\to BC$, $AC\to CA$ with rate $1$.

The asymptotic behavior of the process (and of its variations) has been widely studied in the \emph{weakly asymmetric} regime $q=e^{-\beta/N}$, introduced by Clincy et al.~\cite{cde}, when the system size $N$ goes to infinity and $\beta$ is a fixed control parameter which plays the role of the inverse temperature. In this regime, an interesting phase transition phenomenon arises as $\beta$ is tuned.

We investigate here a \emph{strongly asymmetric} regime, the zero-temperature limit, where $q=e^{-\beta}$, $\beta\uparrow\infty$. We consider two types of asymptotics: In Theorem~\ref{theo1} we examine the behavior of the process in the case where the number of particles of each species, $N_A$, $N_B$ and $N_C$, is fixed and $\beta\uparrow\infty$; in Theorem~\ref{mainresult}, $N_A$, $N_B$ and $N_C$ increase with $\beta$. 

We show in Lemma \ref{lemma3} that the particles almost always form three pure domains, one of each
species, located clockwise in the cyclic-order $ABC$. For fixed volume, we show
that, in the time scale $e^{\min\{N_A,N_B,N_C\}\beta}$, as $\beta\uparrow\infty$, the process converges to a Markov chain that evolves among these $2N+1$ segregated configurations, jumping from any configuration to any other at positive rates. These jump rates can be expressed in terms of some absorption probabilities of a much simpler dynamics. 

When $N$ grows with~$\beta$, with some restrictions on the speed of this growth, we prove in Theorem~\ref{mainresult} that, in the time scale $N^2e^{\min\{N_A,N_B,N_C\}\beta}$, the center of mass of the particles of type $A$ (for example) moves as a Brownian motion on the circle. Without assuming positive proportion of each type of particle, we identify an interesting degenerated case in which the limit Brownian motion has a drift.

Our method for proving Theorem~\ref{mainresult} involves the analysis of the trace process
in the set of the segregated configurations, the process which neglects the time spent in other configurations. 

Results of the same nature (the description of the dynamics among the ground states in finite volume and the convergence to a Brownian motion on a large torus) were obtained for the Kawasaki dynamics for the Ising lattice gas at low temperature in two dimensions by Beltr\'an, Gois and Landim~\cite{bl5,gl5}. Many techniques used in our analysis of the ABC model come from these papers. We emphasize that, in comparison with the Kawasaki dynamics, a significant difference is that, with the exception of the case $N_A=N_B=N_C$, the ABC process is non-reversible and its invariant measure is not explicitly known. 

In this strongly asymmetric regime we are dealing, the model fits the assumptions considered by Olivieri and Scoppola in \cite{os1}. In principle, the general procedure proposed by them, consisting in the analysis of successive time-scales, $e^{\beta}$, $e^{2\beta}, \ldots$, could be applied here, especially if we were interested in understanding the mechanism of nucleation of the process starting from an arbitrary configuration. However, the analysis based on this iterative scheme would become quite complicated in the ABC model due to the combinatorial complexity of the evolution among the metastable configurations which would appear in each scale. Our analysis, which relies on a precise understanding of the microscopic dynamics when the process is close to one of the segregated configurations, leads to a very accurate understanding of the limiting process in the time-scale required for transitions among the most stable configurations.

\section{Notations and results}

\subsection{The ABC Process}

Given an integer $N$, let $\Lambda_N=\{-N,\ldots,N\}$ be the one-dimensional discrete ring of size $2N+1$. A configuration in $\widetilde{\Omega}^N:=\{A,B,C\}^{\Lambda_N}$ is denoted by $\omega=\{\omega(k): k\in\Lambda_N\}$, where $\omega(k)=\alpha$ if site $k$ is occupied by a particle of type $\alpha\in\{A,B,C\}$. We make the convention that $\alpha+1,\alpha+2,\ldots$ denote the particle types that are successors to $\alpha$ in the cyclic-order $ABC$.

For $i,j\in\Lambda_N$ and $\omega\in\widetilde{\Omega}^N$ we denote by $\sigma^{i,j}\omega$ the configuration obtained from~$\omega$ by exchanging the particles at the sites $i$ and $j$:
\begin{equation*}
(\sigma^{i,j}\omega)(k) =
\begin{cases}
\omega(k) & \text{if $k\notin\{i,j\}$}, \\
\omega(j) & \text{if $k = i$}, \\
\omega(i) & \text{if $k = j$}. \\
\end{cases}
\end{equation*}

We consider the continuous-time Markov chain $\{\eta^\beta (t) : t\geq0\}$ on the state space~$\widetilde{\Omega}^N$ whose generator $L_\beta$ acts on functions $f: \widetilde{\Omega}^N \to \bb R$ as
\begin{equation*}
(L_\beta f)(\omega)= \sum_{k\in\Lambda_N} c^\beta_k(\omega)[f(\sigma^{k,k+1}\omega) - f(\omega)]
\end{equation*}
where, for $\beta\geq0$, the jump rates $c^\beta_k$ are given by
\begin{equation*}
c^\beta_k(\omega) =
\begin{cases}
e^{-\beta} & \text{if $(\omega(k),\omega(k+1))\in\{(A,B),(B,C),(C,A)\}$}, \\
1 & \text{otherwise.} \\
\end{cases}
\end{equation*}
Almost always we omit the index $\beta$ and denote $\eta^\beta(t)$ just by $\eta(t)$.

As the system evolves by nearest neighbor transpositions, the number of particles of each species is conserved. Therefore, given three integers $N_\alpha$, $\alpha\in\{A,B,C\}$, such that $N_A+N_B+N_C=2N+1$, we have a well defined process on the component $\Omega^{N_A,N_B,N_C}=\{\omega\in\widetilde{\Omega}^N : \sum_{k\in\Lambda_N}\mb 1 \{\omega(k)=\alpha\}=N_\alpha, \, \alpha\in\{A,B,C\}\}$, which is clearly irreducible and then admits a unique invariant measure. To shorten notation, let us suppose that we have fixed $N_A$, $N_B$, and $N_C$ as functions of $N$ and then we write simply $\Omega^N$ instead of $\Omega^{N_A,N_B,N_C}$.

The invariant measure $\mu_\beta=\mu_{\beta,N}$ is in general not explicit known. However, in the special case of equal densities $N_A=N_B=N_C$, as shown in \cite{ekkm1,ekkm2}, the process is reversible with respect to the Gibbs measure $\mu_\beta$, given by
\begin{equation*}
\mu_\beta(\omega)=\frac{1}{Z_\beta}e^{-\beta \bb H (\omega)},
\end{equation*}
where $Z_\beta$ is the normalizing partition function and $\bb H$ is a non-local Hamiltonian, which may be written as
\begin{equation}\label{1}
\bb H (\omega)=\frac{1}{2N+1}\sum_{k\in\Lambda_N}\sum_{i=1}^{2N}i \mb 1\big\{(\omega(k),\omega(k+i))\in\{(A,B),(B,C),(C,A)\}\big\}.
\end{equation}

A simple computation relying on the equal densities constraint shows that nearest neighbor transpositions of type $(\alpha,\alpha+1)\to(\alpha+1,\alpha)$ increase the energy $\bb H$ by~$1$~unit, while the opposite kind of transposition decreases $\bb H$ by $1$ unit. The reversibility of the process in this special case follows from this observation.

The configurations in which the particles form three pure regions, one of each species, located clockwise in the cyclic-order $ABC$ deserve a special notation. Define $\omega^N_0\in\Omega^N$ as
\begin{equation*}
\omega^N_0(j) =
\begin{cases}
A & \text{if $0\leq j\leq N_A-1$}, \\
B & \text{if $N_A\leq j \leq N_A+N_B-1$}, \\
C & \text{otherwise},
\end{cases}
\end{equation*}
and, for each $k\in\Lambda_N$, define $\omega^N_k=\Theta^k\omega^N_0$, where $\Theta^k:\Omega^N\to\Omega^N$ stands for the shift operator $(\Theta^k\omega)(i)=\omega(i-k)$. By convention we omit the index $N$ in the notation, and we write simply $\omega_k$ instead of $\omega^N_k$. Denote by $\Omega^N_0$ the set of these configurations:
\begin{equation*}
    \Omega^N_0=\{\omega_k : k\in\Lambda_N\}.
\end{equation*}
We remark that, in the equal densities case, $\Omega^N_0$ corresponds to the set of ground states of the energy $\bb H$.

For each $\omega\in\Omega^N$ denote by $\mb P^\beta_\omega$ the probability measure induced by the Markov process $\{\eta(t) : t\geq 0\}$ starting from $\omega$ on the Skorohod space $D([0,\infty), \Omega^N)$ of c\`adl\`ag paths. Expectation with respect to $\mb P^\beta_\omega$ is represented by $\mb E^\beta_\omega$.

\subsection{Main results}\label{mainresults} We analyze in this article the asymptotic evolution, as $\beta\uparrow\infty$, of the Markov process $\{\eta(t):t\geq0\}$, where also the number of particles may depend on~$\beta$. For simplicity, we omit this dependence in the notation.

From now on, we use the notation
\begin{equation*}
M=\min\{N_A,N_B,N_C\}.
\end{equation*}
If the process starts from some $\omega_k\in\Omega^N_0$, at least $M$ jumps of rate $e^{-\beta}$ are needed in order to visit another configuration in $\Omega^N_0$. This suggests that, for fixed~$N$, the interesting time scale to consider is $e^{M\beta}$. In Section \ref{poucofora} we show that, in the time scale $N^2e^{M\beta}$, if $N$ does not increase too fast with $\beta$, the process spends a negligible time outside $\Omega^N_0$:
\begin{lemma}\label{lemma3} Let $M^*=\max\{N_A,N_B,N_C\}$. Assume that
\begin{equation}\label{4}
\lim_{\beta\to\infty}N^34^{M^*}e^{-\beta}=0
\end{equation}
Then, for every $k\in\Lambda_N$, $t\geq0$
\begin{equation}\label{5}
    \lim_{\beta\to\infty}\Exp{\omega_k}{\beta}{\int_0^t\mb1 \{\eta(sN^2e^{M\beta})\notin\Omega^N_0\}ds}=0.
\end{equation}
\end{lemma}

Our main result, Theorem \ref{mainresult}, characterizes the motion of this segregated shape in the time-scale $N^2e^{M\beta}$ when the system size grows with $\beta$. We express the result in terms of the convergence of the evolution of a macroscopic variable associated to the configurations: 
the center of mass of the particles of type $A$. In order to define this center of mass we need to introduce some other notations. For any configurations $\xi,\zeta \in \Omega^N$, by a path from $\xi$ to $\zeta$ we mean a sequence of configurations $\gamma=(\xi=\xi_0,\xi_1,\ldots,\xi_n=\zeta)$  such that $\xi_k$ can be obtained from $\xi_{k-1}$ by a simple nearest neighbor transposition. We define ${\rm dist}(\xi,\zeta)$ as the smallest $n$ such that there exists a path from $\xi$ to $\zeta$ of length $n$. For any $n$, and $k\in\Lambda_N$, define
\begin{equation}\label{Delatank}
    \Delta^n_k=\{\omega\in\Omega^N: {\rm dist}(\omega,\omega_k)=n\}.
\end{equation}
Due to the periodic boundary conditions, for many configurations the centers of mass of the particles of type $\alpha$, $\alpha\in\{A,B,C\}$, are not well defined. However, the proof of Lemma \ref{lemma3} in fact shows that, under \eqref{4}, for any $t\geq0$
\begin{equation}\label{naosaideXiintro}
    \lim_{\beta\to\infty}\Prob{\omega_0}{\beta}{\eta(s)\notin \Xi^N \text{ for some } 0\leq s \leq t N^2e^{M\beta}}=0,
\end{equation}
where $\Xi^N$ is a subset of configurations such that $\Xi^N\subseteq\Gamma^N:=\bigcup_{k\in\Lambda_N}\Gamma^N_k$, where
\begin{equation*}
\Gamma^N_k=\left(\bigcup_{n=0}^M\Delta^n_k\right)\cup \bigcup_{i_1,i_2,i_3,i_4 \in\Lambda_N}\{\sigma^{i_1,i_2}\sigma^{i_3,i_4}\omega_k\}.
\end{equation*}
Note that the configurations in $\Gamma^N_k$ that are not at distance $M$ or less from $\omega_k$ differ from $\omega_k$ by at most two transpositions, not necessarily nearest-neighbor. In $\Gamma^N$, the center of mass can be defined unambiguously. Suppose, for example, that $N_A=M$. For $\omega\in\Gamma^N_0$, define the center of mass (of the particles of type $A$) of the configuration $\omega$, denoted by $\mc C(\omega)$, as
\begin{equation*}
    \mc C(\omega)=\frac{1}{N}\sum_{k\in\Lambda_N}\frac{k \mb 1\{\omega(k)=A\}}{N_A}.
\end{equation*}
Then, for a configuration $\xi\in\Gamma^N_k$, $k\in\Lambda_N$, take $\omega\in\Gamma^N_0$ such that $\xi=\Theta^k\omega$, and define $\mc C (\xi)=\mc C(\omega)+k/N \mod [-1,1]$. Just for completeness, for $\xi\notin\Gamma^N$ define $\mc C(\xi)=0$. Actually, by \eqref{naosaideXiintro} this latter definition is not relevant.

Let
\begin{equation}\label{d}
d:=|\{\alpha\in\{A,B,C\}: N_\alpha=M\}|.
\end{equation}
Note that $d$ may vary with $\beta$, but we omit this dependence to simplify the notation. Define
\begin{equation}\label{timescale}
\theta_\beta:=\frac{1}{2d}e^{M\beta}N^2.
\end{equation}
Our main result is the following theorem.

\begin{theorem}\label{mainresult} Assume that $\eta(0)=\omega_0$ and that $N\uparrow\infty$ as $\beta\uparrow\infty$ in such a way that
$N_A<N_B\leq N_C$, with $N_A\uparrow\infty$, $N_A/N\to r_A\geq0$,
\begin{equation}\label{conddrift}
\lim_{\beta\to\infty}\left(\frac{1}{3}\right)^{N_B}N_C=b\in[0,\infty)
\end{equation}
and
\begin{equation}\label{condNebeta}
\lim_{\beta\to\infty}(N^54^{N_A}+N^34^{N_C})e^{-\beta}=0.
\end{equation}
Then, as $\beta\uparrow\infty$, the process $\{\mc C(\eta(t\theta_\beta)):t\geq0\}$ converges in the uniform topology to the diffusion 
\begin{equation}\label{browwithdrift}
\{(1/2)r_A-(3/2)bt+B_t:t\geq 0\}
\end{equation}
 on the circle $[-1,1]$, where $\{B_t: t\geq0\}$ is a Brownian motion with infinitesimal variance equal to $1$. If $b=0$ in \eqref{conddrift}, we may replace the assumption $N_A<N_B$ by $N_A\leq N_B$.
\end{theorem}

\begin{remark} In Theorem \ref{mainresult}, if $N_B$ increases not so slowly, in the sense that $b=0$ in~\eqref{conddrift}, then the limit is a Brownian motion without drift. This is the case when we have positive proportion of each type of particle:
\begin{equation}\label{3}
    \lim_{\beta\to\infty}\frac{N_\alpha}{N}>0, ~~\alpha\in\{A,B,C\}.
\end{equation}
On the other hand, if $N_B$ increases even more slowly, in the sense that $b=\infty$ in \eqref{conddrift}, then we have to look at the process in another time scale. Suppose, for example, that we can find some $u\in (1,2)$ such that
\begin{equation}\label{otherscale}
    \lim_{\beta\to\infty}\left(\frac{1}{3}\right)^{N_B}N_C^{u-1}=c\in(0,\infty).
\end{equation}
In this case, replacing \eqref{conddrift} by \eqref{otherscale}, we can prove that, as $\beta\uparrow\infty$, the process $\{\mc C(\eta(tN^ue^{N_A\beta})):t\geq0\}$ converges to a deterministic linear function~$\{\mu t: t\geq0\}$ on the circle $[-1,1]$, where $\mu$ is a constant which depends on $c$. This will become clear with the proofs given in Section~\ref{convBro}.
\end{remark}

\begin{remark}\label{remarkpoucotempofora} In Lemma \ref{lemma3}, the restriction $\eqref{4}$ is not optimal. It comes from our crude estimation of $\mu_\beta(\Xi^N)$ in the general densities case, where the invariant measure is not explicitly known. By repeating the steps that lead to~\eqref{56}, we see that assertion \eqref{5} also hods if
\begin{equation}\label{medconcentra}
\lim_{\beta\to\infty}\mu_\beta(\Omega^N\setminus\Omega^N_0)=0.
\end{equation}
In the special case of equal densities, where the invariant measure is explicitly known, we can verify \eqref{medconcentra} without any assumption that controls the growth of~$N$. Details are given in Section~\ref{poucofora}.
\end{remark}

\subsection{Results for the trace process in $\Omega^N_0$}

Now we present some results that are preliminary steps in our method to prove Theorem \ref{mainresult}, but which are interesting by themselves. 

Denote by $\{\eta_0(t): t\geq0\}$ the trace of the process $\{\eta(t):t\geq0\}$ on the set $\Omega^N_0$, i.e, the Markov chain obtained from $\{\eta(t):t\geq0\}$ by neglecting the time spent outside $\Omega^N_0$. More precisely, we define $\{\eta_0(t): t\geq0\}$ by
\begin{equation}\label{deftrace}
\mc T_t=\int_0^t\mb 1\left\{\eta(s)\in\Omega^N_0\right\}ds; \quad \mc S_t=\sup\{s\geq0: \mc T_s\leq t\}; \quad \eta_0(t)=\eta\left(\mc S_t\right).
\end{equation}
We refer to \cite{bl2} for important elementary properties of the trace process. 

To prove Theorem \ref{mainresult} we first analyze the trace process $\{\eta_0(t):t\geq0\}$. For finite volume we obtain the following result which reveals an interesting non-local asymptotic behavior.

\begin{theorem}\label{theo1} For $N_A$, $N_B$ and $N_C$ constant greater than or equal to $3$, as $\beta \uparrow \infty$ the speeded up process $\{\eta_0(e^{M\beta}t): t\geq0\}$ converges to a Markov process on $\Omega^N_0$ which jumps from $\omega_i$ to $\omega_j$ at a strictly positive rate $r(i,j)$.
\end{theorem}

The proof of this theorem, as well as the expression for $r(i,j)$, is given in Section~\ref{sectracoOmega0}. This theorem is complemented with Lemma~\ref{lemma3}, which says that we are not losing much just looking at the trace process in~$\Omega^N_0$. The rates $r(i,j)$, which depend also on $N_A$, $N_B$ and $N_C$, can be expressed in terms of some absorption probabilities of a simple Markov dynamics, the one described by Figure~\ref{fig6} in the next section.

\begin{remark} It is possible to reformulate Theorem \ref{theo1}, without referring to the trace process, by asserting the convergence of the (time re-scaled) original process $\{\eta(e^{M\beta}t):t\geq0\}$ in a topology introduced in \cite{l}, weaker than the Skorohod one.
\end{remark}

Denote by $R^\beta_0(\omega_i,\omega_j)$
the jump rates of the trace process $\{\eta_0(t): t\geq0\}$. By translation invariance, it is clear
that 
\begin{equation}\label{rbeta}
R^\beta_0(\omega_i,\omega_j)=R^\beta_0(\omega_0,\omega_{j-i})=: r_\beta(j-i).
\end{equation}
Let $\mb X(\omega_k) = k$, $k\in \Lambda_N$, so that if $X
(t) = \mb X(\eta_0(t))$ then $\{X(t): t\geq0\}$ is a random walk on $\Lambda_N$ which jumps from
$i$ to $j$ at rate $r_\beta(j-i)$.

The next result refers to the case where $N$ goes to infinity as a function of~$\beta$. We start considering the degenerated case where $N_A$ and $N_B$ are constants and only $N_C$ goes to infinity. We have a ballistic behavior in this situation.

\begin{theorem}\label{ballistic} Assume that $3\leq N_A<N_B$ are constants and that $N_C\uparrow\infty$ as $\beta\uparrow\infty$ in such a way that
\begin{equation}\label{ballcond}
    \lim_{\beta\to\infty}N_C^5\beta e^{-\beta}=0.
\end{equation}
If $\eta(0)=\omega_0$, then, as $\beta\uparrow\infty$, the process $\{X(tNe^{N_A\beta})/N:t\geq0\}$ converges in the uniform topology to a linear function $\{v(N_A,N_B)t: t\geq0\}$ on the circle $[-1,1]$.
\end{theorem}

The condition $N_A<N_B$ is crucial for the ballistic behavior. If $N_A$ and $N_B$ are constant but $N_A=N_B$, then the process $\{X(t):t\geq0\}$ is symmetric. In this case, scaling time by $N^2e^{M\beta}$ we can prove the convergence to a Brownian motion if $N_C^6\beta e^{-\beta}\downarrow0$. In the case $N_A\neq N_B$, the velocity $v(N_A,N_B)$, which is an antisymmetric function of $N_A$ and $N_B$, is negative when $N_A<N_B$. It can be expressed in terms of some absorption probabilities for a random walk in a simple graph, which can be explicitly computed in terms of $N_A$ and $N_B$. The analysis of the asymptotic dependence of  $v(N_A,N_B)$ on $N_A$ and $N_B$, which is presented in Lemma~\ref{lemmaprodrift} helps us to find the specific scenario, \eqref{conddrift}, for the convergence to a Brownian motion with drift.

The proof of Theorem \ref{ballistic} is given in Section~\ref{convBro}, where we also state and prove the version of Theorem~\ref{mainresult} referring to the trace process (Theorem \ref{browdrift}).

\section{Sketch of the proofs}

Our main result, Theorem \ref{mainresult}, is a consequence of the corresponding Theorem~\ref{browdrift} and the fact, to be proved in Section \ref{convcenterofmass}, that (assuming $\eta(0)=\omega_0$) the process $\{\mc C(\eta(t\theta_\beta)):t\geq0\}$ is close to the trace process $\{X(t\theta_\beta)/N+r_A/2:t\geq0\}$ in the Skohorod space $D([0,\infty),[-1,1])$.

The main idea to analyze the trace of the process $\{\eta(t):t\geq0\}$ on $\Omega^N_0$ is to consider first the trace on a larger set $\Omega^N_1$. Now we will see why such a set $\Omega^N_1$ comes naturally.

To fix ideas, suppose that $3\leq N_A\leq N_B\leq N_C$ are constants. Suppose that the process starts from the configuration $\omega_k\in\Omega^N_0$. Note that, in order to visit any other configuration in $\Omega^N_0$, at least $N_A$ jumps of rate $e^{-\beta}$ are needed. The most simple trajectory that we can imagine between $\omega_k$ and another configuration in $\Omega^N_0$ occurs when the whole block of particles of type $A$ is crossed, for example, by a particle of type $C$ walking clockwise (and for this, $N_A$ jumps of rate $e^{-\beta}$ are needed) and then this detached particle of type $C$ can continue moving in clockwise direction inside the domain of particles of type $B$ (now performing rate $1$ jumps) until, after crossing all the particles of type $B$, it meets the other particles of type $C$. This way, we arrive at the configuration $\omega_{k-1}$. In an analogous way, we find a path from $\omega_k$ to $\omega_{k+1}$. This reveals that the correct time scale to analyze the trace process on~$\Omega^N_0$ is~$e^{N_A\beta}$.

At first glance, it appears that these trajectories we described are the only ones possible in the time scale $e^{N_A\beta}$ and that the asymptotic dynamics will be restricted to jumps from $\omega_k$ to $\omega_{k+1}$ or $\omega_{k-1}$. So, Theorem~\ref{theo1} is somewhat surprising. The truth is that, for any $j$, there exists a trajectory from $\omega_k$, which is possible in the time scale $e^{N_A\beta}$, such that the next visited configuration in $\Omega^N_0$ is $\omega_j$. The explanation is the following. Starting from $\omega_k$, in a time of order $e^{N_A\beta}$ it is possible that we have a meeting of a particle of type $C$ and a particle of type $B$ inside the domain of particles of type $A$. Once these two particles meet, they can interchange their positions performing a rate $1$ jump. This way, we fall in a metastable configuration from which all possible jumps have rate $e^{-\beta}$. This configuration is very similar to $\omega_k$ except for the pair $BC$ inside the block of $A$s. For this configuration, transposition of nearest neighbor particles that are far from this pair $BC$, which may occur at the frontiers between two different domains, are reverted with high probability in the next jump of the chain. So, let us focus in what can happen with this pair $BC$.

After a time of order $e^{\beta}$, this pair can disappear if $BC$ turns to $CB$ and then, with rate $1$ jumps, the particles $C$ and $B$ return to their original positions in the configuration $\omega_k$. But also, in a time of order $e^{\beta}$, the pair $BC$ can move inside the domain of particles of type $A$. For example, with a rate $e^{-\beta}$ jump, the particle $C$ can move to the right, in such a way that $ABCAA$ becomes $ABACA$.  Now, with a rate $1$ jump, the particle $B$ moves to the right and we obtain $AABCA$. Clearly, the pair $BC$ can also move to the left. This way, we can move the pair $BC$ until, for example, near to the right end of the block of particles of type $A$, and we arrive at a configuration that is almost $\omega_k$ except for the appearance of a block $AABCABB$ in the frontier of the regions of particles $A$ and $B$. From this configuration, the pair $CA$ can turn to $AC$ and after $N_B$ jumps of rate $1$ we arrive at $\omega_{k-1}$. But also, by the same reason as before, instead of becomes $AC$, in times of order $e^{\beta}$ the pair $CA$ can move inside the domain of particles of type $B$ until eventually the process may arrive at a configuration that is almost $\omega_{k-1}$ except for a block $BBCABCC$ in the frontier of regions of particles of types $B$ and $C$. At this point, after a time of order $e^{\beta}$, the pair $AB$ can turn to $BA$ and then, after $N_C$ jumps of rate $1$, the process can arrive at $\omega_{k-2}$. This shows how, in time scale $e^{M\beta}$, it is possible to find a path starting from $\omega_k$ such that the next visited configuration in $\Omega^N_0$ is $\omega_{k-2}$. Clearly, we could continue moving a pair of particles in order to arrive at any configuration in $\Omega^N_0$.

The class of the metastable configurations which appears in the above described paths will be called $\Omega^N_1$, and it will be precisely defined in Section~\ref{secOmegaN1}. As the above discussion indicates, starting from a configuration in $\Omega^N_0$, after a time of order $e^{N_A\beta}$, we can visit a configuration $\omega\in\Omega^N_1$ where we will stay for a time of order~$e^\beta$. Starting from $\omega\in\Omega^N_1$, in time scale $e^{\beta}$, essentially, what we see is a Markov chain $\{\widehat\eta_1(e^\beta t): t\geq0\}$ in $\Omega^N_1$ for which the configurations in $\Omega^N_0$ are absorbing states. The structure of this dynamics reveals to be pretty simple, as shown in Figure~\ref{fig6}.

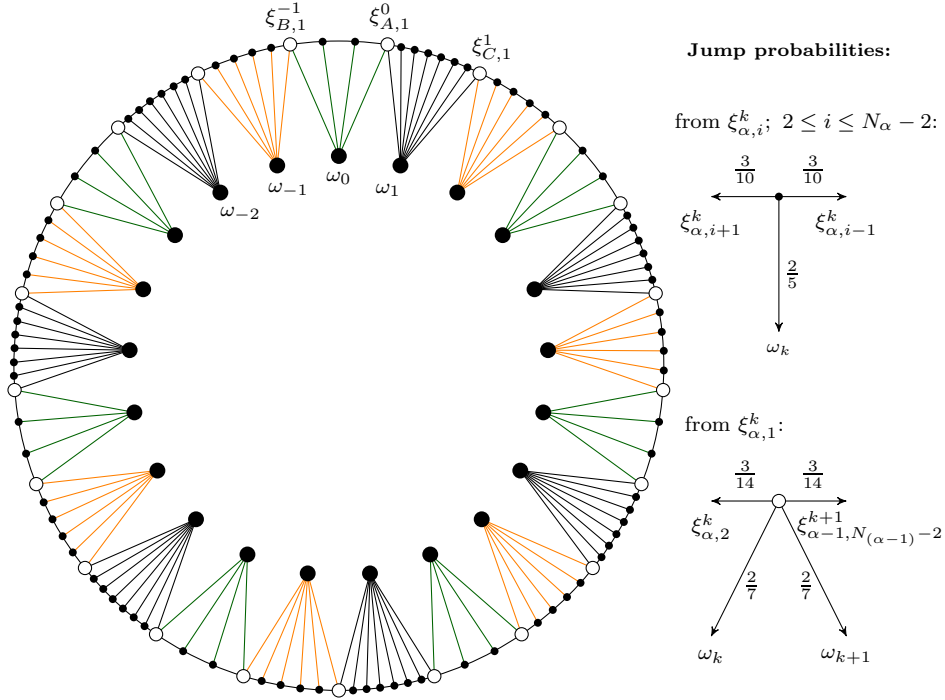
\begin{figure}[h!]
\centering
\definecolor{zzttqq}{rgb}{0.6,0.2,0}
\definecolor{ffxfqq}{rgb}{1,0.5,0}
\definecolor{qqwuqq}{rgb}{0,0.39,0}
\definecolor{ffffff}{rgb}{1,1,1}
\begin{tikzpicture}[line cap=round,line join=round,>=stealth',x=0.9cm,y=0.9cm]
\clip(-5.07,-5.17) rectangle (8.97,5.52);
\draw(0,0) circle (4.32cm);
\draw [color=qqwuqq] (-2.42,1.93)-- (-4.16,2.4);
\draw [color=qqwuqq] (-3.02,-0.69)-- (-4.47,-1.75);
\draw [color=qqwuqq] (-1.35,-2.79)-- (-1.41,-4.59);
\draw [color=qqwuqq] (1.35,-2.79)-- (2.7,-3.97);
\draw [color=qqwuqq] (3.02,-0.69)-- (4.79,-0.36);
\draw [color=qqwuqq] (2.42,1.93)-- (3.26,3.52);
\draw [color=qqwuqq] (0,3.1)-- (-0.72,4.75);
\draw [color=qqwuqq] (-2.42,1.93)-- (-3.26,3.52);
\draw [color=qqwuqq] (-3.02,-0.69)-- (-4.79,-0.36);
\draw [color=qqwuqq] (-1.35,-2.79)-- (-2.7,-3.97);
\draw [color=qqwuqq] (1.35,-2.79)-- (1.41,-4.59);
\draw [color=qqwuqq] (3.02,-0.69)-- (4.47,-1.75);
\draw [color=qqwuqq] (2.42,1.93)-- (4.16,2.4);
\draw [color=qqwuqq] (0,3.1)-- (0.72,4.75);
\draw [color=ffxfqq] (-0.91,2.96)-- (-0.72,4.75);
\draw [color=ffxfqq] (-0.91,2.96)-- (-2.08,4.32);
\draw [color=ffxfqq] (-2.89,1.13)-- (-4.68,1.07);
\draw [color=ffxfqq] (-2.68,-1.55)-- (-3.75,-2.99);
\draw [color=ffxfqq] (-0.46,-3.07)-- (0,-4.8);
\draw [color=ffxfqq] (2.11,-2.27)-- (3.75,-2.99);
\draw [color=ffxfqq] (3.09,0.23)-- (4.68,1.07);
\draw [color=ffxfqq] (1.75,2.56)-- (2.08,4.32);
\draw [color=ffxfqq] (-2.89,1.13)-- (-4.16,2.4);
\draw [color=ffxfqq] (-2.68,-1.55)-- (-4.47,-1.75);
\draw [color=ffxfqq] (-0.46,-3.07)-- (-1.41,-4.59);
\draw [color=ffxfqq] (2.11,-2.27)-- (2.7,-3.97);
\draw [color=ffxfqq] (3.09,0.23)-- (4.79,-0.36);
\draw [color=ffxfqq] (1.75,2.56)-- (3.26,3.52);
\draw (-1.75,2.56)-- (-3.26,3.52);
\draw (-3.09,0.23)-- (-4.79,-0.36);
\draw (-2.11,-2.27)-- (-2.7,-3.97);
\draw (0.46,-3.07)-- (1.41,-4.59);
\draw (2.68,-1.55)-- (4.47,-1.75);
\draw (2.89,1.13)-- (4.16,2.4);
\draw (0.91,2.96)-- (0.72,4.75);
\draw (-1.75,2.56)-- (-2.08,4.32);
\draw (-3.09,0.23)-- (-4.68,1.07);
\draw (-2.11,-2.27)-- (-3.75,-2.99);
\draw (0.46,-3.07)-- (0,-4.8);
\draw (2.68,-1.55)-- (3.75,-2.99);
\draw (2.89,1.13)-- (4.68,1.07);
\draw (0.91,2.96)-- (2.08,4.32);
\draw [->] (6.5,2.5) -- (7.5,2.5);
\draw [->] (6.5,2.5) -- (5.5,2.5);
\draw [->] (6.5,2.5) -- (6.5,0.5);
\draw [->] (6.5,-2) -- (7.5,-2);
\draw [->] (6.5,-2) -- (5.5,-2);
\draw [->] (6.5,-2) -- (5.5,-4);
\draw [->] (6.5,-2) -- (7.5,-4);
\draw [color=qqwuqq] (-3.9,2.8)-- (-2.42,1.93);
\draw [color=qqwuqq] (-3.6,3.18)-- (-2.42,1.93);
\draw [color=qqwuqq] (-4.62,-1.3)-- (-3.02,-0.69);
\draw [color=qqwuqq] (-4.73,-0.83)-- (-3.02,-0.69);
\draw [color=qqwuqq] (-1.86,-4.42)-- (-1.35,-2.79);
\draw [color=qqwuqq] (-2.3,-4.22)-- (-1.35,-2.79);
\draw [color=qqwuqq] (2.3,-4.22)-- (1.35,-2.79);
\draw [color=qqwuqq] (1.86,-4.42)-- (1.35,-2.79);
\draw [color=qqwuqq] (4.73,-0.83)-- (3.02,-0.69);
\draw [color=qqwuqq] (4.62,-1.3)-- (3.02,-0.69);
\draw [color=qqwuqq] (3.6,3.18)-- (2.42,1.93);
\draw [color=qqwuqq] (3.9,2.8)-- (2.42,1.93);
\draw [color=qqwuqq] (-0.24,4.79)-- (0,3.1);
\draw [color=qqwuqq] (0.24,4.79)-- (0,3.1);
\draw [color=ffxfqq] (-1.82,4.44)-- (-0.91,2.96);
\draw [color=ffxfqq] (-1.55,4.54)-- (-0.91,2.96);
\draw [color=ffxfqq] (-1.28,4.63)-- (-0.91,2.96);
\draw [color=ffxfqq] (-1,4.7)-- (-0.91,2.96);
\draw [color=ffxfqq] (-4.61,1.35)-- (-2.89,1.13);
\draw [color=ffxfqq] (-4.52,1.62)-- (-2.89,1.13);
\draw [color=ffxfqq] (-4.41,1.89)-- (-2.89,1.13);
\draw [color=ffxfqq] (-4.29,2.15)-- (-2.89,1.13);
\draw [color=ffxfqq] (-3.93,-2.76)-- (-2.68,-1.55);
\draw [color=ffxfqq] (-4.08,-2.52)-- (-2.68,-1.55);
\draw [color=ffxfqq] (-4.23,-2.27)-- (-2.68,-1.55);
\draw [color=ffxfqq] (-4.36,-2.02)-- (-2.68,-1.55);
\draw [color=ffxfqq] (-0.29,-4.79)-- (-0.46,-3.07);
\draw [color=ffxfqq] (-0.57,-4.77)-- (-0.46,-3.07);
\draw [color=ffxfqq] (-0.86,-4.72)-- (-0.46,-3.07);
\draw [color=ffxfqq] (-1.14,-4.66)-- (-0.46,-3.07);
\draw [color=ffxfqq] (3.57,-3.21)-- (2.11,-2.27);
\draw [color=ffxfqq] (3.37,-3.42)-- (2.11,-2.27);
\draw [color=ffxfqq] (3.16,-3.61)-- (2.11,-2.27);
\draw [color=ffxfqq] (2.94,-3.8)-- (2.11,-2.27);
\draw [color=ffxfqq] (4.74,0.79)-- (3.09,0.23);
\draw [color=ffxfqq] (4.77,0.5)-- (3.09,0.23);
\draw [color=ffxfqq] (4.8,0.22)-- (3.09,0.23);
\draw [color=ffxfqq] (4.8,-0.07)-- (3.09,0.23);
\draw [color=ffxfqq] (2.34,4.19)-- (1.75,2.56);
\draw [color=ffxfqq] (2.58,4.05)-- (1.75,2.56);
\draw [color=ffxfqq] (2.82,3.88)-- (1.75,2.56);
\draw [color=ffxfqq] (3.05,3.71)-- (1.75,2.56);
\draw (-3.11,3.65)-- (-1.75,2.56);
\draw (-2.95,3.78)-- (-1.75,2.56);
\draw (-2.79,3.91)-- (-1.75,2.56);
\draw (-2.62,4.02)-- (-1.75,2.56);
\draw (-2.44,4.13)-- (-1.75,2.56);
\draw (-2.27,4.23)-- (-1.75,2.56);
\draw (-4.8,-0.15)-- (-3.09,0.23);
\draw (-4.8,0.05)-- (-3.09,0.23);
\draw (-4.79,0.26)-- (-3.09,0.23);
\draw (-4.78,0.46)-- (-3.09,0.23);
\draw (-4.75,0.66)-- (-3.09,0.23);
\draw (-4.72,0.87)-- (-3.09,0.23);
\draw (-2.87,-3.85)-- (-2.11,-2.27);
\draw (-3.03,-3.72)-- (-2.11,-2.27);
\draw (-3.19,-3.59)-- (-2.11,-2.27);
\draw (-3.34,-3.45)-- (-2.11,-2.27);
\draw (-3.48,-3.3)-- (-2.11,-2.27);
\draw (-3.62,-3.15)-- (-2.11,-2.27);
\draw (1.22,-4.64)-- (0.46,-3.07);
\draw (1.02,-4.69)-- (0.46,-3.07);
\draw (0.82,-4.73)-- (0.46,-3.07);
\draw (0.61,-4.76)-- (0.46,-3.07);
\draw (0.41,-4.78)-- (0.46,-3.07);
\draw (0.21,-4.8)-- (0.46,-3.07);
\draw (4.39,-1.94)-- (2.68,-1.55);
\draw (4.3,-2.13)-- (2.68,-1.55);
\draw (4.21,-2.31)-- (2.68,-1.55);
\draw (4.1,-2.49)-- (2.68,-1.55);
\draw (3.99,-2.66)-- (2.68,-1.55);
\draw (3.88,-2.83)-- (2.68,-1.55);
\draw (4.26,2.22)-- (2.89,1.13);
\draw (4.35,2.04)-- (2.89,1.13);
\draw (4.43,1.85)-- (2.89,1.13);
\draw (4.5,1.66)-- (2.89,1.13);
\draw (4.57,1.46)-- (2.89,1.13);
\draw (4.63,1.27)-- (2.89,1.13);
\draw (0.92,4.71)-- (0.91,2.96);
\draw (1.12,4.67)-- (0.91,2.96);
\draw (1.32,4.62)-- (0.91,2.96);
\draw (1.51,4.56)-- (0.91,2.96);
\draw (1.71,4.49)-- (0.91,2.96);
\draw (1.9,4.41)-- (0.91,2.96);
\begin{scriptsize}
\fill [color=black] (-0.91,2.96) circle (3.0pt);
\fill [color=black] (-1.75,2.56) circle (3.0pt);
\fill [color=black] (-2.42,1.93) circle (3.0pt);
\fill [color=black] (-2.89,1.13) circle (3.0pt);
\fill [color=black] (-3.09,0.23) circle (3.0pt);
\fill [color=black] (-3.02,-0.69) circle (3.0pt);
\fill [color=black] (-2.68,-1.55) circle (3.0pt);
\fill [color=black] (-2.11,-2.27) circle (3.0pt);
\fill [color=black] (-1.35,-2.79) circle (3.0pt);
\fill [color=black] (-0.46,-3.07) circle (3.0pt);
\fill [color=black] (0.46,-3.07) circle (3.0pt);
\fill [color=black] (1.35,-2.79) circle (3.0pt);
\fill [color=black] (2.11,-2.27) circle (3.0pt);
\fill [color=black] (2.68,-1.55) circle (3.0pt);
\fill [color=black] (3.02,-0.69) circle (3.0pt);
\fill [color=black] (3.09,0.23) circle (3.0pt);
\fill [color=black] (2.89,1.13) circle (3.0pt);
\fill [color=black] (2.42,1.93) circle (3.0pt);
\fill [color=black] (1.75,2.56) circle (3.0pt);
\fill [color=black] (0.91,2.96) circle (3.0pt);
\fill [color=black] (0,3.1) circle (3.0pt);
\fill [color=ffffff] (-2.08,4.32) circle (2.5pt);
\fill [color=ffffff] (-3.26,3.52) circle (2.5pt);
\fill [color=ffffff] (-4.16,2.4) circle (2.5pt);
\fill [color=ffffff] (-4.68,1.07) circle (2.5pt);
\fill [color=ffffff] (-4.79,-0.36) circle (2.5pt);
\fill [color=ffffff] (-4.47,-1.75) circle (2.5pt);
\fill [color=ffffff] (-3.75,-2.99) circle (2.5pt);
\fill [color=ffffff] (-2.7,-3.97) circle (2.5pt);
\fill [color=ffffff] (-1.41,-4.59) circle (2.5pt);
\fill [color=ffffff] (0,-4.8) circle (2.5pt);
\fill [color=ffffff] (1.41,-4.59) circle (2.5pt);
\fill [color=ffffff] (2.7,-3.97) circle (2.5pt);
\fill [color=ffffff] (3.75,-2.99) circle (2.5pt);
\fill [color=ffffff] (4.47,-1.75) circle (2.5pt);
\fill [color=ffffff] (4.79,-0.36) circle (2.5pt);
\fill [color=ffffff] (4.68,1.07) circle (2.5pt);
\fill [color=ffffff] (4.16,2.4) circle (2.5pt);
\fill [color=ffffff] (3.26,3.52) circle (2.5pt);
\fill [color=ffffff] (2.08,4.32) circle (2.5pt);
\fill [color=ffffff] (0.72,4.75) circle (2.5pt);
\fill [color=ffffff] (-0.72,4.75) circle (2.5pt);
\fill [color=black] (-3.9,2.8) circle (1.5pt);
\fill [color=black] (-3.6,3.18) circle (1.5pt);
\fill [color=black] (-4.62,-1.3) circle (1.5pt);
\fill [color=black] (-4.73,-0.83) circle (1.5pt);
\fill [color=black] (-1.86,-4.42) circle (1.5pt);
\fill [color=black] (-2.3,-4.22) circle (1.5pt);
\fill [color=black] (2.3,-4.22) circle (1.5pt);
\fill [color=black] (1.86,-4.42) circle (1.5pt);
\fill [color=black] (4.73,-0.83) circle (1.5pt);
\fill [color=black] (4.62,-1.3) circle (1.5pt);
\fill [color=black] (3.6,3.18) circle (1.5pt);
\fill [color=black] (3.9,2.8) circle (1.5pt);
\fill [color=black] (-0.24,4.79) circle (1.5pt);
\fill [color=black] (0.24,4.79) circle (1.5pt);
\fill [color=black] (-1.82,4.44) circle (1.5pt);
\fill [color=black] (-1.55,4.54) circle (1.5pt);
\fill [color=black] (-1.28,4.63) circle (1.5pt);
\fill [color=black] (-1,4.7) circle (1.5pt);
\fill [color=black] (-4.61,1.35) circle (1.5pt);
\fill [color=black] (-4.52,1.62) circle (1.5pt);
\fill [color=black] (-4.41,1.89) circle (1.5pt);
\fill [color=black] (-4.29,2.15) circle (1.5pt);
\fill [color=black] (-3.93,-2.76) circle (1.5pt);
\fill [color=black] (-4.08,-2.52) circle (1.5pt);
\fill [color=black] (-4.23,-2.27) circle (1.5pt);
\fill [color=black] (-4.36,-2.02) circle (1.5pt);
\fill [color=black] (-0.29,-4.79) circle (1.5pt);
\fill [color=black] (-0.57,-4.77) circle (1.5pt);
\fill [color=black] (-0.86,-4.72) circle (1.5pt);
\fill [color=black] (-1.14,-4.66) circle (1.5pt);
\fill [color=black] (3.57,-3.21) circle (1.5pt);
\fill [color=black] (3.37,-3.42) circle (1.5pt);
\fill [color=black] (3.16,-3.61) circle (1.5pt);
\fill [color=black] (2.94,-3.8) circle (1.5pt);
\fill [color=black] (4.74,0.79) circle (1.5pt);
\fill [color=black] (4.77,0.5) circle (1.5pt);
\fill [color=black] (4.8,0.22) circle (1.5pt);
\fill [color=black] (4.8,-0.07) circle (1.5pt);
\fill [color=black] (2.34,4.19) circle (1.5pt);
\fill [color=black] (2.58,4.05) circle (1.5pt);
\fill [color=black] (2.82,3.88) circle (1.5pt);
\fill [color=black] (3.05,3.71) circle (1.5pt);
\fill [color=black] (-3.11,3.65) circle (1.5pt);
\fill [color=black] (-2.95,3.78) circle (1.5pt);
\fill [color=black] (-2.79,3.91) circle (1.5pt);
\fill [color=black] (-2.62,4.02) circle (1.5pt);
\fill [color=black] (-2.44,4.13) circle (1.5pt);
\fill [color=black] (-2.27,4.23) circle (1.5pt);
\fill [color=black] (-4.8,-0.15) circle (1.5pt);
\fill [color=black] (-4.8,0.05) circle (1.5pt);
\fill [color=black] (-4.79,0.26) circle (1.5pt);
\fill [color=black] (-4.78,0.46) circle (1.5pt);
\fill [color=black] (-4.75,0.66) circle (1.5pt);
\fill [color=black] (-4.72,0.87) circle (1.5pt);
\fill [color=black] (-2.87,-3.85) circle (1.5pt);
\fill [color=black] (-3.03,-3.72) circle (1.5pt);
\fill [color=black] (-3.19,-3.59) circle (1.5pt);
\fill [color=black] (-3.34,-3.45) circle (1.5pt);
\fill [color=black] (-3.48,-3.3) circle (1.5pt);
\fill [color=black] (-3.62,-3.15) circle (1.5pt);
\fill [color=black] (1.22,-4.64) circle (1.5pt);
\fill [color=black] (1.02,-4.69) circle (1.5pt);
\fill [color=black] (0.82,-4.73) circle (1.5pt);
\fill [color=black] (0.61,-4.76) circle (1.5pt);
\fill [color=black] (0.41,-4.78) circle (1.5pt);
\fill [color=black] (0.21,-4.8) circle (1.5pt);
\fill [color=black] (4.39,-1.94) circle (1.5pt);
\fill [color=black] (4.3,-2.13) circle (1.5pt);
\fill [color=black] (4.21,-2.31) circle (1.5pt);
\fill [color=black] (4.1,-2.49) circle (1.5pt);
\fill [color=black] (3.99,-2.66) circle (1.5pt);
\fill [color=black] (3.88,-2.83) circle (1.5pt);
\fill [color=black] (4.26,2.22) circle (1.5pt);
\fill [color=black] (4.35,2.04) circle (1.5pt);
\fill [color=black] (4.43,1.85) circle (1.5pt);
\fill [color=black] (4.5,1.66) circle (1.5pt);
\fill [color=black] (4.57,1.46) circle (1.5pt);
\fill [color=black] (4.63,1.27) circle (1.5pt);
\fill [color=black] (0.92,4.71) circle (1.5pt);
\fill [color=black] (1.12,4.67) circle (1.5pt);
\fill [color=black] (1.32,4.62) circle (1.5pt);
\fill [color=black] (1.51,4.56) circle (1.5pt);
\fill [color=black] (1.71,4.49) circle (1.5pt);
\fill [color=black] (1.9,4.41) circle (1.5pt);
\draw [color=black] (-2.08,4.32) circle (2.5pt);
\draw [color=black] (-3.26,3.52) circle (2.5pt);
\draw [color=black] (-4.16,2.4) circle (2.5pt);
\draw [color=black] (-4.68,1.07) circle (2.5pt);
\draw [color=black] (-4.79,-0.36) circle (2.5pt);
\draw [color=black] (-4.47,-1.75) circle (2.5pt);
\draw [color=black] (-3.75,-2.99) circle (2.5pt);
\draw [color=black] (-2.7,-3.97) circle (2.5pt);
\draw [color=black] (-1.41,-4.59) circle (2.5pt);
\draw [color=black] (0,-4.8) circle (2.5pt);
\draw [color=black] (1.41,-4.59) circle (2.5pt);
\draw [color=black] (2.7,-3.97) circle (2.5pt);
\draw [color=black] (3.75,-2.99) circle (2.5pt);
\draw [color=black] (4.47,-1.75) circle (2.5pt);
\draw [color=black] (4.79,-0.36) circle (2.5pt);
\draw [color=black] (4.68,1.07) circle (2.5pt);
\draw [color=black] (4.16,2.4) circle (2.5pt);
\draw [color=black] (3.26,3.52) circle (2.5pt);
\draw [color=black] (2.08,4.32) circle (2.5pt);
\draw [color=black] (0.72,4.75) circle (2.5pt);
\draw [color=black] (-0.72,4.75) circle (2.5pt);
\fill [color=black] (6.5,2.5) circle (1.5pt);
\fill [color=ffffff] (6.5,-2) circle (2.5pt);
\draw [color=black] (6.5,-2) circle (2.5pt);
\draw (6.65,4.65) node {\textbf{Jump probabilities:}};
\draw (7,2.9) node {\footnotesize $\frac{3}{10}$};
\draw (6,2.9) node {\footnotesize $\frac{3}{10}$};
\draw (6.7,1.35) node {\footnotesize $\frac{2}{5}$};
\draw (5.5,2.1) node {\footnotesize $\xi^k_{\alpha,i+1}$};
\draw (7.5,2.1) node {\footnotesize $\xi^k_{\alpha,i-1}$};
\draw (6.5,0.24) node {$\omega_k$};
\draw (5,-0.58) node[anchor=north west] {\footnotesize from $\xi^k_{\alpha,1}$:};
\draw (7,-1.6) node {\footnotesize $\frac{3}{14}$};
\draw (6,-1.6) node {\footnotesize $\frac{3}{14}$};
\draw (6.9,-3.25) node {\footnotesize $\frac{2}{7}$};
\draw (6.1,-3.25) node {\footnotesize $\frac{2}{7}$};
\draw (5.5,-4.3) node {\footnotesize $\omega_k$};
\draw (7.5,-4.3) node {\footnotesize $\omega_{k+1}$};
\draw (5.5,-2.4) node {\footnotesize $\xi^k_{\alpha,2}$};
\draw (7.85,-2.4) node {\footnotesize $\xi^{k+1}_{\alpha-1,N_{(\alpha-1)}-2}$};
\draw (4.85,3.9) node [anchor=north west]{\footnotesize from $\xi^k_{\alpha,i}; ~2 \leq i \leq N_\alpha-2$:};
\draw (0,2.77) node {\footnotesize $\omega_0$};
\draw (0.76,5.14) node {\footnotesize $\xi^0_{A,1}$};
\draw (-0.74,2.6) node {\footnotesize $\omega_{-1}$};
\draw (0.73,2.6) node {\footnotesize $\omega_1$};
\draw (-1.45,2.24) node {\footnotesize $\omega_{-2}$};
\draw (2.27,4.69) node {\footnotesize $\xi^1_{C,1}$};
\draw (-0.77,5.14) node {\footnotesize $\xi^{-1}_{B,1}$};
\end{scriptsize}
\end{tikzpicture}
\caption{Graph structure of the ideal dynamics $\{\widehat\eta_1(t):t\geq0\}$ related to the process. The picture illustrates the case $N_A=5$, $N_B=7$ and $N_C=9$. The inner vertices represent the configurations in $\Omega^N_0$, absorbing states for this dynamics. The outer vertices represent the metastable configurations in $\Omega^1_N\setminus\Omega^0_N$, which will be precisely defined in the next section. The arrows on the right completely describe the corresponding discrete-time jump chain.}
\label{fig6}
\end{figure}

The dynamics described in Figure~\ref{fig6} indicates that the trace on $\Omega^N_0$ is not a symmetric process and that this asymmetry can be balanced depending on the relative quantities of each type of particle. This behavior is at the origin of the drift which appears in Theorem~\ref{mainresult}.

We have, therefore, a strategy to analyze the trace of $\{\eta(t):t\geq0\}$ on $\Omega^N_0$. At first, we consider the trace on $\Omega^N_1$. This will be the subject of Section~\ref{trace1}.  Essentially, we have to answer the following question: starting from a configuration $\omega$ in $\Omega^N_1$, what is the distribution of the next visited configuration in $\Omega^N_1$? We split this question into two, depending if $\omega$ belongs to $\Omega^N_0$ or not. In the first case, we will see an interesting ``uniformity" for this distribution, in a sense to be clarified at Proposition \ref{taxadosabsorventes}. At this point the error terms in our estimates increase exponentially in~$N$, and here is where some constraints referring to the growth of $N$ arise. In the second case, we observe that the process is well approximated by the asymptotic Markov chain $\{\widehat\eta_1(t): t\geq0\}$.

To pass from the trace on $\Omega^N_1$ to the trace on $\Omega^N_0$, we have to look at the absorptions probabilities on $\Omega^N_0$ for the chain $\{\widehat\eta_1(t): t\geq0\}$ starting from $\Omega^N_1\setminus\Omega^N_0$. In Section~\ref{sectracoOmega0}, we present (approximations of) the jump rates $r_\beta(k)$, $k\in\Lambda_N$, defined on \eqref{rbeta}, as functions of these absorption probabilities, which are estimated in Section~\ref{secgnk} allowing us to prove, in Section~\ref{convBro}, the results for the trace process on $\Omega^N_0$ in the case where $N\uparrow\infty$ with~$\beta$.

All the above discussion also suggests what are the typical configurations that may appear between two consecutive visits to the set $\Omega^N_0$. In Section~\ref{poucofora} we will estimate the measure $\mu_\beta$ of these configurations and this will allow us to prove that the process spends a negligible time outside $\Omega^N_0$.

\section{The subset of configurations $\Omega^N_1$}\label{secOmegaN1}

In this section we define the set of configurations $\Omega^N_1$ establishing notation that identifies each one of its elements. Throughout the text, even when not explicitly mentioned, we are assuming that $M\geq3$.

\subsection{The configurations $\zeta^k_{\alpha,i}$}

For $k\in\Lambda_N$, $\alpha\in\{A,B,C\}$ and $0\leq i \leq N_\alpha$, denote by $\zeta^k_{\alpha,i}$ the configuration at distance $N_\alpha$ from $\omega_k$, obtained from $\omega_k\in\Omega^N_0$ by a meeting of two distinct particles of types different from $\alpha$ in the block of particles of type $\alpha$. The index $i$ indicates the position of this meeting. More precisely, let $f_\alpha$ be the position of the first particle of type $\alpha$ in the configuration $\omega_0$, i.e,
\begin{equation*}
f_\alpha=N_A\mb1\{\alpha\in\{B,C\}\}+N_B\mb1\{\alpha=C\}.
\end{equation*}
Then,
\begin{equation*}
\zeta^k_{\alpha,i}=\Theta^k\sigma^{f_{(\alpha+1)},f_\alpha+i}\sigma^{f_\alpha-1,f_\alpha-1+i}\omega_0.
\end{equation*}
Note that the extreme case $i=0$ (respectively $i=N_\alpha$) indicates that a particle of type $\alpha+1$ (respectively $\alpha-1$) has crossed the whole block of particles of type $\alpha$ until meeting a particle of type $\alpha-1$ (respectively $\alpha+1$).
As illustrated in Figure~\ref{fig1}, with this notation, $\zeta^k_{\alpha, N_\alpha}=\zeta^{k-1}_{\alpha+1,0}$, and these are the only configurations of this type with double representation.

\definecolor{ffqqqq}{rgb}{1,0,0} 
\definecolor{qqqqff}{rgb}{0,0,1} 
\definecolor{xfxfxf}{rgb}{0.5,0.5,0.5} 
\definecolor{ffffff}{rgb}{1,1,1} 

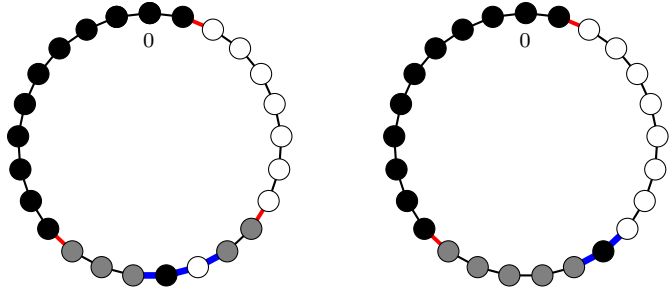
\begin{figure}[h!]
\centering
\begin{tikzpicture}[line cap=round,line join=round,>=triangle 45,x=1.0cm,y=1.0cm]
\clip(-2.17,-2.18) rectangle (7.15,2.06);
\draw [line width=0.8pt] (0,1.75)-- (0.44,1.7);
\draw [line width=1.5pt,color=ffqqqq] (0.44,1.7)-- (0.84,1.53);
\draw [line width=0.8pt] (0.84,1.53)-- (1.2,1.28);
\draw [line width=0.8pt] (1.2,1.28)-- (1.48,0.94);
\draw [line width=0.8pt] (1.48,0.94)-- (1.66,0.54);
\draw [line width=0.8pt] (1.66,0.54)-- (1.75,0.11);
\draw [line width=0.8pt] (1.75,0.11)-- (1.72,-0.33);
\draw [line width=0.8pt] (1.72,-0.33)-- (1.58,-0.75);
\draw [line width=1.5pt,color=ffqqqq] (1.58,-0.75)-- (1.35,-1.12);
\draw [line width=0.8pt] (1.35,-1.12)-- (1.03,-1.42);
\draw [line width=2.5pt,color=qqqqff] (1.03,-1.42)-- (0.64,-1.63);
\draw [line width=2.5pt,color=qqqqff] (0.64,-1.63)-- (0.22,-1.74);
\draw [line width=2.5pt,color=qqqqff] (0.22,-1.74)-- (-0.22,-1.74);
\draw [line width=0.8pt] (-0.22,-1.74)-- (-0.64,-1.63);
\draw [line width=0.8pt] (-0.64,-1.63)-- (-1.03,-1.42);
\draw [line width=1.5pt,color=ffqqqq] (-1.03,-1.42)-- (-1.35,-1.12);
\draw [line width=0.8pt] (-1.35,-1.12)-- (-1.58,-0.75);
\draw [line width=0.8pt] (-1.58,-0.75)-- (-1.72,-0.33);
\draw [line width=0.8pt] (-1.72,-0.33)-- (-1.75,0.11);
\draw [line width=0.8pt] (-1.75,0.11)-- (-1.66,0.54);
\draw [line width=0.8pt] (-1.66,0.54)-- (-1.48,0.94);
\draw [line width=0.8pt] (-1.48,0.94)-- (-1.2,1.28);
\draw [line width=0.8pt] (-1.2,1.28)-- (-0.84,1.53);
\draw [line width=0.8pt] (-0.84,1.53)-- (-0.44,1.7);
\draw (-0.21,1.6) node[anchor=north west] {\footnotesize $0$};
\draw [line width=0.8pt] (-0.44,1.7)-- (0,1.75);
\draw [line width=0.8pt] (5,1.75)-- (5.44,1.7);
\draw [line width=1.5pt,color=ffqqqq] (5.44,1.7)-- (5.84,1.53);
\draw [line width=0.8pt] (5.84,1.53)-- (6.2,1.28);
\draw [line width=0.8pt] (6.2,1.28)-- (6.48,0.94);
\draw [line width=0.8pt] (6.48,0.94)-- (6.66,0.54);
\draw [line width=0.8pt] (6.66,0.54)-- (6.75,0.11);
\draw [line width=0.8pt] (6.75,0.11)-- (6.72,-0.33);
\draw [line width=0.8pt] (6.72,-0.33)-- (6.58,-0.75);
\draw [line width=0.8pt] (6.58,-0.75)-- (6.35,-1.12);
\draw [line width=2.5pt,color=qqqqff] (6.35,-1.12)-- (6.03,-1.42);
\draw [line width=2.5pt,color=qqqqff] (6.03,-1.42)-- (5.64,-1.63);
\draw [line width=0.8pt] (5.64,-1.63)-- (5.22,-1.74);
\draw [line width=0.8pt] (5.22,-1.74)-- (4.78,-1.74);
\draw [line width=0.8pt] (4.78,-1.74)-- (4.36,-1.63);
\draw [line width=0.8pt] (4.36,-1.63)-- (3.97,-1.42);
\draw [line width=1.5pt,color=ffqqqq] (3.97,-1.42)-- (3.65,-1.12);
\draw [line width=0.8pt] (3.65,-1.12)-- (3.42,-0.75);
\draw [line width=0.8pt] (3.42,-0.75)-- (3.28,-0.33);
\draw [line width=0.8pt] (3.28,-0.33)-- (3.25,0.11);
\draw [line width=0.8pt] (3.25,0.11)-- (3.34,0.54);
\draw [line width=0.8pt] (3.34,0.54)-- (3.52,0.94);
\draw [line width=0.8pt] (3.52,0.94)-- (3.8,1.28);
\draw [line width=0.8pt] (3.8,1.28)-- (4.16,1.53);
\draw [line width=0.8pt] (4.16,1.53)-- (4.56,1.7);
\draw [line width=0.8pt] (4.56,1.7)-- (5,1.75);
\draw (4.79,1.6) node[anchor=north west] {\footnotesize $0$};
\begin{scriptsize}
\fill [color=black] (0,1.75) circle (4.0pt);
\fill [color=black] (0.44,1.7) circle (4.0pt);
\fill [color=ffffff] (0.84,1.53) circle (4.0pt);
\fill [color=ffffff] (1.2,1.28) circle (4.0pt);
\fill [color=ffffff] (1.48,0.94) circle (4.0pt);
\fill [color=ffffff] (1.66,0.54) circle (4.0pt);
\fill [color=ffffff] (1.75,0.11) circle (4.0pt);
\fill [color=ffffff] (1.72,-0.33) circle (4.0pt);
\fill [color=ffffff] (1.58,-0.75) circle (4.0pt);
\fill [color=xfxfxf] (1.35,-1.12) circle (4.0pt);
\fill [color=xfxfxf] (1.03,-1.42) circle (4.0pt);
\fill [color=ffffff] (0.64,-1.63) circle (4.0pt);
\fill [color=black] (0.22,-1.74) circle (4.0pt);
\fill [color=xfxfxf] (-0.22,-1.74) circle (4.0pt);
\fill [color=xfxfxf] (-0.64,-1.63) circle (4.0pt);
\fill [color=xfxfxf] (-1.03,-1.42) circle (4.0pt);
\fill [color=black] (-1.35,-1.12) circle (4.0pt);
\fill [color=black] (-1.58,-0.75) circle (4.0pt);
\fill [color=black] (-1.72,-0.33) circle (4.0pt);
\fill [color=black] (-1.75,0.11) circle (4.0pt);
\fill [color=black] (-1.66,0.54) circle (4.0pt);
\fill [color=black] (-1.48,0.94) circle (4.0pt);
\fill [color=black] (-1.2,1.28) circle (4.0pt);
\fill [color=black] (-0.84,1.53) circle (4.0pt);
\fill [color=black] (-0.44,1.7) circle (4.0pt);
\draw [color=black] (0,1.75) circle (4.0pt);
\draw [color=black] (0.44,1.7) circle (4.0pt);
\draw [color=black] (0.84,1.53) circle (4.0pt);
\draw [color=black] (1.2,1.28) circle (4.0pt);
\draw [color=black] (1.75,0.11) circle (4.0pt);
\draw [color=black] (1.72,-0.33) circle (4.0pt);
\draw [color=black] (1.58,-0.75) circle (4.0pt);
\draw [color=black] (1.35,-1.12) circle (4.0pt);
\draw [color=black] (1.03,-1.42) circle (4.0pt);
\draw [color=black] (1.66,0.54) circle (4.0pt);
\draw [color=black] (1.48,0.94) circle (4.0pt);
\draw [color=black] (0.64,-1.63) circle (4.0pt);
\draw [color=black] (0.22,-1.74) circle (4.0pt);
\draw [color=black] (-0.22,-1.74) circle (4.0pt);
\draw [color=black] (-0.64,-1.63) circle (4.0pt);
\draw [color=black] (-1.03,-1.42) circle (4.0pt);
\draw [color=black] (-1.35,-1.12) circle (4.0pt);
\draw [color=black] (-1.58,-0.75) circle (4.0pt);
\draw [color=black] (-1.72,-0.33) circle (4.0pt);
\draw [color=black] (-1.75,0.11) circle (4.0pt);
\draw [color=black] (-1.66,0.54) circle (4.0pt);
\draw [color=black] (-1.48,0.94) circle (4.0pt);
\draw [color=black] (-1.2,1.28) circle (4.0pt);
\draw [color=black] (-0.84,1.53) circle (4.0pt);
\draw [color=black] (-0.44,1.7) circle (4.0pt);
\fill [color=black] (0,1.75) circle (4.0pt);
\draw [color=black] (0,1.75) circle (4.0pt);
\fill [color=black] (-0.44,1.7) circle (4.0pt);
\draw [color=black] (-0.44,1.7) circle (4.0pt);
\fill [color=black] (5,1.75) circle (4.0pt);
\fill [color=black] (5.44,1.7) circle (4.0pt);
\fill [color=ffffff] (5.84,1.53) circle (4.0pt);
\fill [color=ffffff] (6.2,1.28) circle (4.0pt);
\fill [color=ffffff] (6.48,0.94) circle (4.0pt);
\fill [color=ffffff] (6.66,0.54) circle (4.0pt);
\fill [color=ffffff] (6.75,0.11) circle (4.0pt);
\fill [color=ffffff] (6.72,-0.33) circle (4.0pt);
\fill [color=ffffff] (6.58,-0.75) circle (4.0pt);
\fill [color=ffffff] (6.35,-1.12) circle (4.0pt);
\fill [color=black] (6.03,-1.42) circle (4.0pt);
\fill [color=xfxfxf] (5.64,-1.63) circle (4.0pt);
\fill [color=xfxfxf] (5.22,-1.74) circle (4.0pt);
\fill [color=xfxfxf] (4.78,-1.74) circle (4.0pt);
\fill [color=xfxfxf] (4.36,-1.63) circle (4.0pt);
\fill [color=xfxfxf] (3.97,-1.42) circle (4.0pt);
\fill [color=black] (3.65,-1.12) circle (4.0pt);
\fill [color=black] (3.42,-0.75) circle (4.0pt);
\fill [color=black] (3.28,-0.33) circle (4.0pt);
\fill [color=black] (3.25,0.11) circle (4.0pt);
\fill [color=black] (3.34,0.54) circle (4.0pt);
\fill [color=black] (3.52,0.94) circle (4.0pt);
\fill [color=black] (3.8,1.28) circle (4.0pt);
\fill [color=black] (4.16,1.53) circle (4.0pt);
\fill [color=black] (4.56,1.7) circle (4.0pt);
\draw [color=black] (5,1.75) circle (4.0pt);
\draw [color=black] (5.44,1.7) circle (4.0pt);
\draw [color=black] (5.84,1.53) circle (4.0pt);
\draw [color=black] (6.2,1.28) circle (4.0pt);
\draw [color=black] (6.48,0.94) circle (4.0pt);
\draw [color=black] (6.66,0.54) circle (4.0pt);
\draw [color=black] (6.75,0.11) circle (4.0pt);
\draw [color=black] (6.72,-0.33) circle (4.0pt);
\draw [color=black] (6.58,-0.75) circle (4.0pt);
\draw [color=black] (6.35,-1.12) circle (4.0pt);
\draw [color=black] (6.03,-1.42) circle (4.0pt);
\draw [color=black] (5.64,-1.63) circle (4.0pt);
\draw [color=black] (5.22,-1.74) circle (4.0pt);
\draw [color=black] (4.78,-1.74) circle (4.0pt);
\draw [color=black] (4.36,-1.63) circle (4.0pt);
\draw [color=black] (3.97,-1.42) circle (4.0pt);
\draw [color=black] (3.65,-1.12) circle (4.0pt);
\draw [color=black] (3.42,-0.75) circle (4.0pt);
\draw [color=black] (3.28,-0.33) circle (4.0pt);
\draw [color=black] (3.25,0.11) circle (4.0pt);
\draw [color=black] (3.34,0.54) circle (4.0pt);
\draw [color=black] (3.52,0.94) circle (4.0pt);
\draw [color=black] (3.8,1.28) circle (4.0pt);
\draw [color=black] (4.16,1.53) circle (4.0pt);
\draw [color=black] (4.56,1.7) circle (4.0pt);
\end{scriptsize}
\end{tikzpicture}
\caption{The configurations $\zeta^2_{B,2}$ and $\zeta^2_{B,0}=\zeta^3_{A,8}$ for $N_A=8,N_B=5,N_C=12$. The white, gray and black circles represent respectively particles of type $A$, $B$ and $C$. We use blue (respectively red) edges to indicate transpositions that occur at rate $1$ (respectively $e^{-\beta}$).}
\label{fig1}
\end{figure}

For $k\in\Lambda_N$ and $\alpha\in\{A,B,C\}$, we denote by $\mc F^{N,k}_\alpha$ the corresponding set of such configurations:
\begin{equation}\label{FNk}
    \mc F^{N,k}_\alpha=\{\zeta^k_{\alpha,i} : 0\leq i \leq N_\alpha\}.
\end{equation}

We will see in Section \ref{trace1} that, if the process starts from $\omega_k$, as $\beta\uparrow\infty$, the configurations in the set $\bigcup_{\alpha: N_\alpha=M}\mc F^{N,k}_\alpha$ are those that can be reached in a time of order $e^{M\beta}$ that allow the process to escape from the basin of attraction of the configuration $\omega_k$.

\subsection{The configurations $\xi^k_{\alpha,i}$} For $k\in\Lambda_N$, $\alpha\in\{A,B,C\}$ and $1\leq\ i \leq N_\alpha-1$ we denote by $\xi^k_{\alpha,i}$ the configuration obtained from $\zeta^k_{\alpha,i}$ by interchanging the positions of the two distinct particles that have met in the block of particles of type $\alpha$. More precisely:
\begin{equation*}
\xi^k_{\alpha,i}=\Theta^k\sigma^{f_\alpha+i-1,f_\alpha+i}\zeta^0_{\alpha,i}.
\end{equation*}
Note that the jump leading $\zeta^k_{\alpha,i}$ to $\xi^k_{\alpha,i}$ has rate $1$.

Again, as illustrated in Figure \ref{fig2}, it happens that some of these configurations have double representations, namely $\xi^k_{\alpha,N_\alpha-1}=\xi^{k-1}_{\alpha+1,1}$.

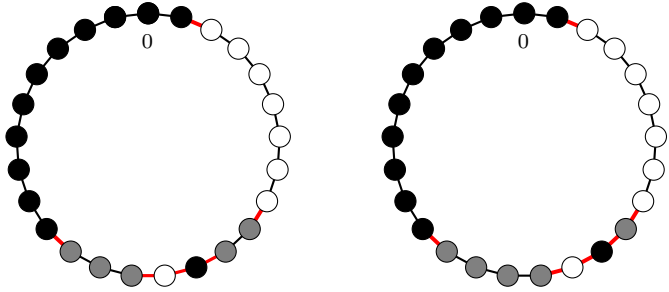
\begin{figure}[h!]
\centering
\begin{tikzpicture}[line cap=round,line join=round,>=triangle 45,x=1.0cm,y=1.0cm]
\clip(-2.05,-2.2) rectangle (7.07,2.04);
\draw [line width=0.8pt] (0,1.75)-- (0.44,1.7);
\draw [line width=1.5pt,color=ffqqqq] (0.44,1.7)-- (0.84,1.53);
\draw [line width=0.8pt] (0.84,1.53)-- (1.2,1.28);
\draw [line width=0.8pt] (1.2,1.28)-- (1.48,0.94);
\draw [line width=0.8pt] (1.48,0.94)-- (1.66,0.54);
\draw [line width=0.8pt] (1.66,0.54)-- (1.75,0.11);
\draw [line width=0.8pt] (1.75,0.11)-- (1.72,-0.33);
\draw [line width=0.8pt] (1.72,-0.33)-- (1.58,-0.75);
\draw [line width=1.5pt,color=ffqqqq] (1.58,-0.75)-- (1.35,-1.12);
\draw [line width=0.8pt] (1.35,-1.12)-- (1.03,-1.42);
\draw [line width=1.2pt,color=ffqqqq] (1.03,-1.42)-- (0.64,-1.63);
\draw [line width=1.2pt,color=ffqqqq] (0.64,-1.63)-- (0.22,-1.74);
\draw [line width=1.2pt,color=ffqqqq] (0.22,-1.74)-- (-0.22,-1.74);
\draw [line width=0.8pt] (-0.22,-1.74)-- (-0.64,-1.63);
\draw [line width=0.8pt] (-0.64,-1.63)-- (-1.03,-1.42);
\draw [line width=1.5pt,color=ffqqqq] (-1.03,-1.42)-- (-1.35,-1.12);
\draw [line width=0.8pt] (-1.35,-1.12)-- (-1.58,-0.75);
\draw [line width=0.8pt] (-1.58,-0.75)-- (-1.72,-0.33);
\draw [line width=0.8pt] (-1.72,-0.33)-- (-1.75,0.11);
\draw [line width=0.8pt] (-1.75,0.11)-- (-1.66,0.54);
\draw [line width=0.8pt] (-1.66,0.54)-- (-1.48,0.94);
\draw [line width=0.8pt] (-1.48,0.94)-- (-1.2,1.28);
\draw [line width=0.8pt] (-1.2,1.28)-- (-0.84,1.53);
\draw [line width=0.8pt] (-0.84,1.53)-- (-0.44,1.7);
\draw (-0.21,1.6) node[anchor=north west] {\footnotesize $0$};
\draw [line width=0.8pt] (-0.44,1.7)-- (0,1.75);
\draw [line width=0.8pt] (5,1.75)-- (5.44,1.7);
\draw [line width=1.5pt,color=ffqqqq] (5.44,1.7)-- (5.84,1.53);
\draw [line width=0.8pt] (5.84,1.53)-- (6.2,1.28);
\draw [line width=0.8pt] (6.2,1.28)-- (6.48,0.94);
\draw [line width=0.8pt] (6.48,0.94)-- (6.66,0.54);
\draw [line width=0.8pt] (6.66,0.54)-- (6.75,0.11);
\draw [line width=0.8pt] (6.75,0.11)-- (6.72,-0.33);
\draw [line width=0.8pt] (6.72,-0.33)-- (6.58,-0.75);
\draw [line width=1.5pt,color=ffqqqq] (6.58,-0.75)-- (6.35,-1.12);
\draw [line width=1.5pt,color=ffqqqq] (6.35,-1.12)-- (6.03,-1.42);
\draw [line width=1.5pt,color=ffqqqq] (6.03,-1.42)-- (5.64,-1.63);
\draw [line width=1.5pt,color=ffqqqq] (5.64,-1.63)-- (5.22,-1.74);
\draw [line width=0.8pt] (5.22,-1.74)-- (4.78,-1.74);
\draw [line width=0.8pt] (4.78,-1.74)-- (4.36,-1.63);
\draw [line width=0.8pt] (4.36,-1.63)-- (3.97,-1.42);
\draw [line width=1.5pt,color=ffqqqq] (3.97,-1.42)-- (3.65,-1.12);
\draw [line width=0.8pt] (3.65,-1.12)-- (3.42,-0.75);
\draw [line width=0.8pt] (3.42,-0.75)-- (3.28,-0.33);
\draw [line width=0.8pt] (3.28,-0.33)-- (3.25,0.11);
\draw [line width=0.8pt] (3.25,0.11)-- (3.34,0.54);
\draw [line width=0.8pt] (3.34,0.54)-- (3.52,0.94);
\draw [line width=0.8pt] (3.52,0.94)-- (3.8,1.28);
\draw [line width=0.8pt] (3.8,1.28)-- (4.16,1.53);
\draw [line width=0.8pt] (4.16,1.53)-- (4.56,1.7);
\draw [line width=0.8pt] (4.56,1.7)-- (5,1.75);
\draw (4.79,1.6) node[anchor=north west] {\footnotesize $0$};
\begin{scriptsize}
\fill [color=black] (0,1.75) circle (4.0pt);
\fill [color=black] (0.44,1.7) circle (4.0pt);
\fill [color=ffffff] (0.84,1.53) circle (4.0pt);
\fill [color=ffffff] (1.2,1.28) circle (4.0pt);
\fill [color=ffffff] (1.48,0.94) circle (4.0pt);
\fill [color=ffffff] (1.66,0.54) circle (4.0pt);
\fill [color=ffffff] (1.75,0.11) circle (4.0pt);
\fill [color=ffffff] (1.72,-0.33) circle (4.0pt);
\fill [color=ffffff] (1.58,-0.75) circle (4.0pt);
\fill [color=xfxfxf] (1.35,-1.12) circle (4.0pt);
\fill [color=xfxfxf] (1.03,-1.42) circle (4.0pt);
\fill [color=black] (0.64,-1.63) circle (4.0pt);
\fill [color=ffffff] (0.22,-1.74) circle (4.0pt);
\fill [color=xfxfxf] (-0.22,-1.74) circle (4.0pt);
\fill [color=xfxfxf] (-0.64,-1.63) circle (4.0pt);
\fill [color=xfxfxf] (-1.03,-1.42) circle (4.0pt);
\fill [color=black] (-1.35,-1.12) circle (4.0pt);
\fill [color=black] (-1.58,-0.75) circle (4.0pt);
\fill [color=black] (-1.72,-0.33) circle (4.0pt);
\fill [color=black] (-1.75,0.11) circle (4.0pt);
\fill [color=black] (-1.66,0.54) circle (4.0pt);
\fill [color=black] (-1.48,0.94) circle (4.0pt);
\fill [color=black] (-1.2,1.28) circle (4.0pt);
\fill [color=black] (-0.84,1.53) circle (4.0pt);
\fill [color=black] (-0.44,1.7) circle (4.0pt);
\draw [color=black] (0,1.75) circle (4.0pt);
\draw [color=black] (0.44,1.7) circle (4.0pt);
\draw [color=black] (0.84,1.53) circle (4.0pt);
\draw [color=black] (1.2,1.28) circle (4.0pt);
\draw [color=black] (1.75,0.11) circle (4.0pt);
\draw [color=black] (1.72,-0.33) circle (4.0pt);
\draw [color=black] (1.58,-0.75) circle (4.0pt);
\draw [color=black] (1.35,-1.12) circle (4.0pt);
\draw [color=black] (1.03,-1.42) circle (4.0pt);
\draw [color=black] (1.66,0.54) circle (4.0pt);
\draw [color=black] (1.48,0.94) circle (4.0pt);
\draw [color=black] (0.64,-1.63) circle (4.0pt);
\draw [color=black] (0.22,-1.74) circle (4.0pt);
\draw [color=black] (-0.22,-1.74) circle (4.0pt);
\draw [color=black] (-0.64,-1.63) circle (4.0pt);
\draw [color=black] (-1.03,-1.42) circle (4.0pt);
\draw [color=black] (-1.35,-1.12) circle (4.0pt);
\draw [color=black] (-1.58,-0.75) circle (4.0pt);
\draw [color=black] (-1.72,-0.33) circle (4.0pt);
\draw [color=black] (-1.75,0.11) circle (4.0pt);
\draw [color=black] (-1.66,0.54) circle (4.0pt);
\draw [color=black] (-1.48,0.94) circle (4.0pt);
\draw [color=black] (-1.2,1.28) circle (4.0pt);
\draw [color=black] (-0.84,1.53) circle (4.0pt);
\draw [color=black] (-0.44,1.7) circle (4.0pt);
\fill [color=black] (0,1.75) circle (4.0pt);
\draw [color=black] (0,1.75) circle (4.0pt);
\fill [color=black] (-0.44,1.7) circle (4.0pt);
\draw [color=black] (-0.44,1.7) circle (4.0pt);
\fill [color=black] (5,1.75) circle (4.0pt);
\fill [color=black] (5.44,1.7) circle (4.0pt);
\fill [color=ffffff] (5.84,1.53) circle (4.0pt);
\fill [color=ffffff] (6.2,1.28) circle (4.0pt);
\fill [color=ffffff] (6.48,0.94) circle (4.0pt);
\fill [color=ffffff] (6.66,0.54) circle (4.0pt);
\fill [color=ffffff] (6.75,0.11) circle (4.0pt);
\fill [color=ffffff] (6.72,-0.33) circle (4.0pt);
\fill [color=ffffff] (6.58,-0.75) circle (4.0pt);
\fill [color=xfxfxf] (6.35,-1.12) circle (4.0pt);
\fill [color=black] (6.03,-1.42) circle (4.0pt);
\fill [color=ffffff] (5.64,-1.63) circle (4.0pt);
\fill [color=xfxfxf] (5.22,-1.74) circle (4.0pt);
\fill [color=xfxfxf] (4.78,-1.74) circle (4.0pt);
\fill [color=xfxfxf] (4.36,-1.63) circle (4.0pt);
\fill [color=xfxfxf] (3.97,-1.42) circle (4.0pt);
\fill [color=black] (3.65,-1.12) circle (4.0pt);
\fill [color=black] (3.42,-0.75) circle (4.0pt);
\fill [color=black] (3.28,-0.33) circle (4.0pt);
\fill [color=black] (3.25,0.11) circle (4.0pt);
\fill [color=black] (3.34,0.54) circle (4.0pt);
\fill [color=black] (3.52,0.94) circle (4.0pt);
\fill [color=black] (3.8,1.28) circle (4.0pt);
\fill [color=black] (4.16,1.53) circle (4.0pt);
\fill [color=black] (4.56,1.7) circle (4.0pt);
\draw [color=black] (5,1.75) circle (4.0pt);
\draw [color=black] (5.44,1.7) circle (4.0pt);
\draw [color=black] (5.84,1.53) circle (4.0pt);
\draw [color=black] (6.2,1.28) circle (4.0pt);
\draw [color=black] (6.48,0.94) circle (4.0pt);
\draw [color=black] (6.66,0.54) circle (4.0pt);
\draw [color=black] (6.75,0.11) circle (4.0pt);
\draw [color=black] (6.72,-0.33) circle (4.0pt);
\draw [color=black] (6.58,-0.75) circle (4.0pt);
\draw [color=black] (6.35,-1.12) circle (4.0pt);
\draw [color=black] (6.03,-1.42) circle (4.0pt);
\draw [color=black] (5.64,-1.63) circle (4.0pt);
\draw [color=black] (5.22,-1.74) circle (4.0pt);
\draw [color=black] (4.78,-1.74) circle (4.0pt);
\draw [color=black] (4.36,-1.63) circle (4.0pt);
\draw [color=black] (3.97,-1.42) circle (4.0pt);
\draw [color=black] (3.65,-1.12) circle (4.0pt);
\draw [color=black] (3.42,-0.75) circle (4.0pt);
\draw [color=black] (3.28,-0.33) circle (4.0pt);
\draw [color=black] (3.25,0.11) circle (4.0pt);
\draw [color=black] (3.34,0.54) circle (4.0pt);
\draw [color=black] (3.52,0.94) circle (4.0pt);
\draw [color=black] (3.8,1.28) circle (4.0pt);
\draw [color=black] (4.16,1.53) circle (4.0pt);
\draw [color=black] (4.56,1.7) circle (4.0pt);
\end{scriptsize}
\end{tikzpicture}
\caption{The configurations $\xi^2_{B,2}$ and $\xi^2_{B,1}=\xi^3_{A,7}$ for $N_A=8$, $N_B=5$, $N_C=12$. The white, gray and black circles represent respectively particles of type $A$, $B$ and $C$.}
\label{fig2}
\end{figure}
We denote by $\mc G^N$ the space of such configurations:
\begin{equation}\label{GNk}
\mc G^{N,k}_\alpha=\{\xi^k_{\alpha,i} : 1\leq i \leq N_\alpha-1\}, \; \quad \mc G^N=\bigcup_{k\in\Lambda_N}\;\bigcup_{\alpha\in\{A,B,C\}}\mc G^{N,k}_\alpha.
\end{equation}
From a configuration in $\mc G^N$, each possible jump has rate $e^{-\beta}$. In the particular case of equal densities these configurations are local minima of the energy $\bb H$ defined in~\eqref{1}. In the next section we analyze the trace of the ABC model on the set
\begin{equation*}
    \Omega^N_1:=\Omega^N_0\cup\mc G^N.
\end{equation*}

\section{Trace of $\{\eta(t):t\geq0\}$ on $\Omega^N_1$}\label{trace1}

For a configuration $\omega_k\in\Omega^N_0$, denote by $V(\omega_k)$ the set of configurations that can be obtained from $\omega_k$ after a sequence of nearest neighbor transpositions of type $(\alpha,\alpha+1)\to(\alpha+1,\alpha)$, i.e, rate $e^{-\beta}$ jumps. In other words, $V(\omega_k)$ is the set of configurations from which we can arrive at $\omega_k$ just performing rate $1$ jumps. Recall the definition of $\Delta^n_k$ in~\eqref{Delatank} and note that $\bigcup_{n=0}^M\Delta^n_k\subset V(\omega_k)$. For $\omega\in\Omega^N$, denote by $R(\omega)$ and $B(\omega)$ the sets of configurations that can be obtained from $\omega$ by a simple nearest neighbor transposition of types $(\alpha, \alpha+1)\to(\alpha+1,\alpha)$ and $(\alpha+1,\alpha)\to(\alpha,\alpha+1)$, respectively. 
Note that, if $\omega\in\Delta^n_k$, for $1\leq n\leq M-1$, then
\begin{equation}\label{R+B-}
R(\omega)\subseteq\Delta^{n+1}_k, ~~~B(\omega)\subseteq\Delta^{n-1}_k.
\end{equation}
Now, define $B^*_k(\omega)=B(\omega)\cap V(\omega_k)$ and $D^*_k(\omega)=B(\omega)\setminus V(\omega_k)$. In words, for a configuration $\omega\in V(\omega_k)$, the set $D^*_k(\omega)$ is formed by the configurations that are not in $V(\omega_k)$ which can be reached from $\omega$ after a rate~$1$~jump. The configurations $\omega\in\bigcup_{\alpha: N_\alpha=M}\mc F^{N,k}_\alpha$ are those in $\bigcup_{n=0}^M\Delta^n_k$ for which $D^*_k(\omega)\neq\emptyset$. We can see in Figure \ref{fig1} that, for example, $|R(\zeta^2_{B,2})|=3$, $|B^*_k(\zeta^2_{B,2})|=2$ and $D^*_k(\zeta^2_{B,2})=\{\xi^2_{B,2}\}$.

In the graph representation of a configuration, call an edge \emph{blue}, \emph{red}, or \emph{black} if, respectively, the particles it links exchange their positions at rate $1$, $e^{-\beta}$ or are of the same type. With this convention, $|R(\omega)|$ and $|B(\omega)|$ are the numbers of red and blue edges of the configuration $\omega$ and $V(\omega_k)$ is the set of configurations obtained from $\omega_k$ by a sequence of transpositions performed only in red edges. A configuration $\omega\in V(\omega_k)$ can have two kinds of blue edges (call them blue$_1$ and blue$_2$), whose transposition leads to configurations in $B^*_k(\omega)$ and $D^*_k(\omega)$ respectively.

\begin{lemma}\label{lemma1} Let $k\in\Lambda_N$ and $\omega\in V(\omega_k)$ then
\begin{equation}\label{ineqlemmma1}
    |R(\omega)|\leq|B^*_k(\omega)|+3.
\end{equation}
\end{lemma}

\begin{proof} For $\omega=\omega_k$, equality holds in \eqref{ineqlemmma1}, because the configuration $\omega_k$ has three red edges and no blue edges. Therefore, to conclude the proof by induction, we just have to check that if \eqref{ineqlemmma1} holds for a configuration $\omega\in V(\omega_k)$, then it remains true after a transposition in a red edge $(l,l+1)$ of $\omega$.

Observe that the transposition in the edge $(l,l+1)$ only changes the color of the three adjacent edges $(l-1,l)$, $(l,l+1)$ and $(l+1,l+2)$. After this transposition, the initially red edge $(l,l+1)$ becomes a blue$_1$ edge. For the other two edges, it is easy to check that black becomes red, blue becomes black, and red becomes blue. And then, by checking all the possible cases we see that
\begin{equation*}
|R(\sigma^{l,l+1}\omega)|-|B^*_k(\sigma^{l,l+1}\omega)|\leq |R(\omega)|-|B^*_k(\omega)|,
\end{equation*}
which completes the prove.
\end{proof}

For any subset $\Pi\subset\Omega^N$, denote by $H_\Pi$ and $H^+_\Pi$, respectively, the hitting time and the first return to $\Pi$:
\begin{equation*}
\begin{split}
& H_\Pi \,=\, \inf \big\{ t > 0 : \eta(t) \in \Pi \big\}\,, \\
&\quad H^+_\Pi \,=\, \inf\big\{ t>0 : \eta (t) \in \Pi \,,\,
\eta(s) \not= \eta(0) \;\;\textrm{for some $0< s < t$}\big\}\,.
\end{split}
\end{equation*}

\begin{corollary}\label{corol1} For any $\beta>\log4$, and $k\in\Lambda_N$
\begin{equation}\label{7}
\Prob{\omega_k}{\beta}{H_{\Delta^M_k}<H^+_{\omega_k}}\leq \left(4e^{-\beta}\right)^{M-1},
\end{equation}
and, for each $\omega\in\Delta^{M-1}_k$
\begin{equation}\label{6}
\Prob{\omega}{\beta}{H_{\Delta^M_k}<H_{\omega_k}}\leq 4e^{-\beta}.
\end{equation}
\end{corollary}

\begin{proof} By the observation \eqref{R+B-}, if the current state of the process is a configuration $\omega\in\Delta^n_k$, $1\leq n \leq M-1$, the next visited configuration belongs to $\Delta^{n+1}_k$ with probability
\begin{equation}\label{16}
    p_\beta(\omega)=\frac{|R(\omega)|e^{-\beta}}{|B(\omega)|+|R(\omega)|e^{-\beta}},
\end{equation}
and to $\Delta^{n-1}_k$ with probability $1-p_\beta(\omega)$.
By Lemma \ref{lemma1},
\begin{equation}\label{17}
p_\beta(\omega)\leq\frac{4e^{-\beta}}{1+4e^{-\beta}}=: p_\beta.
\end{equation}
Consider the random walk on $\{0,1,\ldots,M\}$ that jumps from $n\in\{1,\ldots,M-1\}$ to $n+1$ with probability $p_\beta$ and to $n-1$ with probability $1-p_\beta$. We know that, starting from $n$, this walk reaches $M$ before $0$ with probability
\begin{equation*}
    p_\beta(n,M)=\frac{\left(e^\beta/4\right)^n-1}{\left(e^\beta/4\right)^M-1}.
\end{equation*}
A simple coupling argument allows us to dominate the probabilities appearing in \eqref{7} and \eqref{6} respectively by $p_\beta(1,M)$ and $p_\beta(M-1,M)$, from which we get the corresponding bounds.
\end{proof}

Recall the definition of the set $\mc F^{N,k}_\alpha$ given in \eqref{FNk}. For each $\omega\in\Delta^M_k$, $k\in\Lambda_N$, consider the probability measure $\Phi^N(\omega,\cdot)$ defined on~$\Omega^N_1$ in the following way:

\noindent If $\omega\notin\bigcup_{\alpha: N_\alpha=M}\mc F^{N,k}_\alpha$,
\begin{equation*}
\Phi^N(\omega,\{\omega_k\})=1;
\end{equation*}
If $\alpha$ is 	such that $N_\alpha=M$,
\begin{equation*}
\Phi^N(\zeta^k_{\alpha,0},\{\omega_k\})=\Phi^N(\zeta^k_{\alpha,0},\{\omega_{k+1}\})=\Phi^N(\zeta^k_{\alpha,M},\{\omega_k\})=\Phi^N(\zeta^k_{\alpha,M},\{\omega_{k-1}\})=\frac{1}{2};
\end{equation*}
and, for $1\leq i \leq M-1$,
\begin{equation*}
\Phi^N(\zeta^k_{\alpha,i},\Pi)=\begin{cases}
\frac{1}{3} & \text{if $\Pi=\{\xi^k_{\alpha,i}\}$,}\\
\frac{2}{3} & \text{if $\Pi=\{\omega_k\}$}.
\end{cases}
\end{equation*}

Throughout the paper we adopt the convention that $C_0<\infty$ is a constant independent of $N_A$, $N_B$, $N_C$ and $\beta$ whose value may change from line to line.

\begin{lemma}\label{lemma2} There exists a constant $C_0$ such that, for all $\beta>0$, $k\in\Lambda_N$, and $\omega\in\Delta^M_k$,
\begin{equation*}
\left|\Prob{\omega}{\beta}{\eta(H_{\Omega^N_1})\in \Pi}-\Phi^N(\omega,\Pi)\right|\leq C_0Ne^{-\beta}, ~\Pi\subset\Omega^N_1.
\end{equation*}
\end{lemma}

\begin{proof} We examine each case separately. Suppose the process starts from a configuration $\omega\in\Delta^M_k\setminus\bigcup_{\alpha: N_\alpha=M}\mc F^{N,k}_\alpha$. In this case $D^*_k(\omega)=\emptyset$ and $B(\omega)\subseteq\Delta^{M-1}_k$, and then, as in the previous proof, starting from $\omega$, the next visited configuration belongs to $\Delta^{M-1}_k$ with high probability $1-p_\beta(\omega)\geq 1-p_\beta$, for $p_\beta(\omega)$ and $p_\beta$ given in \eqref{16} and \eqref{17}. Now observe that from $\Delta^{M-1}_k$ to reach a configuration in $\Omega^N_1\setminus\{\omega_k\}$ the process has to cross $\Delta^M_k$. So, conditioning in the first jump and using the second part of Corollary \ref{corol1} we get that $\Prob{\omega}{\beta}{H_{\Omega^N_1}\neq H_{\omega_k}}\leq C_0e^{-\beta}$.

Now suppose the process starts from $\zeta^k_{\alpha,i}$ for $1\leq i \leq N_\alpha-1$, $N_\alpha=M$. As illustrated in Figure~\ref{fig1}, from $\zeta^k_{\alpha,i}$ there are three possible rate $e^{-\beta}$ jumps and three possible rate $1$ jumps, one of these leading to $\xi^k_{\alpha,i}$ and the others two leading to~$\Delta^{M-1}_k$. So, as before, we obtain the corresponding hitting probability conditioning in the first jump and using the second part of Corollary \ref{corol1}.

Finally, suppose the process starts from $\zeta^k_{\alpha,0}$. As illustrated in Figure~\ref{fig1}, from this configuration there are two possible rate $e^{-\beta}$ jumps and two possible rate $1$ jumps, one of these leading to $B^*_k(\zeta^k_{\alpha,0})\subset\Delta^{M-1}_k$ and the other leading to $D^*_k(\zeta^k_{\alpha,0})\subset V(\omega_{k+1})$, from which we can reach $\omega_{k+1}$ (before any other configuration in $\Omega^N_1$) performing $N_{(\alpha-1)}-1<2N$ jumps that have high probability $1/(1+4e^{-\beta})$. This is done by movements of the detached particle of type $\alpha+1$ in counterclockwise direction inside the domain of particles of type $\alpha-1$ until it meets the other particles of type $\alpha+1$.  Therefore
\begin{eqnarray*}
\Prob{\zeta^k_{\alpha,0}}{\beta}{\eta\left(H_{\Omega^N_1}\right)=\omega_{k\pm 1}}&\geq&\frac{1}{2+2e^{-\beta}}\left(\frac{1}{1+4e^{-\beta}}\right)^{2N} \\
&=&\frac{1}{2}+\mc O(Ne^{-\beta}).
\end{eqnarray*}
The argument for the case in which the process starts from $\zeta^k_{\alpha,M}$ is analogous.
\end{proof}

Actually, for many configurations $\omega\in\Delta^M_k$ we could have obtained a better estimation of the hitting distribution in $\Omega^N_1$, considering from $\omega$ how many rate $e^{-\beta}$ jumps are necessary in order to avoid that the first visited configuration in $\Omega^N_1$ will be $\omega_k$. However, to take advantage of such more precise information a much more complex analysis would be needed in the proof of Proposition~\ref{taxadosabsorventes} below.

Denote by $\{\eta_1(t): t\geq0\}$ the trace of the process $\{\eta(t): t\geq0\}$ on $\Omega^N_1$. It is defined as in \eqref{deftrace} with $\Omega^N_0$ changed by $\Omega^N_1$. Surprisingly, as we will see in Proposition~\ref{taxadosabsorventes}, for this process the jump rates from $\omega_k$ to any configuration in $\bigcup_{\alpha: N_\alpha=M}\mc G^{N,k}_\alpha$ are, asymptotically, the same. This is due to the remarkable fact (somewhat hidden in the next proof) that, as $\beta\uparrow\infty$, the position $i$ of the meeting of the two different detached particles is asymptotically distributed in $\{0,1,\ldots,M\}$ as $\left(\frac{1}{2M},\frac{1}{M},\ldots,\frac{1}{M},\frac{1}{2M}\right)$.

The proof of the next lemma is based on a combinatorial identity that was first obtained by computing, in two different ways, some hitting probabilities for a simplified dynamics related the ABC model (two biased random walks). However, in Section~\ref{appendix} a completely elementary proof for this identity is presented.

Later we will assume stronger restrictions in the way that $N\uparrow\infty$ as $\beta\uparrow\infty$, but for now, inspired in the estimate obtained in Lemma~\ref{lemma2}, it is already natural to assume that
\begin{equation}\label{Ne-betato0}
\lim_{\beta\to\infty}Ne^{-\beta}=0.
\end{equation} 
 
\begin{lemma}\label{lemma4} Assume \eqref{Ne-betato0}. There exist constants $C_0$ and $\beta_0$ such that for all $\beta>\beta_0$, $k\in\Lambda_N$, and $\alpha\in\{A,B,C\}$ such that $N_\alpha=M$,
\begin{equation}\label{9}
\left|\Prob{\omega_k}{\beta}{H_{\Delta^M_k}=H_{\zeta^k_{\alpha,i}}<H^+_{\omega_k}}-q_ie^{-(M-1)\beta}\right|\leq C_0M\left(4e^{-\beta}\right)^M,
\end{equation}
where
\begin{equation*}
q_i=
\begin{cases}
\frac{2}{3} & \text{if $1\leq i \leq M-1$}, \\
\frac{1}{3} & \text{if $i=0$ or $M$}.
\end{cases}
\end{equation*}
\end{lemma}

\begin{proof} Let us  decompose the event $[H_{\Delta^M_k}=H_{\zeta^k_{\alpha,i}}<H^+_{\omega_k}]$ in the number of jumps to go from $\omega_k$ to $\zeta^k_{\alpha,i}$. Denote by $\tau_l$ the instant of the $l$-th jump of the chain $\{\eta(t): t\geq0\}$. By the observation \eqref{R+B-}, in each step of a path corresponding to this event, the distance to $\omega_k$ either increases or decreases by $1$ unit. So
\begin{equation}\label{9,1}
\Prob{\omega_k}{\beta}{H_{\Delta^M_k}=H_{\zeta^k_{\alpha,i}}<H^+_{\omega_k}}=\sum_{l=0}^\infty\Prob{\omega_k}{\beta}{\tau_{M+2l}=H_{\Delta^M_k}=H_{\zeta^k_{\alpha,i}}<H^+_{\omega_k}}.   
\end{equation}
A path from $\omega_k$
to $\Delta^M_k$ of size $M+2l$ should increase the distance to $\omega_k$ in $M+l$ steps and decrease in $l$ steps.
As already observed in the proof of Corollary~\ref{corol1}, from any configuration $\omega\in\cup_{n=1}^{M-1}\Delta^n_k$, the probability that the distance to $\omega_k$ increases in the next jump of the chain is bounded by $4e^{-\beta}$.
Now, observe that the number of possible evolutions of the distance to $\omega_k$ along a path from $\omega_k$ 
to $\Delta^M_k$ of size $M+2l$ is bounded by $\binom{M+2l}{l}$. Decomposing the event $[\tau_{M+2l}=H_{\Delta^M_k}<H^+_{\omega_k}]$ in these possible profiles and then applying inductively the strong Markov property for each term, we obtain
\begin{equation*}
\Prob{\omega_k}{\beta}{\tau_{M+2l}=H_{\Delta^M_k}<H^+_{\omega_k}}\leq\binom{M+2l}{l} \left[4e^{-\beta}\right]^{M+l-1}.
\end{equation*}
Using, for $1\leq l\leq M$, the bound $\binom{M+2l}{l}\leq (M+2l)^l\leq (3M)^l$, and, for $l>M$, the universal bound  $\binom{M+2l}{l}\leq 2^{M+2l}$, we obtain that
\begin{equation}\label{9,2}
\begin{split}
& \sum_{l=1}^\infty\Prob{\omega_k}{\beta}{\tau_{M+2l}=H_{\Delta^M_k}=H_{\zeta^k_{\alpha,i}}<H^+_{\omega_k}}\\
&\quad\quad\leq C_0[4e^{-\beta}]^M\left(M\sum_{l=1}^{M}[12Me^{-\beta}]^{l-1}+2^M\sum_{l=M}^\infty[16e^{-\beta}]^{l-1}\right)\leq C_0M(4e^{-\beta})^M,\\
\end{split}
\end{equation}
for $\beta$ large enough, in view of \eqref{Ne-betato0}.

Let us focus now in the term corresponding to $l=0$, which computes the probability of the trajectories from $\omega_k$ to $\zeta^k_{\alpha,i}$ with exactly $M$ jumps. Consider first the case $1\leq i \leq M-1$. Without loss of generality, let us suppose that $\alpha=B$. The configuration $\zeta^k_{B,i}$ is obtained from $\omega_k$ when a particle of type $A$ meets a particle of type $C$ in the region of particles of type $B$, in such a way that the particle $A$ has done $i$ jumps, and the particle $C$ has done $M-i$ jumps. For $j\in\{1,\ldots,i\}$ denote by $\mc A_j$ the event in which the first $j$ jumps are made by the particle $A$ and the $(j+1)$-th jump is made by the particle $C$. For $r\in\{1,\ldots,M-i\}$, define $\mc C_r$ in a analogous way. Then
\begin{eqnarray*}
\Prob{\omega_k}{\beta}{\tau_M=H_{\Delta^M_k}=H_{\zeta^k_{\alpha,i}}<H^+_{\omega_k}}\!   &\!=\!&\!\sum_{j=1}^i\Prob{\omega_k}{\beta}{\tau_M=H_{\Delta^M_k}=H_{\zeta^k_{\alpha,i}}<H^+_{\omega_k},\mathcal{A}_j} \\
   &&\!\!+\sum_{r=1}^{M-i}\Prob{\omega_k}{\beta}{\tau_M=H_{\Delta^M_k}=H_{\zeta^k_{\alpha,i}}<H^+_{\omega_k},\mathcal{C}_r}.
\end{eqnarray*}
Now note that there are $\binom{M-j-1}{i-j}$ possible paths of size $M$ corresponding to the event $[\tau_M=H_{\Delta^M_k}=H_{\zeta^k_{\alpha,i}}<H^+_{\omega_k},\mathcal{A}_j]$. In each of these paths, the first jump has probability $1/3$, the next $j$ jumps have probability $e^{-\beta}/(1+4e^{-\beta})$ and the next $M-j-1$ jumps have probability $e^{-\beta}/(2+5e^{-\beta})$. Figure \ref{fig3} illustrates this situation in a particular example.
\definecolor{yqyqyq}{rgb}{0.5,0.5,0.5}
\begin{figure}[h!]
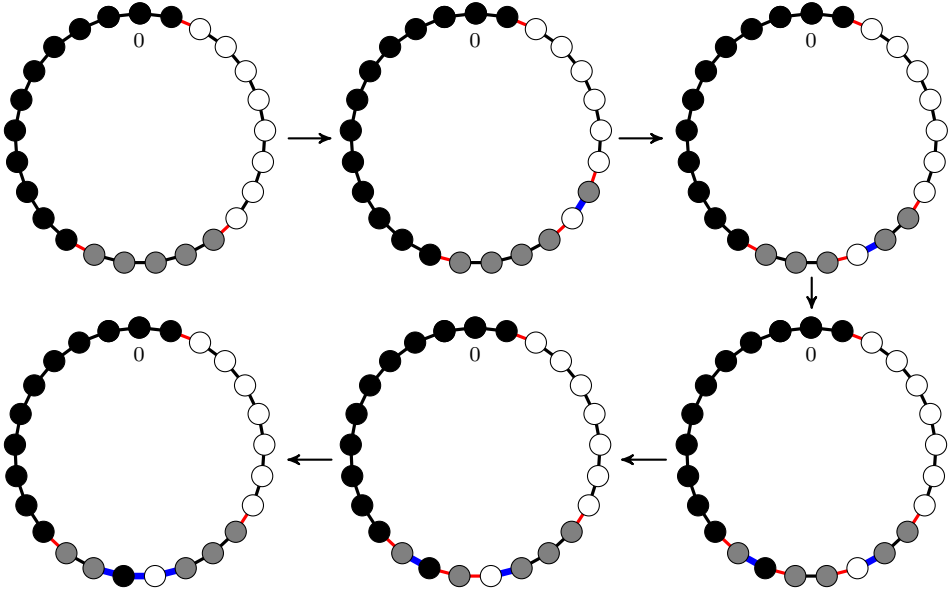

\centering

\caption{One of the two paths from $\omega_2$ to $\xi^2_{B,3}$ that correspond to the event $\mc A_2$. In this example, $N_A=8$, $N_B=5$, $N_C=12$. The white, gray and black circles represent respectively particles of type $A$, $B$ and $C$.}
\label{fig3}
\end{figure}
Doing the same analysis for the trajectories corresponding to the events $\mc C_r$, we conclude that
\begin{equation*}
\begin{split}
&\Prob{\omega_k}{\beta}{\tau_M=H_{\Delta^M_k}=H_{\zeta^k_{\alpha,i}}<H^+_{\omega_k}}\\
&\quad\quad\quad\quad\quad=~ \sum_{j=1}^{i}\binom{M-j-1}{i-j}\frac{1}{3}\left(\frac{e^{-\beta}}{1+4e^{-\beta}}\right)^j\left(\frac{e^{-\beta}}{2+5e^{-\beta}}\right)^{M-j-1}\\
&\quad\quad\quad\quad\quad\quad+\sum_{r=1}^{M-i}\binom{M-r-1}{M-i-r}\frac{1}{3}\left(\frac{e^{-\beta}}{1+4e^{-\beta}}\right)^r\left(\frac{e^{-\beta}}{2+5e^{-\beta}}\right)^{M-r-1}.
\end{split}
\end{equation*}
This can be rewritten as
\begin{eqnarray*}
 && \frac{1}{3}e^{-(M-1)\beta}\left(1+\mc O(Me^{-\beta})\right) \\
   &&\quad\quad\times\left[\sum_{j=1}^{i}\binom{M-j-1}{i-j}\left(\frac{1}{2}\right)^{M-j-1}+\sum_{r=1}^{M-i}\binom{M-r-1}{M-i-r}\left(\frac{1}{2}\right)^{M-r-1}\right].
\end{eqnarray*}
And then, by Lemma \ref{comblemma}, for $i\in\{1,\ldots,M-1\}$
\begin{equation}\label{9,3}
\Prob{\omega_k}{\beta}{\tau_M=H_{\Delta^M_k}=H_{\zeta^k_{\alpha,i}}<H^+_{\omega_k}}=\frac{2}{3}e^{-(M-1)\beta}+\mc O(Me^{-M\beta}).
\end{equation}
In the cases $i=0$ or $i=M$, there is a unique path of size $M$ from $\omega_k$ to $\zeta^k_{\alpha,i}$, which has probability equal to 
\begin{equation*}
\frac{1}{3}\left(\frac{e^{-\beta}}{1+4e^{-\beta}}\right)^{M-1}.
\end{equation*}
Therefore, for $i=0$ or $i=M$
\begin{equation}\label{9,4}
\Prob{\omega_k}{\beta}{\tau_M=H_{\Delta^M_k}=H_{\zeta^k_{\alpha,i}}<H^+_{\omega_k}}=\frac{1}{3}e^{-(M-1)\beta}+\mc O(Me^{-M\beta}).
\end{equation}
Using \eqref{9,1}, \eqref{9,2}, \eqref{9,3} and $\eqref{9,4}$ we obtain \eqref{9}.
\end{proof}

Recall the definitions of $d$ in \eqref{d} and of $\mc G^{N,k}_\alpha$ in \eqref{GNk} and let $R^{\beta}_1(\cdot,\cdot)$ denote the jump rates of $\{\eta_1(t):t\geq0\}$, the trace process on $\Omega^N_1$. 

\begin{proposition}\label{taxadosabsorventes} Assume \eqref{Ne-betato0}. There exist finite constants $C_0$ and $\beta_0$ such that for all $\beta>\beta_0$, $k\in\Lambda_N$ and $\xi\in\Omega^N_1$,
\begin{equation}\label{10}
\left|R^\beta_1(\omega_k,\xi)-\mb R_1(\omega_k,\xi)e^{-M\beta}\right|\leq C_0 N4^Me^{-(M+1)\beta},
\end{equation}
where
\begin{equation*}
\mb R_1(\omega_k,\xi)=\begin{cases}
\frac{d}{2} & \text{if $\xi\in\{\omega_{k-1},\omega_{k+1}\}$}, \\
\frac{2}{3} & \text{if $\xi\in\bigcup_{\alpha: N_\alpha=M}\mc G^{N,k}_\alpha$},\\
0           & \text{if $\xi\notin\bigcup_{\alpha: N_\alpha=M}\mc G^{N,k}_\alpha\cup\{\omega_{k-1},\omega_{k+1}\}$}.
\end{cases}
\end{equation*}
\end{proposition}

\begin{proof} By \cite[Proposition 6.1]{bl2}, for every $\omega,\xi\in\Omega^N_1$,
\begin{equation}\label{11}
    R^\beta_1(\omega,\xi)=\lambda_\beta(\omega)\Prob{\omega}{\beta}{H^+_{\Omega^N_1}=H_\xi},
\end{equation}
where $\lambda_\beta(\omega)$ is the total jump rate from $\omega$ for the original chain $\{\eta(t):t\geq0\}$. A crucial observation is that, starting from $\omega_k$, to reach any other configuration in $\Omega^N_1$ the process has to cross $\Delta^M_k$. Thus, for every $\xi\in\Omega^N_1$, $\xi\neq\omega_k$, on the event $\{H^+_{\Omega^N_1}=H_\xi\}$
we have that $H_{\Delta^M_k}<H^+_{\Omega^N_1}$, and then by the strong Markov property,
\begin{equation}\label{12}
\Prob{\omega_k}{\beta}{H^+_{\Omega^N_1}=H_\xi}=\sum_{\omega\in\Delta^M_k}\Prob{\omega_k}{\beta}{H_{\Delta^M_k}=H_\omega<H^+_{\omega_k}}\Prob{\omega}{\beta}{H_{\Omega^N_1}=H_\xi}.
\end{equation}
If $\xi\notin\bigcup_{\alpha: N_\alpha=M}\mc G^{N,k}_\alpha\cup\{\omega_{k-1},\omega_{k+1}\}$ then, by Lemma \ref{lemma2},
\begin{equation*} 
\Prob{\omega}{\beta}{H_{\Omega^N_1}=H_\xi}\leq C_0Ne^{-\beta},
\end{equation*} 
for every $\omega\in\Delta^M_k$. So
\begin{eqnarray*}
\Prob{\omega_k}{\beta}{H^+_{\Omega^N_1}=H_\xi}   &\leq&  C_0Ne^{-\beta}\sum_{\omega\in\Delta^M_k}\Prob{\omega_k}{\beta}{H_{\Delta^M_k}=H_\omega<H^+_{\omega_k}}\\
   &=& C_0Ne^{-\beta}\Prob{\omega_k}{\beta}{H_{\Delta^M_k}<H^+_{\omega_k}}\leq C_0N \left(4e^{-\beta}\right)^M
\end{eqnarray*}
by the first part of Corollary \ref{corol1}. So, by \eqref{11} and the fact that $\lambda_\beta(\omega_k)=3e^{-\beta}$ we get \eqref{10} for $\xi\notin\bigcup_{\alpha: N_\alpha=M}\mc G^{N,k}_\alpha\cup\{\omega_{k-1},\omega_{k+1}\}$.

Now let us consider the case $\xi=\xi^k_{\alpha,i}$ for $N_\alpha=M$, $1\leq i \leq M-1$. By Lemma~\ref{lemma2}, for every $\omega\in\Delta^M_k$
\begin{equation*}
   \Prob{\omega}{\beta}{H_{\Omega^N_1}=H_{\xi^k_{\alpha,i}}}=\frac{1}{3}\mb1\left\{\omega=\zeta^k_{\alpha,i}\right\}+\mc O\left(Ne^{-\beta}\right).
\end{equation*}
Therefore, by \eqref{11}, \eqref{12}, and the first part of Corollary \ref{corol1},
\begin{equation*}
R^\beta_1\left(\omega_k,\xi^k_{\alpha,i}\right)=e^{-\beta}\Prob{\omega_k}{\beta}{H_{\Delta^M_k}=H_{\zeta^k_{\alpha,i}}<H^+_{\omega_k}}+ \mc O\left(N(4e^{-\beta})^{M+1}\right).
\end{equation*}
And then, by Lemma \ref{lemma4},
\begin{equation*}
R^\beta_1\left(\omega_k,\xi^k_{\alpha,i}\right)=\frac{2}{3}e^{-M\beta}+\mc O\left(N(4e^{-\beta})^{M+1}\right).
\end{equation*}

Let us make the same argument for the case $\xi=\omega_{k+1}$. By Lemma \ref{lemma2}, for every $\omega\in\Delta^M_k$
\begin{equation*}
   \Prob{\omega}{\beta}{H_{\Omega^N_1}=H_{\omega_{k+1}}}=\frac{1}{2}\mb1\left\{\omega\in\{\zeta^k_{\alpha,0}: N_\alpha=M\}\right\}+\mc O\left(Ne^{-\beta}\right).
\end{equation*}
Again, using \eqref{11}, \eqref{12}, and the first part of Corollary \ref{corol1}, we obtain that
\begin{equation*}
R^\beta_1\left(\omega_k,\omega_{k+1}\right)=\frac{3e^{-\beta}}{2}\sum_{\alpha: N_\alpha=M}\Prob{\omega_k}{\beta}{H_{\Delta^M_k}=H_{\zeta^k_{\alpha,0}}<H^+_{\omega_k}}+ \mc O\left(N(4e^{-\beta})^{M+1}\right).
\end{equation*}
And therefore, by Lemma \ref{lemma4},
\begin{equation*}
R^\beta_1\left(\omega_k,\omega_{k+1}\right)=\frac{d}{2}e^{-M\beta}+\mc O\left(N(4e^{-\beta})^{M+1}\right).
\end{equation*}
The case $\xi=\omega_{k-1}$ is analogous.
\end{proof}

The proposition we just proved estimates the jump rates of the trace process $\{\eta_1(t):t\geq0\}$ from the configurations in $\Omega^N_0$. Now we want to estimate the jump rates from the configurations in $\mc G^N$. Arguing  as in this last proof, for each configuration in $\mc G^N$ as initial distribution, we will need to compute the distribution of the process in the first return to $\Omega^N_1$. If we allow errors of order $Ne^{-\beta}$, these hitting probabilities are easily obtained.

From any configuration $\xi^k_{\alpha,i}\in\mc G^N$, as illustrated in Figure~\ref{fig2}, there are six possibilities for the first jump. Each of these six configurations is associated to one of the six red edges of the configuration $\xi^k_{\alpha,i}$. This association provides a way to label these configurations. We will denote these configurations by $\xi^{k,j}_{\alpha,i}$, for $j=1,\ldots,6$. We do this in such a way that, if we enumerate the red edges of $\xi^k_{\alpha,i}$ as $r_1,\ldots,r_6$ clockwise, then $\xi^{k,j}_{\alpha,i}$ is the configuration obtained from $\xi^k_{\alpha,i}$ after a transposition in~$r_j$. To fix a initial point for the enumeration of the red edges, we impose that this is done in such a way that $\xi^{k,2}_{\alpha,i}=\zeta^k_{\alpha,i}$.
Note that, with this convention we have that $\xi^{k,j}_{\alpha,N_\alpha-1}=\xi^{k-1,j+1}_{\alpha+1,1}$ and, for $1\leq i \leq N_\alpha-2$, $\xi^{k,3}_{\alpha,i}=\xi^{k,1}_{\alpha,i+1}$. See Figures \ref{fig4} and~\ref{fig5} for an example.

For each $\omega\in R(\xi^k_{\alpha,i})$, $k\in\Lambda_N$, $\alpha\in\{A,B,C\}$, $1\leq i \leq N_\alpha-1$, consider the probability measure $\Phi^N(\omega,\cdot)$ defined on $\Omega^N_1$ in the following way:
\begin{enumerate}
  \item[(i)] If $2\leq i \leq N_\alpha-2$,
\begin{equation*}
    \Phi^N\left(\xi^{k,j}_{\alpha,i},\{\xi^k_{\alpha,i}\}\right)=1, \text{ for } j=4,5,6;
\end{equation*}
\begin{equation*}
\Phi^N\left(\xi^{k,j}_{\alpha,i},\Pi\right)=\begin{cases}
\frac{1}{2} & \text{if $\Pi=\{\xi^k_{\alpha,i}\}$}, \\
\frac{1}{2} & \text{if $\Pi=\{\xi^k_{\alpha,i-2+j}\}$}
\end{cases}, \text{ for } j=1,3;
\end{equation*}
and
\begin{equation*}
    \Phi^N\left(\xi^{k,2}_{\alpha,i},\Pi\right)=\begin{cases}
\frac{1}{3} & \text{if $\Pi=\{\xi^k_{\alpha,i}\}$}, \\
\frac{2}{3} & \text{if $\Pi=\{\omega_k\}$}.
\end{cases}
\end{equation*}
  \item[(ii)] And if $i=1$,
\begin{equation*}
\Phi^N\left(\xi^{k,j}_{\alpha,1},\Pi\right)=\begin{cases}
\frac{1}{3} & \text{if $\Pi=\{\xi^k_{\alpha,1}\}$}, \\
\frac{2}{3} & \text{if $\Pi=\{\omega_{k+2-j}\}$}
\end{cases}, \text{ for } j=1,2;
\end{equation*}
\begin{equation*}
\!\!\!\Phi^N\left(\xi^{k,3}_{\alpha,1},\Pi\right)=\begin{cases}
\frac{1}{2} & \text{if $\Pi=\{\xi^k_{\alpha,1}\}$}, \\
\frac{1}{2} & \text{if $\Pi=\{\xi^k_{\alpha,2}\}$}
\end{cases}, ~~
\Phi^N\left(\xi^{k,6}_{\alpha,1},\Pi\right)=\begin{cases}
\frac{1}{2} & \text{if $\Pi=\{\xi^k_{\alpha,1}\}$}, \\
\frac{1}{2} & \text{if $\Pi=\{\xi^{k+1}_{\alpha-1,N_{\alpha-1}-2}\}$}
\end{cases}
\end{equation*}
and
\begin{equation*}
    \Phi^N\left(\xi^{k,j}_{\alpha,1},\{\xi^k_{\alpha,i}\}\right)=1, \text{ for } j=4,5.
\end{equation*}
\end{enumerate}

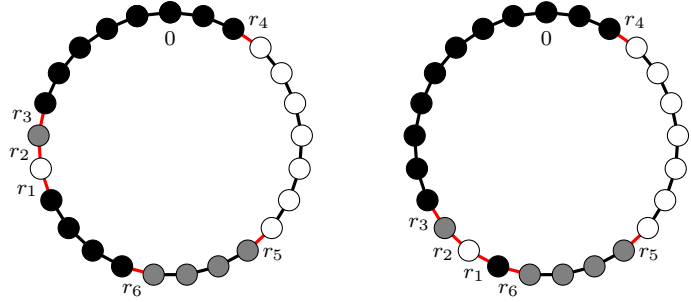
\begin{figure}[h!]
\centering
\definecolor{ffqqff}{rgb}{1,0,1}
\definecolor{qqffqq}{rgb}{0,1,0}
\definecolor{yqyqyq}{rgb}{0.5,0.5,0.5}
\begin{tikzpicture}[line cap=round,line join=round,>=triangle 45,x=1.0cm,y=1.0cm]
\clip(-2.18,-2.16) rectangle (7.04,2.08);
\draw [line width=1.2pt] (0,1.75)-- (0.44,1.7);
\draw [line width=1.2pt] (0.44,1.7)-- (0.84,1.53);
\draw [line width=1.2pt,color=ffqqqq] (0.84,1.53)-- (1.2,1.28);
\draw [line width=1.2pt] (1.2,1.28)-- (1.48,0.94);
\draw [line width=1.2pt] (1.48,0.94)-- (1.66,0.54);
\draw [line width=1.2pt] (1.66,0.54)-- (1.75,0.11);
\draw [line width=1.2pt] (1.75,0.11)-- (1.72,-0.33);
\draw [line width=1.2pt] (1.72,-0.33)-- (1.58,-0.75);
\draw [line width=1.2pt] (1.58,-0.75)-- (1.35,-1.12);
\draw [line width=1.2pt,color=ffqqqq] (1.35,-1.12)-- (1.03,-1.42);
\draw [line width=1.2pt] (1.03,-1.42)-- (0.64,-1.63);
\draw [line width=1.2pt] (0.64,-1.63)-- (0.22,-1.74);
\draw [line width=1.2pt] (0.22,-1.74)-- (-0.22,-1.74);
\draw [line width=1.2pt,color=ffqqqq] (-0.22,-1.74)-- (-0.64,-1.63);
\draw [line width=1.2pt] (-0.64,-1.63)-- (-1.03,-1.42);
\draw [line width=1.2pt] (-1.03,-1.42)-- (-1.35,-1.12);
\draw [line width=1.2pt] (-1.35,-1.12)-- (-1.58,-0.75);
\draw [line width=1.2pt,color=ffqqqq] (-1.58,-0.75)-- (-1.72,-0.33);
\draw [line width=1.2pt,color=ffqqqq] (-1.72,-0.33)-- (-1.75,0.11);
\draw [line width=1.2pt,color=ffqqqq] (-1.75,0.11)-- (-1.66,0.54);
\draw [line width=1.2pt] (-1.66,0.54)-- (-1.48,0.94);
\draw [line width=1.2pt] (-1.48,0.94)-- (-1.2,1.28);
\draw [line width=1.2pt] (-1.2,1.28)-- (-0.84,1.53);
\draw [line width=1.2pt] (-0.84,1.53)-- (-0.44,1.7);
\draw [line width=1.2pt] (-0.44,1.7)-- (0,1.75);
\draw [line width=1.2pt] (5,1.75)-- (5.44,1.7);
\draw [line width=1.2pt] (5.44,1.7)-- (5.84,1.53);
\draw [line width=1.2pt,color=ffqqqq] (5.84,1.53)-- (6.2,1.28);
\draw [line width=1.2pt] (6.2,1.28)-- (6.48,0.94);
\draw [line width=1.2pt] (6.48,0.94)-- (6.66,0.54);
\draw [line width=1.2pt] (6.66,0.54)-- (6.75,0.11);
\draw [line width=1.2pt] (6.75,0.11)-- (6.72,-0.33);
\draw [line width=1.2pt] (6.72,-0.33)-- (6.58,-0.75);
\draw [line width=1.2pt] (6.58,-0.75)-- (6.35,-1.12);
\draw [line width=1.2pt,color=ffqqqq] (6.35,-1.12)-- (6.03,-1.42);
\draw [line width=1.2pt] (6.03,-1.42)-- (5.64,-1.63);
\draw [line width=1.2pt] (5.64,-1.63)-- (5.22,-1.74);
\draw [line width=1.2pt] (5.22,-1.74)-- (4.78,-1.74);
\draw [line width=1.2pt,color=ffqqqq] (4.78,-1.74)-- (4.36,-1.63);
\draw [line width=1.2pt,color=ffqqqq] (4.36,-1.63)-- (3.97,-1.42);
\draw [line width=1.2pt,color=ffqqqq] (3.97,-1.42)-- (3.65,-1.12);
\draw [line width=1.2pt,color=ffqqqq] (3.65,-1.12)-- (3.42,-0.75);
\draw [line width=1.2pt] (3.42,-0.75)-- (3.28,-0.33);
\draw [line width=1.2pt] (3.28,-0.33)-- (3.25,0.11);
\draw [line width=1.2pt] (3.25,0.11)-- (3.34,0.54);
\draw [line width=1.2pt] (3.34,0.54)-- (3.52,0.94);
\draw [line width=1.2pt] (3.52,0.94)-- (3.8,1.28);
\draw [line width=1.2pt] (3.8,1.28)-- (4.16,1.53);
\draw [line width=1.2pt] (4.16,1.53)-- (4.56,1.7);
\draw [line width=1.2pt] (4.56,1.7)-- (5,1.75);
\begin{scriptsize}
\fill [color=black] (0,1.75) circle (4.0pt);
\fill [color=black] (0.44,1.7) circle (4.0pt);
\fill [color=black] (0.84,1.53) circle (4.0pt);
\fill [color=ffffff] (1.2,1.28) circle (4.0pt);
\fill [color=ffffff] (1.48,0.94) circle (4.0pt);
\fill [color=ffffff] (1.66,0.54) circle (4.0pt);
\fill [color=ffffff] (1.75,0.11) circle (4.0pt);
\fill [color=ffffff] (1.72,-0.33) circle (4.0pt);
\fill [color=ffffff] (1.58,-0.75) circle (4.0pt);
\fill [color=ffffff] (1.35,-1.12) circle (4.0pt);
\fill [color=xfxfxf] (1.03,-1.42) circle (4.0pt);
\fill [color=xfxfxf] (0.64,-1.63) circle (4.0pt);
\fill [color=xfxfxf] (0.22,-1.74) circle (4.0pt);
\fill [color=xfxfxf] (-0.22,-1.74) circle (4.0pt);
\fill [color=black] (-0.64,-1.63) circle (4.0pt);
\fill [color=black] (-1.03,-1.42) circle (4.0pt);
\fill [color=black] (-1.35,-1.12) circle (4.0pt);
\fill [color=black] (-1.58,-0.75) circle (4.0pt);
\fill [color=ffffff] (-1.72,-0.33) circle (4.0pt);
\fill [color=xfxfxf] (-1.75,0.11) circle (4.0pt);
\fill [color=black] (-1.66,0.54) circle (4.0pt);
\fill [color=black] (-1.48,0.94) circle (4.0pt);
\fill [color=black] (-1.2,1.28) circle (4.0pt);
\fill [color=black] (-0.84,1.53) circle (4.0pt);
\fill [color=black] (-0.44,1.7) circle (4.0pt);
\draw [color=black] (0,1.75) circle (4.0pt);
\draw [color=black] (0.44,1.7) circle (4.0pt);
\draw [color=black] (0.84,1.53) circle (4.0pt);
\draw [color=black] (1.2,1.28) circle (4.0pt);
\draw [color=black] (1.75,0.11) circle (4.0pt);
\draw [color=black] (1.72,-0.33) circle (4.0pt);
\draw [color=black] (1.58,-0.75) circle (4.0pt);
\draw [color=black] (1.35,-1.12) circle (4.0pt);
\draw [color=black] (1.03,-1.42) circle (4.0pt);
\draw [color=black] (1.66,0.54) circle (4.0pt);
\draw [color=black] (1.48,0.94) circle (4.0pt);
\draw [color=black] (0.64,-1.63) circle (4.0pt);
\draw [color=black] (0.22,-1.74) circle (4.0pt);
\draw [color=black] (-0.22,-1.74) circle (4.0pt);
\draw [color=black] (-0.64,-1.63) circle (4.0pt);
\draw [color=black] (-1.03,-1.42) circle (4.0pt);
\draw [color=black] (-1.35,-1.12) circle (4.0pt);
\draw [color=black] (-1.58,-0.75) circle (4.0pt);
\draw [color=black] (-1.72,-0.33) circle (4.0pt);
\draw [color=black] (-1.75,0.11) circle (4.0pt);
\draw [color=black] (-1.66,0.54) circle (4.0pt);
\draw [color=black] (-1.48,0.94) circle (4.0pt);
\draw [color=black] (-1.2,1.28) circle (4.0pt);
\draw [color=black] (-0.84,1.53) circle (4.0pt);
\draw [color=black] (-0.44,1.7) circle (4.0pt);
\fill [color=black] (0,1.75) circle (4.0pt);
\draw [color=black] (0,1.75) circle (4.0pt);
\fill [color=black] (-0.44,1.7) circle (4.0pt);
\draw [color=black] (-0.44,1.7) circle (4.0pt);
\fill [color=black] (5,1.75) circle (4.0pt);
\fill [color=black] (5.44,1.7) circle (4.0pt);
\fill [color=black] (5.84,1.53) circle (4.0pt);
\fill [color=ffffff] (6.2,1.28) circle (4.0pt);
\fill [color=ffffff] (6.48,0.94) circle (4.0pt);
\fill [color=ffffff] (6.66,0.54) circle (4.0pt);
\fill [color=ffffff] (6.75,0.11) circle (4.0pt);
\fill [color=ffffff] (6.72,-0.33) circle (4.0pt);
\fill [color=ffffff] (6.58,-0.75) circle (4.0pt);
\fill [color=ffffff] (6.35,-1.12) circle (4.0pt);
\fill [color=yqyqyq] (6.03,-1.42) circle (4.0pt);
\fill [color=xfxfxf] (5.64,-1.63) circle (4.0pt);
\fill [color=xfxfxf] (5.22,-1.74) circle (4.0pt);
\fill [color=xfxfxf] (4.78,-1.74) circle (4.0pt);
\fill [color=black] (4.36,-1.63) circle (4.0pt);
\fill [color=ffffff] (3.97,-1.42) circle (4.0pt);
\fill [color=yqyqyq] (3.65,-1.12) circle (4.0pt);
\fill [color=black] (3.42,-0.75) circle (4.0pt);
\fill [color=black] (3.28,-0.33) circle (4.0pt);
\fill [color=black] (3.25,0.11) circle (4.0pt);
\fill [color=black] (3.34,0.54) circle (4.0pt);
\fill [color=black] (3.52,0.94) circle (4.0pt);
\fill [color=black] (3.8,1.28) circle (4.0pt);
\fill [color=black] (4.16,1.53) circle (4.0pt);
\fill [color=black] (4.56,1.7) circle (4.0pt);
\draw [color=black] (5,1.75) circle (4.0pt);
\draw [color=black] (5.44,1.7) circle (4.0pt);
\draw [color=black] (5.84,1.53) circle (4.0pt);
\draw [color=black] (6.2,1.28) circle (4.0pt);
\draw [color=black] (6.48,0.94) circle (4.0pt);
\draw [color=black] (6.66,0.54) circle (4.0pt);
\draw [color=black] (6.75,0.11) circle (4.0pt);
\draw [color=black] (6.72,-0.33) circle (4.0pt);
\draw [color=black] (6.58,-0.75) circle (4.0pt);
\draw [color=black] (6.35,-1.12) circle (4.0pt);
\draw [color=black] (6.03,-1.42) circle (4.0pt);
\draw [color=black] (5.64,-1.63) circle (4.0pt);
\draw [color=black] (5.22,-1.74) circle (4.0pt);
\draw [color=black] (4.78,-1.74) circle (4.0pt);
\draw [color=black] (4.36,-1.63) circle (4.0pt);
\draw [color=black] (3.97,-1.42) circle (4.0pt);
\draw [color=black] (3.65,-1.12) circle (4.0pt);
\draw [color=black] (3.42,-0.75) circle (4.0pt);
\draw [color=black] (3.28,-0.33) circle (4.0pt);
\draw [color=black] (3.25,0.11) circle (4.0pt);
\draw [color=black] (3.34,0.54) circle (4.0pt);
\draw [color=black] (3.52,0.94) circle (4.0pt);
\draw [color=black] (3.8,1.28) circle (4.0pt);
\draw [color=black] (4.16,1.53) circle (4.0pt);
\draw [color=black] (4.56,1.7) circle (4.0pt);
\draw (0,1.4) node {\footnotesize $0$};
\draw (5,1.4) node {\footnotesize $0$};
\draw (-1.9,-0.62) node {\footnotesize $r_1$};
\draw (-2,-0.13) node {\footnotesize $r_2$};
\draw (-1.96,0.37) node {\footnotesize $r_3$};
\draw (1.18,1.62) node {\footnotesize $r_4$};
\draw (1.37,-1.46) node {\footnotesize $r_5$};
\draw (-0.5,-1.94) node {\footnotesize $r_6$};
\draw (4.04,-1.75) node {\footnotesize $r_1$};
\draw (3.63,-1.46) node {\footnotesize $r_2$};
\draw (3.31,-1.07) node {\footnotesize $r_3$};
\draw (6.18,1.62) node {\footnotesize $r_4$};
\draw (6.37,-1.46) node {\footnotesize $r_5$};
\draw (4.5,-1.94) node {\footnotesize $r_6$};
\end{scriptsize}
\end{tikzpicture}
\caption{The configurations $\xi^2_{C,4}$ and $\xi^2_{C,1}$ with the corresponding labels of the red edges. In this example $N_A=8$, $N_B=5$, $N_C=12$. The white, gray and black circles represents respectively particles of types $A$, $B$ and $C$.}
\label{fig4}
\end{figure}

\definecolor{yqyqyq}{rgb}{0.5,0.5,0.5}
\begin{figure}[h!]
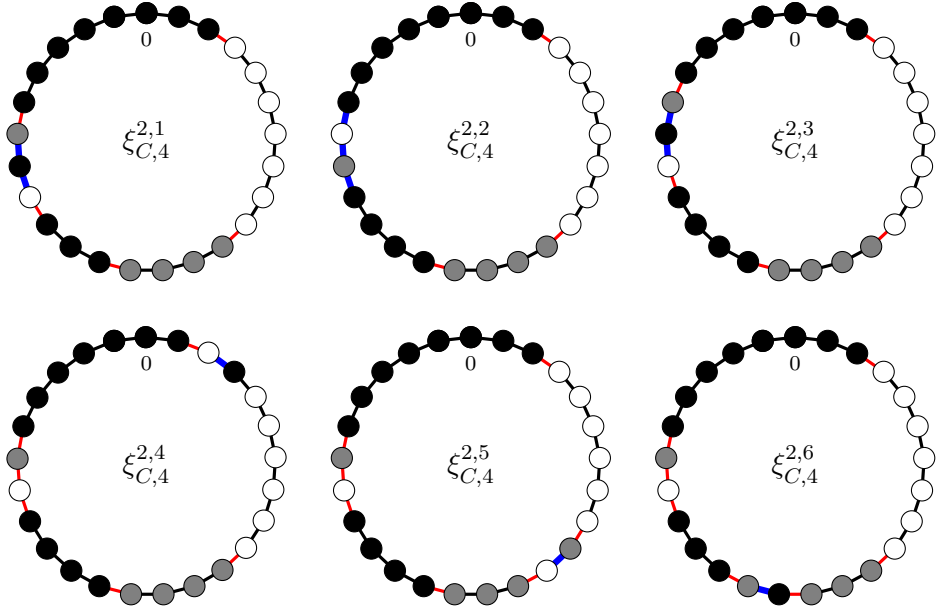

\centering
\definecolor{yqyqyq}{rgb}{0.5,0.5,0.5}
\definecolor{qqqqff}{rgb}{0,0,1}
\definecolor{ffqqqq}{rgb}{1,0,0}
\definecolor{xfxfxf}{rgb}{0.5,0.5,0.5}
\definecolor{ffffff}{rgb}{1,1,1}

\caption{The six configurations that can be reached after one jump from the configuration $\xi^{2}_{C,4}$, which is illustrated in Figure~\ref{fig4}.}
\label{fig5}
\end{figure}

\begin{lemma}\label{vizinhosdexi} There exists a constant $C_0$, such that, for any $\beta>0$, $k\in\Lambda_N$, $\alpha\in\{A,B,C\}$, $1\leq i \leq N_\alpha-1$ and $\omega\in R(\xi^k_{\alpha,i})$
\begin{equation}\label{29}
    \left|\Prob{\omega}{\beta}{\eta(H_{\Omega^N_1})\in\Pi}-\Phi^N(\omega,\Pi)\right|\leq C_0Ne^{-\beta}.
\end{equation}
\end{lemma}

\begin{proof} We present the proof for the case $2\leq i \leq N_\alpha-2$, the extreme case $i=1$ is similar and the verification is left to the reader.

As observed in the proof of Lemma \ref{lemma1}, a transposition in one of the red edges of $\xi^k_{\alpha,i}$ only changes the color of this edge, which always becomes a blue edge, and of the two that are adjacent to it, for which black becomes red, and red becomes blue.

In the configuration $\xi^k_{\alpha,i}$, for $2\leq i \leq N_\alpha-2$, as illustrated in Figure~\ref{fig4}, the edges adjacent to $r_4$, $r_5$ and $r_6$ are all black.
Then, by the above observation, as illustrated in Figure~\ref{fig5}, for $j=4,5,6$, we have that $|R(\xi^{k,j}_{\alpha,i})|=7$ and $B(\xi^{k,j}_{\alpha,i})=\{\xi^k_{\alpha,i}\}$. Therefore, if $\tau_1$ is the instant of the first jump of the chain, for $j=4,5,6$,
\begin{equation*}
\Prob{\xi^{k,j}_{\alpha,i}}{\beta}{H_{\Omega^N_1}=H_{\xi^k_{\alpha,i}}}\geq\Prob{\xi^{k,j}_{\alpha,i}}{\beta}{\eta(\tau_1)=\xi^k_{\alpha,i}}=\frac{1}{1+7e^{-\beta}}=1+\mc O(e^{-\beta}).
\end{equation*}
This proves \eqref{29} for $\omega=\xi^{k,j}_{\alpha,i}$, $j=4,5,6$.

Doing the same kind of analysis for the edge $r_1$ of the configuration $\xi^k_{\alpha,i}$, we note that $|R(\xi^{k,1}_{\alpha,i})|=5$ and $|B(\xi^{k,1}_{\alpha,i})|=2$. In fact, we can see that
\begin{equation*}
B(\xi^{k,1}_{\alpha,i})=\left\{\xi^k_{\alpha,i},\xi^k_{\alpha,i-1}\right\}.
\end{equation*}
Therefore, for $l=i-1,i$
\begin{equation*}
\Prob{\xi^{k,1}_{\alpha,i}}{\beta}{H_{\Omega^N_1}=H_{\xi^k_{\alpha,l}}}\geq\Prob{\xi^{k,1}_{\alpha,i}}{\beta}{\eta(\tau_1)=\xi^k_{\alpha,l}}=\frac{1}{2+5e^{-\beta}}=\frac{1}{2}+\mc O(e^{-\beta}).
\end{equation*}
This proves \eqref{29} for $\omega=\xi^{k,1}_{\alpha,i}$, and the proof for $\omega=\xi^{k,3}_{\alpha,i}$ is analogous.

Now let us consider the case $\omega=\xi^{k,2}_{\alpha,i}=\zeta^k_{\alpha,i}\in V(\omega_k)\cap\Delta^{N_\alpha}_k$. The case $N_\alpha=M$ was already considered in Lemma~\ref{lemma2}. In that proof a better approximation (error of order $e^{-\beta}$ instead of $Ne^{-\beta}$) was given using the second part of Corollary~\ref{corol1}. In the general case we need to argue differently. We have that $|B(\zeta^k_{\alpha,i})|=|R(\zeta^k_{\alpha,i})|=3$. One of the three configurations in $B(\zeta^k_{\alpha,i})$ is $\xi^k_{\alpha,i}$. So
\begin{equation}\label{30}
 \Prob{\zeta^k_{\alpha,i}}{\beta}{H_{\Omega^N_1}=H_{\xi^k_{\alpha,i}}}\geq\Prob{\zeta^k_{\alpha,i}}{\beta}{\eta(\tau_1)=\xi^k_{\alpha,i}}=\frac{1}{3+3e^{-\beta}}=\frac{1}{3}+\mc O \left(e^{-\beta}\right).
\end{equation}
From the other two configurations in $B(\zeta^k_{\alpha_i})$, the configuration $\omega_k$ is the only one in $\Omega^N_1$ that can be reached after a sequence of rate~$1$ jumps. Moreover, starting from some of the two configurations in $B(\zeta^k_{\alpha,i})\setminus\{\xi^k_{\alpha,i}\}=B^*_k(\zeta^k_{\alpha,i})$, if the next $N_\alpha-1$ jumps of the chain correspond to transpositions in blue edges of the configurations, then, after these jumps, the process will ultimately fall in $\omega_k$. This gives
\begin{equation}\label{31}
    \Prob{\zeta^k_{\alpha,i}}{\beta}{H_{\Omega^N_1}=H_{\omega_k}}\geq\frac{1}{3+3e^{-\beta}}\sum_{\omega\in B^*_k(\zeta^k_{\alpha,i})}\Prob{\omega}{\beta}{\bigcap_{l=1}^{N_\alpha-1}\left\{\eta(\tau_l)\in B(\eta(\tau_{l-1}))\right\}}.
\end{equation}
In this situation, transpositions in blue edges corresponds to movements of the detached particle of type $\alpha+1$ in clockwise direction inside the domain of particles of type $\alpha$ or movements of the detached particle of type $\alpha-1$ in counterclockwise direction. Therefore, any configuration $\omega$ that can be reached from $B^*_k(\zeta^k_{\alpha,i})$ after a sequence of transposition in blue edges belongs to $V(\omega_k)$. In fact, to arrive in $\omega_k$, we just need to keep moving the detached particles in the correct direction until they meet the corresponding region of particles of their type. So the inequality of Lemma~\ref{lemma1} holds for these configurations. And then, applying inductively the strong Markov property in~\eqref{31} we see that
\begin{equation}\label{32}
    \Prob{\zeta^k_{\alpha,i}}{\beta}{H_{\Omega^N_1}=H_{\omega_k}}\geq\frac{2}{3+3e^{-\beta}}\left(\frac{1}{1+4e^{-\beta}}\right)^{N_\alpha-1}=\frac{2}{3}+\mc O \left(Ne^{-\beta}\right).
\end{equation}
Inequalities \eqref{30} and \eqref{32} prove \eqref{29} for the case $\omega=\xi^{k,2}_{\alpha,i}$. This concludes the proof of the lemma.
\end{proof}

\begin{proposition}\label{taxadostrasientes} There exists a finite constant $C_0$ such that for any $\beta>0$, $k\in\Lambda_N$, $\alpha\in\{A,B,C\}$ and $\omega\in\Omega^N_1$,
\begin{equation*}
    \left|R^\beta_1(\xi^k_{\alpha,i},\omega)-\mb R_1(\xi^k_{\alpha,i},\omega)e^{-\beta}\right|\leq C_0Ne^{-2\beta},
\end{equation*}
where, for $2\leq i\leq N_\alpha-2$
\begin{equation*}
\mb R_1(\xi^k_{\alpha,i},\omega)=\begin{cases}
\frac{1}{2} & \text{if $\omega\in\{\xi^k_{\alpha,i-1},\xi^k_{\alpha,i+1}\}$}, \\
\frac{2}{3} & \text{if $\omega=\omega_k$},\\
0           & \text{if $\omega\notin\{\xi^k_{\alpha,i-1},\xi^k_{\alpha,i+1},\omega_k\}$}.
\end{cases}
\end{equation*}
and for $i=1$ (recall that $\xi^k_{\alpha,1}=\xi^{k+1}_{\alpha-1,N_{(\alpha-1)}-1}$),
\begin{equation*}
\mb R_1(\xi^k_{\alpha,1},\omega)=\begin{cases}
\frac{1}{2} & \text{if $\omega\in\{\xi^k_{\alpha,2},\xi^{k+1}_{\alpha-1,N_{(\alpha-1)}-2}\}$}, \\
\frac{2}{3} & \text{if $\omega\in\{\omega_k,\omega_{k+1}\}$},\\
0           & \text{if $\omega\notin\{\xi^k_{\alpha,2},\xi^{k+1}_{\alpha-1,N_{(\alpha-1)}-2},\omega_k,\omega_{k+1}\}$}.
\end{cases}
\end{equation*}
\end{proposition}

\begin{proof} By \eqref{11},
\begin{equation*}
  R^\beta_1(\xi^k_{\alpha,i},\omega) = \lambda_\beta(\xi^k_{\alpha,i})\Prob{\xi^k_{\alpha,i}}{\beta}{H^+_{\Omega^N_1}=H_\omega} =  6e^{-\beta}\Prob{\xi^k_{\alpha,i}}{\beta}{H^+_{\Omega^N_1}=H_\omega}.
\end{equation*}
Now, conditioning in the first jump, and using Lemma \ref{vizinhosdexi} we get
\begin{equation*}
     R^\beta_1(\xi^k_{\alpha,i},\omega)=e^{-\beta}\sum_{j=1}^6\Prob{\xi^{k,j}_{\alpha,i}}{\beta}{H^+_{\Omega^N_1}=H_\omega}=\mb R_1(\xi^k_{\alpha,i},\omega)e^{-\beta}+ \mc O\left(Ne^{-2\beta}\right),
\end{equation*}
as desired.
\end{proof}

\section{Trace of $\{\eta(t):t\geq0\}$ on $\Omega^N_0$}\label{sectracoOmega0}

Knowing the jump rates for the trace process $\{\eta_1(t):t\geq0\}$ on the set $\Omega^N_1$, we can obtain the jump rates for the trace process on the set $\Omega^N_0$ computing, for each configuration $\omega\in\mc G^N$, the distribution of the first visited configuration in $\Omega^N_0$ for the process $\{\eta_1(t): t\geq0\}$ starting from $\omega$. We start replacing the original process $\{\eta_1(t):t\geq0\}$ by an ideal process $\{\widehat\eta_1(t):t\geq0\}$ for which these absorption probabilities are more easily obtained.

For $\alpha$ such that $N_\alpha=M$, define
\begin{equation*}
    \mc G^N_\alpha=\{\xi^k_{\gamma,j}: j=1,\ldots,N_\gamma-1; ~~\gamma+k=\alpha\}.
\end{equation*}
To understand why we defined $\mc G^N_\alpha$ this way note that, starting from some $\xi^0_{\alpha,i}$, the configurations $\xi^k_{\gamma,j}$ in $\mc G^N$ which may be visited by the process after times of order $e^\beta$ are those such that $\gamma+k=\alpha$. On the set $\Omega^N_0\cup\mc G^N_\alpha$ consider the continuous-time Markov chain $\{\widehat{\eta}_1(t): t\geq0\}$ with absorbing states $\Omega^N_0$ that jumps from $\xi^k_{\gamma,i}$ to $\omega$ with the ideal rates
\begin{equation*}
    \widehat R^\beta_1(\xi^k_{\gamma,i},\omega):=\mb R_1(\xi^k_{\gamma,i},\omega)e^{-\beta}
\end{equation*}
given in Proposition \ref{taxadostrasientes}. Note that the corresponding discrete-time jump chain depends only on $N_A$, $N_B$ and $N_C$, and not directly on $\beta$. Figure \ref{fig6} presents the graph structure of this simple dynamics.

For $\alpha$ such that $N_\alpha=M$, $1\leq i \leq M-1$ and $k\in\Lambda_N$, denote by $p^N_\alpha(i,k)$ the probability for the chain $\{\widehat{\eta}_1(t):t\geq0\}$ of, starting from $\xi^0_{\alpha,i}$ being absorbed in $\omega_k$.

\begin{lemma}\label{lemmaideal} There exists a constant $C_0$ such that, for any $\beta>0$, $k\in\Lambda_N$, $\alpha$ such that $N_\alpha=M$, and $1\leq i \leq M-1$
\begin{equation}\label{18}
    \left|\Prob{\xi^0_{\alpha,i}}{\beta}{H_{\Omega^N_0}=H_{\omega_k}}-p^N_\alpha(i,k)\right|\leq C_0 N^3\beta e^{-\beta}.
\end{equation}
\end{lemma}

\begin{proof} By Proposition \ref{taxadostrasientes}, there exists a constant $C_0$ such that
\begin{equation*}
\max_{\xi\in\mathcal{G}^N_\alpha\cup\Omega^N_0}\sum_{\omega\in\mathcal{G}^N_\alpha\cup\Omega^N_0}\left|R^\beta_1(\xi,\omega)-\widehat R^\beta_1(\xi,\omega)\right|\leq C_0N^3e^{-2\beta}.
\end{equation*}
Therefore, for all $1\leq i \leq M$, there exists a coupling $\overline{\mathbf{P}}^\beta$ of the two processes such that $\eta_1(0)=\widehat{\eta}_1(0)=\xi^0_{\alpha,i}$, $\overline{\mathbf{P}}^\beta$-a.s, and such that for all $t>0$,
\begin{equation}\label{19}
\overline{\mathbf{P}}^\beta\left[T_{cp}\leq t\right]\leq P\left[\mathcal{E}(C_0N^3e^{-2\beta})\leq t\right],
\end{equation}
where $T_{cp}=\inf\{t>0: \eta_1(t)\neq\widehat{\eta}_1(t)\}$ and $\mathcal{E}(C_0N^3e^{-2\beta})$ is a mean $(C_0N^3e^{-2\beta})^{-1}$ exponential random variable. To prove \eqref{18} it is sufficient to prove that
\begin{equation*}
\overline{\mathbf{P}}^\beta\left[\eta_1(H^{\eta_1}_{\Omega^N_0})\neq\widehat{\eta}_1( H^{\widehat{\eta}_1}_{\Omega^N_0})\right]\leq \widetilde{C}N^3\beta e^{-\beta},
\end{equation*}
for some universal constant $\widetilde{C}$, where $H^{\eta_1}_{\Omega^N_0}$ and $H^{\widehat{\eta}_1}_{\Omega^N_0}$ are the hitting times of $\Omega^N_0$ for this two processes. Let $T_\beta=\frac{3}{2}\beta e^\beta$.
\begin{eqnarray}
\nonumber  \overline{\mathbf{P}}^\beta\left[\eta_1(H^{\eta_1}_{\Omega^N_0})\neq \widehat\eta_1(H^{\widehat{\eta}_1}_{\Omega^N_0})\right] &=& \overline{\mathbf{P}}^\beta\left[\eta_1(H^{\eta_1}_{\Omega^N_0})\neq \widehat\eta_1(H^{\widehat{\eta}_1}_{\Omega^N_0}),~H^{\widehat{\eta}_1}_{\Omega^N_0}\leq T_\beta\right]\\
\nonumber& & + ~\overline{\mathbf{P}}^\beta\left[\eta_1(H^{\eta_1}_{\Omega^N_0})\neq \widehat\eta_1(H^{\widehat{\eta}_1}_{\Omega^N_0}),~H^{\widehat{\eta}_1}_{\Omega^N_0}>T_\beta\right]\\
\label{21}   &\leq& \overline{\mathbf{P}}^\beta\left[T_{cp}\leq T_\beta\right]+\overline{\mathbf{P}}^\beta\left[H^{\widehat{\eta}_1}_{\Omega^N_0}>T_\beta\right].
\end{eqnarray}
By the definition of the process $\{\widehat{\eta}_1(t):t\geq0\}$, the absorption time $H^{\widehat{\eta}_1}_{\Omega^N_0}$ is stochastically dominated by a mean $\left(\frac{2}{3}e^{-\beta}\right)^{-1}$ exponential random variable. Using this fact and \eqref{19} in \eqref{21}, we obtain that
\begin{eqnarray*}
  \overline{\mathbf{P}}^\beta\left[\eta_1(H^{\eta_1}_{\Omega^N_0})\neq \widehat\eta_1(H^{\widehat{\eta}_1}_{\Omega^N_0})\right] &\leq& P\left[\mathcal{E}(C_0N^3e^{-2\beta})\leq T_\beta\right]+P\left[\mathcal{E}\left(\frac{2}{3}e^{-\beta}\right)> T_\beta\right] \\
   &=& 1-\exp\left\{-C_0N^3e^{-2\beta}T_\beta\right\}+\exp\left\{-\frac{2}{3}e^{-\beta}T_\beta\right\}\\
   &\leq&C_0N^3e^{-2\beta}T_\beta+\exp\left\{-\frac{2}{3}e^{-\beta}T_\beta\right\}\\
   &\leq&\widetilde{C}N^3\beta e^{-\beta},
\end{eqnarray*}
by the definition of $T_\beta$.
\end{proof}

For $\alpha$ such that $N_\alpha=M$, and $k\neq0$ define
\begin{equation*}
g^N_\alpha(k)=\sum_{i=1}^{M-1}p^N_\alpha(i,k).
\end{equation*}

Recall that we have defined $r_\beta(k)$ as the jump rate from $\omega_0$ to $\omega_k$ for the trace process $\{\eta_0(t):t\geq0\}$ on $\Omega^N_0$. Next theorem expresses $r_\beta(k)$ in terms of $g^N_\alpha(k)$ with an error term $\psi(\beta)$, such that, as $\beta\uparrow\infty$, $e^{M\beta}\psi(\beta)$ vanishes, if we impose proper restrictions in the way that $N$~grows with~$\beta$.

\begin{proposition}\label{tracoemOmega0} Assume \eqref{Ne-betato0}. There exist constants $C_0$ and $\beta_0$ such that for any $\beta>\beta_0$ and $k\in\Lambda_N$, $k\neq0$,
\begin{equation}\label{27}
    \left|r_\beta(k)-r(N_A,N_B,N_C,k)e^{-M\beta}\right|\leq\psi(\beta)
\end{equation}
where
\begin{equation}\label{28}
r(N_A,N_B,N_C,k)=\frac{d}{2}\mb1\{|k|=1\}+\frac{2}{3}\sum_{\alpha: N_\alpha=M}g^N_\alpha(k)
\end{equation}
and $\psi$ is some function such that, for any $\beta>0$,

\begin{equation}\label{28,5}
    \psi(\beta)\leq C_0\left(N^24^M+N^3M\beta+ N^64^M\beta e^{-\beta}\right)e^{-(M+1)\beta}.
\end{equation}

\end{proposition}

\begin{proof} Noting that $\{\eta_0(t):t\geq0\}$ is also the trace of the process $\{\eta_1(t):t\geq0\}$, by \cite[Corollary 6.2]{bl2}
\begin{eqnarray}
\nonumber    r_\beta(k)&=&R^\beta_1(\omega_0,\omega_k)+\sum_{\alpha: N_\alpha=M}\sum_{i=1}^{M-1}R^\beta_1(\omega_0,\xi^0_{\alpha,i})\Prob{\xi^0_{\alpha,i}}{\beta}{H_{\Omega^N_0}=H_{\omega_k}} \\
  &&\label{24} + \sum_{\omega\in \mc G^N\setminus\bigcup_{\alpha: N_\alpha=M} \mathcal{G}^{N,0}_\alpha}R^\beta_1(\omega_0,\omega)\Prob{\omega}{\beta}{H_{\Omega^N_0}=H_{\omega_k}}.
\end{eqnarray}
By Proposition \ref{taxadosabsorventes} the first of the three terms above is equal to
\begin{equation}\label{22}
    \frac{d}{2}\mb 1\{|k|=1\}e^{-M\beta}+ \mc O\left(N4^M\right)e^{-(M+1)\beta}.
\end{equation}
A simple analysis, such as the one made to obtain \eqref{23} below, indicates that $p^N_\alpha(i,k)\leq\left(3/5\right)^{\min\{i-1,M-i-1\}}$ for every $k\neq0$, and then $g^N_\alpha(k)\leq C_0$, where $C_0$ is a constant independent of $M$. Using this observation, Proposition $\ref{taxadosabsorventes}$ and Lemma \ref{lemmaideal}, we obtain that the second term in \eqref{24} is equal to
\begin{equation}\label{25}
    \frac{2}{3}\sum_{\alpha:N_\alpha=M}g^N_\alpha(k)e^{-M\beta}+\mc O\left(N4^M+N^3M\beta+N^4M4^M\beta e^{-\beta}\right)e^{-(M+1)\beta}.
\end{equation}
Now let's look at the third term in \eqref{24}. For $\omega\in \mc G^N\setminus\bigcup_{\alpha: N_\alpha=M} \mathcal{G}^{N,0}_\alpha$, by Proposition \ref{taxadosabsorventes}, $R^\beta_1(\omega_0,\omega)\leq C_0 N4^Me^{-(M+1)\beta}$. Now observe that for $\omega\in\mathcal{G}^{N,j}_\alpha$ with $|j-k|\geq4$, $\alpha\in\{A,B,C\}$,
\begin{equation*}
\Prob{\omega}{\beta}{H_{\Omega^N_0}=H_{\omega_k}}\leq C_0\left( (3/5)^N + N^3\beta e^{-\beta}\right).
\end{equation*}
To see this we first approximate by the ideal process, as in Lemma~\ref{lemmaideal}, and then we make a simple analysis as in~\eqref{23} below. Then, the third term in \eqref{24} is
\begin{equation}\label{26}
    \mc O\left(N^24^M+N^64^M\beta e^{-\beta}\right)e^{-(M+1)\beta}.
\end{equation}
The result follows summing \eqref{22}, \eqref{25} and \eqref{26}.
\end{proof}

We obtain Theorem \ref{theo1} as an immediate corollary of Proposition \ref{tracoemOmega0}.

\begin{proof}[Proof of Theorem \ref{theo1}] The rates of the speeded up process $\{\eta_0(e^{M\beta}t):t\geq0\}$ are simply the rates for the process $\{\eta_0(t):t\geq0\}$ multiplied by $e^{M\beta}$. In the case where $N_A$, $N_B$ and $N_C$ are constants, we obtain, multiplying \eqref{27} by $e^{M\beta}$ and sending $\beta\uparrow\infty$, that the process $\{\eta_0(e^{M\beta}t):t\geq0\}$ converges to a Markov chain in $\Omega^N_0$, which jumps from $\omega_i$ to $\omega_j$ with rate $r(N_A,N_B,N_C,j-i)$ given in~\eqref{28}.
\end{proof}

\section{Understanding $g^N_\alpha(k)$}\label{secgnk}

To understand the scaling limits of the system when $N\uparrow\infty$ with~$\beta$, we need to estimate $g^N_\alpha(k)$. Consider the ideal random walk $\{\widehat \eta_1 (t):t\geq0\}$ starting from $\xi^0_{\alpha,i}$, $1\leq i \leq N_\alpha-1$. For its jump chain, in each step, the probability of no absorption in $\Omega^N_0$ is less than or equal to $3/5$. To arrive at $\omega_k$, $k\neq0$, it is necessary to survive at least $(|k|-1)(M-2)$ steps without absorption in $\Omega^N_0$. Therefore, for every $k\neq0$ and $1\leq i \leq M-1$,
\begin{equation}\label{23}
    p^N_\alpha(i,k)\leq\left(\frac{3}{5}\right)^{(M-2)(|k|-1)}.
\end{equation}
This simple analysis indicates the fast decaying of $g^N_\alpha(k)$, both in $M$ and in $k$, but it does not give information about $g^N_\alpha(1)$ or $g^N_\alpha(-1)$ which are not negligible. So, we need to go further into the calculations. The next lemma reduces the analysis of $g^N_\alpha(k)$ to the analysis of the terms $p^N_\alpha(1,k)$ and $p^N_\alpha(M-1,k)$.

\begin{lemma}\label{carasdaponta}Let $\alpha$ be such that $N_\alpha=M$. For every $k\in\Lambda_N$, $k\neq0$, we have that
\begin{equation}\label{35}
    g^N_\alpha(k)=\left(\frac{3}{2}+\frac{2}{3^{M-2}+1}\right)\big(p^N_\alpha(1,k)+p^N_\alpha(M-1,k)\big)
\end{equation}
\end{lemma}

\begin{proof} Let us fix $N_A$, $N_B$, $N_C$ and $k\neq0$. By the standard conditioning argument we have that, for $2\leq i \leq M-2$,
\begin{equation}\label{recorpNik}
    p^N_\alpha(i,k)=\frac{3}{10}p^N_\alpha(i-1,k)+\frac{3}{10}p^N_\alpha(i+1,k).
\end{equation}
This recurrence relation has characteristic equation $(3/10)\lambda^2-\lambda+(3/10)=0$, whose roots are $\lambda_1=3$ and $\lambda_2=1/3$. So that, we find the closed form
\begin{equation}\label{closedpNik}
    p^N_\alpha(i,k)=h^1_{N,k}3^i+h^2_{N,k}\left(\frac{1}{3}\right)^i, ~~i=1,\ldots,M-1,
\end{equation}
where $h^1_{N,k}$ and $h^2_{N,k}$ are constants independent of $i$, which may be computed in terms of $p^N_\alpha(1,k)$ and $p^N_\alpha(M-1,k)$ using the relation \eqref{closedpNik} for $i=1$ and $i=M-1$.
Now
\begin{equation*}
    g^N_\alpha(k)=\sum_{i=1}^{M-1}p^N_\alpha(i,k)=h^1_{N,k}\sum_{i=1}^{M-1}3^i+h^2_{N,k}\sum_{i=1}^{M-1}\left(\frac{1}{3}\right)^i,
\end{equation*}
and we get \eqref{35} after elementary calculations.
\end{proof}

This lemma has an easy and useful corollary that would be sufficient to prove the particular case of Theorem~\ref{browdrift} under the assumption \eqref{3}.

\begin{corollary}\label{corolgnk} There exists an universal constant $C_0$ such that for all $N$ and $k\in\Lambda_N$, $k\neq0$,
\begin{equation*}
    \left|g^N_\alpha(k)-\frac{3}{4}\mb 1\{|k|=1\}\right|\leq C_0\left(\frac{3}{5}\right)^{M}.
\end{equation*}
\end{corollary}

\begin{proof} For $|k|>1$, the result follows from the previous lemma and observation~\eqref{23}.
Let us consider the case $k=1$ (the case $k=-1$ is analogous). By the same argument that leads to \eqref{23} we get that
$p^N_\alpha(M-1,1)\leq(3/5)^{M-2}$. So, using the previous lemma, we just need to care about $p^N_\alpha(1,1)$. If $M$ is large, we expect that $p^N_\alpha(1,1)$ is near to $1/2$. A way of formalizing this without much effort involving computations is to couple the jump chain of $\{\widehat{\eta}_1(t):t\geq0\}$ with another process for which we can use symmetry. The idea is very simple but requires some notations.
Let $(X_n)_{n\geq0}$ be the discrete-time jump chain associated to the process $\{\widehat{\eta}_1(t):t\geq0\}$ starting from $\xi^0_{\alpha,1}$ (its jump probabilities are given in Figure~\ref{fig6}). Define the hitting time~$H^X$,
\begin{equation*}
H^X=H^X_{\left\{\xi^{-1}_{(\alpha+1),1},\xi^{1}_{(\alpha-1),1},\omega_0,\omega_1\right\}},
\end{equation*}
as the first time the chain $(X_n)_{n\geq0}$ visits any of the configurations $\xi^{-1}_{(\alpha+1),1}$, $\xi^{1}_{(\alpha-1),1}$, $\omega_0$, $\omega_1$.
We can conclude that $p^N_\alpha(1,1)=1/2+\mc O ((3/5)^M)$ if we show that both the event $H^X=H^X_{\omega_1}$ and the event $H^X=H^X_{\omega_0}$ have probability $1/2+\mc O ((3/5)^M).$

To achieve this, consider an auxiliary discrete-time chain~$(\widehat X_n)_{n\geq0}$ defined on the infinite set
$\bb Z \cup \{u_-,u_+\}$ starting from $0$. To define the jump probabilities of this chain, consider $\widehat S:\bb Z\setminus\{0\}\to\{u_-,u_+\}$ defined as
\begin{equation*}
   \widehat S(i)=
\begin{cases}
 u_-    & \text{if $i<0$}, \\
 u_+ & \text{if $i>0$}. \\
\end{cases}
\end{equation*}
We define the chain $(\widehat X_n)_{n\geq0}$ imposing that, from $i\in\bb Z\setminus\{0\}$, it jumps to $i\pm1$ with probability $3/10$ and to $\widehat S(i)$ with probability $2/5$. From $0$, it jumps to $\pm1$ with probability $3/14$ and to $u_-$ and $u_+$ with probability $2/7$. We define $u_-$ and $u_+$ as absorbing states. By symmetry, this chain is absorbed in $u_-$, with probability $1/2$, or in $u_+$, with probability $1/2$. There is an obvious correspondence between the states of the chain $(\widehat X_n)_{n\geq0}$ near to $0$ and the states of the chain $(X_n)_{n\geq0}$ near to $\xi^0_{\alpha,1}$. We can couple these two chains in such a way that, with this correspondence, they walk together until the time $H^X$. Let $P^{X,\widehat X}$ denote such a coupling. As already done before, observe that $P^{X,\widehat X}\left[H^X\geq M-2\right]\leq(3/5)^{M-3}$. Therefore,
\begin{eqnarray*}
  P^{X,\widehat X}\left[H^X=H^X_{\omega_1}\right] &=& P^{X,Y}\left[H^X=H^X_{\omega_1},H^X<M-2\right]+\mc O \left(\left(\frac{3}{5}\right)^M\right)  \\
   &=& P^{X,\widehat X}\left[H^{\widehat X}=H^{\widehat X}_{u_+},H^X<M-2\right]+\mc O \left(\left(\frac{3}{5}\right)^M\right)  \\
   &=& \frac{1}{2}+ \mc O \left(\left(\frac{3}{5}\right)^M\right),
\end{eqnarray*}
and the same holds changing $\omega_1$ by $\omega_0$ and $u_+$ by $u_-$. With this, we conclude that $p^N_\alpha(1,1)=1/2+\mc O \left(\left(3/5\right)^M\right)$, and then the result follows from Lemma \ref{carasdaponta}.
\end{proof}

Now, for $l,m\geq3$ we will define a quantity $v(l,m)$ in terms of some absorbtion probabilities of a simple discrete-time Markov chain $(Y^{l,m}_n)_{n\geq0}$ that depends on $l$ and $m$. We will see later that this quantity will represent the velocity of the ballistic process that appears in the statement of Theorem~\ref{ballistic}.

Let us define the chain $(Y^{l,m}_n)_{n\geq0}$. Its state space is the set $\bb Z \cup\left\{u_1,u_0,u_{-1},u_{-2}\right\}$. To define its jump probabilities consider the function $S:\bb Z\setminus\{-(m-2),0,l-2\}\to\left\{u_{1},u_{0},u_{-1},u_{-2}\right\}$ defined as
\begin{equation*}
    S(i)=
\begin{cases}
 u_{-2}    & \text{if $i<-(m-2)$}, \\
 u_{-1}    & \text{if $-(m-2)<i<0$}, \\
 u_{0}     & \text{if $0<i<l-2$}, \\
 u_{1}     & \text{if $i>l-2$}.
\end{cases}
\end{equation*}
We define the chain $(Y^{l,m}_n)_{n\geq0}$ imposing that, from $i\in\bb Z\setminus\{-(m-2),0,l-2)\}$, it jumps to $i\pm1$ with probability $3/10$ and to $S(i)$ with probability $2/5$. From $i\in\{-(m-2),0,l-2)\}$ it jumps to $i\pm1$ with probability $3/14$ and to $S(i\pm1)$ with probability $2/7$. The states $u_1$, $u_{0}$, $u_{-1}$ and $u_{-2}$ are defined as absorbing states. Figure \ref{fig7} illustrates the structure of this simple chain. Roughly speaking, after an identification of the states,  $(Y^{l,m}_n)_{n\geq0}$ is the jump chain of the process $\{\widehat\eta_1(t):t\geq0\}$, illustrated in Figure~\ref{fig6}, with $N_A=l$, $N_B=m$ and $N_C=\infty$.

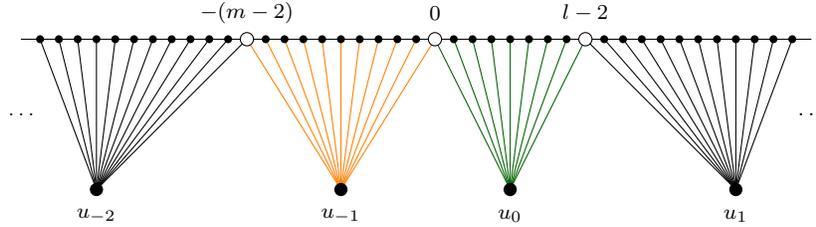
\begin{figure}[h!]
\centering
\definecolor{ffqqff}{rgb}{1,0,1}
\definecolor{ffxfqq}{rgb}{1,0.5,0}
\definecolor{qqwuqq}{rgb}{0,0.39,0}
\definecolor{ffffff}{rgb}{1,1,1}
\begin{tikzpicture}[line cap=round,line join=round,>=triangle 45,x=1.0cm,y=1.0cm]
\clip(1.14,1.03) rectangle (12.54,4.85);
\draw (1.5,4)-- (12,4);
\draw [color=qqwuqq] (7,4)-- (8,2);
\draw [color=qqwuqq] (7.25,4)-- (8,2);
\draw [color=qqwuqq] (7.5,4)-- (8,2);
\draw [color=qqwuqq] (7.75,4)-- (8,2);
\draw [color=qqwuqq] (8,4)-- (8,2);
\draw [color=qqwuqq] (8.25,4)-- (8,2);
\draw [color=qqwuqq] (8.5,4)-- (8,2);
\draw [color=qqwuqq] (8.75,4)-- (8,2);
\draw [color=qqwuqq] (9,4)-- (8,2);
\draw [color=ffxfqq] (4.5,4)-- (5.75,2);
\draw [color=ffxfqq] (4.75,4)-- (5.75,2);
\draw [color=ffxfqq] (5,4)-- (5.75,2);
\draw [color=ffxfqq] (5.25,4)-- (5.75,2);
\draw [color=ffxfqq] (5.5,4)-- (5.75,2);
\draw [color=ffxfqq] (5.75,4)-- (5.75,2);
\draw [color=ffxfqq] (6,4)-- (5.75,2);
\draw [color=ffxfqq] (6.25,4)-- (5.75,2);
\draw [color=ffxfqq] (6.5,4)-- (5.75,2);
\draw [color=ffxfqq] (6.75,4)-- (5.75,2);
\draw [color=ffxfqq] (7,4)-- (5.75,2);
\draw (1.75,4)-- (2.5,2);
\draw (2,4)-- (2.5,2);
\draw (2.25,4)-- (2.5,2);
\draw (2.5,4)-- (2.5,2);
\draw (2.75,4)-- (2.5,2);
\draw (3,4)-- (2.5,2);
\draw (3.25,4)-- (2.5,2);
\draw (3.5,4)-- (2.5,2);
\draw (3.75,4)-- (2.5,2);
\draw (4,4)-- (2.5,2);
\draw (4.25,4)-- (2.5,2);
\draw (4.5,4)-- (2.5,2);
\draw (9,4)-- (11,2);
\draw (9.25,4)-- (11,2);
\draw (9.5,4)-- (11,2);
\draw (9.75,4)-- (11,2);
\draw (10,4)-- (11,2);
\draw (10.25,4)-- (11,2);
\draw (10.5,4)-- (11,2);
\draw (10.75,4)-- (11,2);
\draw (11,4)-- (11,2);
\draw (11.25,4)-- (11,2);
\draw (11.5,4)-- (11,2);
\draw (11.75,4)-- (11,2);
\begin{scriptsize}
\fill [color=ffffff] (9,4) circle (2.5pt);
\fill [color=ffffff] (7,4) circle (2.5pt);
\fill [color=ffffff] (4.5,4) circle (2.5pt);
\fill [color=black] (11,2) circle (2.5pt);
\fill [color=black] (8,2) circle (2.5pt);
\fill [color=black] (5.75,2) circle (2.5pt);
\fill [color=black] (2.5,2) circle (2.5pt);
\fill [color=black] (1.75,4) circle (1.5pt);
\fill [color=black] (2,4) circle (1.5pt);
\fill [color=black] (2.25,4) circle (1.5pt);
\fill [color=black] (2.5,4) circle (1.5pt);
\fill [color=black] (2.75,4) circle (1.5pt);
\fill [color=black] (3,4) circle (1.5pt);
\fill [color=black] (3.25,4) circle (1.5pt);
\fill [color=black] (3.5,4) circle (1.5pt);
\fill [color=black] (3.75,4) circle (1.5pt);
\fill [color=black] (4,4) circle (1.5pt);
\fill [color=black] (4.25,4) circle (1.5pt);
\fill [color=black] (9.25,4) circle (1.5pt);
\fill [color=black] (9.5,4) circle (1.5pt);
\fill [color=black] (9.75,4) circle (1.5pt);
\fill [color=black] (10,4) circle (1.5pt);
\fill [color=black] (10.25,4) circle (1.5pt);
\fill [color=black] (10.5,4) circle (1.5pt);
\fill [color=black] (10.75,4) circle (1.5pt);
\fill [color=black] (11,4) circle (1.5pt);
\fill [color=black] (11.25,4) circle (1.5pt);
\fill [color=black] (11.5,4) circle (1.5pt);
\fill [color=black] (11.75,4) circle (1.5pt);
\fill [color=black] (4.75,4) circle (1.5pt);
\fill [color=black] (5,4) circle (1.5pt);
\fill [color=black] (5.25,4) circle (1.5pt);
\fill [color=black] (5.5,4) circle (1.5pt);
\fill [color=black] (5.75,4) circle (1.5pt);
\fill [color=black] (6,4) circle (1.5pt);
\fill [color=black] (6.25,4) circle (1.5pt);
\fill [color=black] (6.5,4) circle (1.5pt);
\fill [color=black] (6.75,4) circle (1.5pt);
\fill [color=black] (7.25,4) circle (1.5pt);
\fill [color=black] (7.5,4) circle (1.5pt);
\fill [color=black] (7.75,4) circle (1.5pt);
\fill [color=black] (8,4) circle (1.5pt);
\fill [color=black] (8.25,4) circle (1.5pt);
\fill [color=black] (8.5,4) circle (1.5pt);
\fill [color=black] (8.75,4) circle (1.5pt);
\draw [color=black] (9,4) circle (2.5pt);
\draw [color=black] (7,4) circle (2.5pt);
\draw [color=black] (4.5,4) circle (2.5pt);
\draw (1.5,3) node {$\ldots$};
\draw (12,3) node {\footnotesize $\ldots$};
\draw (4.5,4.35) node {\footnotesize $-(m-2)$};
\draw (7,4.35) node {\footnotesize $0$};
\draw (9,4.35) node {\footnotesize $l-2$};
\draw (11,1.65) node {\footnotesize $u_1$};
\draw (8,1.65) node {\footnotesize $u_0$};
\draw (5.75,1.65) node {\footnotesize $u_{-1}$};
\draw (2.5,1.65) node {\footnotesize $u_{-2}$};
\end{scriptsize}
\end{tikzpicture}
\caption{The graph structure of the chain  $(Y^{l,m}_n)_{n\geq0}$.}
\label{fig7}
\end{figure}

For $0\leq i \leq l-2$ and $k\in\{1,0,-1,-2\}$ denote by $p^{l,m}(i,u_k)$ the probability for the chain $(Y^{l,m}_n)_{n\geq0}$ starting from $i$ being absorbed in $u_k$. We define $v(l,m)$ as
\begin{equation*}
    v(l,m):=\frac{2}{3}\sum_{k=-2}^1k\left(\sum_{i=0}^{l-2}p^{l,m}(i,u_k)\right).
\end{equation*}
Repeating the proof of Lemma~\ref{carasdaponta}, with  an obvious identification of the states, we get that
\begin{equation}\label{37}
    v(l,m)=\frac{2}{3}\left(\frac{3}{2}+\frac{2}{3^{l-2}+1}\right)\sum_{k\in\{-2,-1,1\}}\sum_{i\in\{0,l-2\}}kp^{l,m}(i,u_k).
\end{equation}

\begin{lemma}\label{lemavelocity} There exists a constant $C_0$ such that for any $3\leq N_A<N_B<N_C$,
\begin{equation*} \left|\frac{2}{3}\sum_{k\in\Lambda_N}kg^N_A(k)-v(N_A,N_B)\right|\leq C_0N_C^2 \left(\frac{3}{5}\right)^{N_C}
\end{equation*}
\end{lemma}

\begin{proof} If $k\notin\{-2,-1,0,1\}$ we note that for $i=1,N_A-1$, starting from $\xi^0_{A,i}$, the chain $\{\widehat \eta_1 (t):t\geq0\}$ must make at least $N_C$ jumps to arrive at $\omega_k$. So, using Lemma~\ref{carasdaponta} we get that
\begin{equation}\label{41}
    g^N_A(k)\leq C_0(3/5)^{N_C}, \quad k \notin\{-2,-1,0,1\}.
\end{equation}
For $k\in\{-2,-1,0,1\}$ we use the same coupling argument used in the last corollary, now coupling the jump chain of $\{\widehat \eta_1(t):t\geq0\}$ with the chain $(Y^{N_A,N_B}_n)_{n\geq0}$. If we do not survive at least $N_C$ steps without absorption, we do not feel the difference of these two chains. So, in \eqref{35} we may change $p^N_A(i+1,k)$ by $p^{N_A,N_B}(i,u_k)$, $i=0,N_A-2$, causing errors of order $\mc O((3/5)^{N_C})$.
\end{proof}

The key to obtain the specific scenario in which the process will converge to a Brownian motion with drift is to understand the dependence of $v(N_A,N_B)$ on $N_A$ and $N_B$. The next lemma is sufficient for this purpose.

\begin{lemma}\label{lemmaprodrift}Suppose that $3\leq N_A<N_B$, then
\begin{equation}\label{39}
v(N_A,N_B)=\left[-3+\mc O \left(\left(\frac{1}{3}\right)^{N_A}\right)\right]\left(\frac{1}{3}\right)^{N_B}+\mc O \left(\left(\frac{1}{3}\right)^{2N_B}\right).
\end{equation}
\end{lemma}

\begin{proof} By \eqref{37}, in order to explicitly compute $v(l,m)$ we just need to compute the six absorption probabilities
\begin{equation}\label{38}
 p^{l,m}(i,u_k), \quad i=0,l-2, \quad k=1,-1,-2.
\end{equation}
Solving a recurrence as in \eqref{recorpNik}, we can compute the distribution of the first visited state in the set $\{0,l-2,u_0\}$ if the chain $(Y^{l,m}_n)_{n\geq0}$ starts from $1$. We find that, the process will first visit $l-2$, $0$ or $u_0$ with probabilities $p_l$, $q_l$ and $r_l$, respectively, where
\begin{equation}\label{vel01}
    p_l=\frac{3-(1/3)}{3^{l-2}-(1/3)^{l-2}}, \quad q_l=\frac{3^{l-3}-(1/3)^{l-3}}{3^{l-2}-(1/3)^{l-2}}, \quad r_l=1-p_l-q_l.
\end{equation}
This provides a simplification of the chain $(Y^{l,m}_n)_{n\geq0}$, which is best explained with Figure \ref{fig8}.

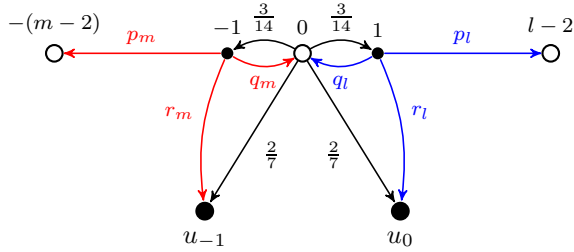
\begin{figure}[h!]
\centering
\tikzstyle{furado}=[circle,draw=black,thick,inner sep=0pt,minimum size=2.2mm]
\tikzstyle{cheio}=[circle,fill=black,thick,inner sep=0pt,minimum size=1.6mm]
\tikzstyle{chao}=[circle,fill=black,thick,inner sep=0pt,minimum size=2.5mm]
\begin{tikzpicture}[line cap=round,line join=round,>=stealth',semithick,x=1.0cm,y=1.0cm,bend angle=30,auto]
\clip(-4.58,-2.7) rectangle (4.7,0.77);
\node (0) at (0,0) [furado,label=above:\footnotesize $0$] {};
\node (m) at (-3.3,0) [furado,label=above:\footnotesize $-(m-2)$] {};
\node (l) at (3.3,0) [furado,label=above:\footnotesize $l-2$] {};
\node (1) at (1,0) [cheio,label=above:\footnotesize $1$] {};
\node (m1) at (-1,0) [cheio,label=above:\footnotesize $-1$] {};
\node (u1) at (1.3,-2.1) [chao,label=below:$u_0$]{};
\node (u0) at (-1.3,-2.1) [chao,label=below:$u_{-1}$] {};
\path (0) edge [->,bend left] node{\footnotesize $\frac{3}{14}$} (1)
          edge [->]            node{\footnotesize $\frac{2}{7}$} (u0)
          edge [<-,bend left,red] node{\footnotesize $q_m$} (m1)
      (1) edge [->,blue]            node{\footnotesize $p_l$}(l)
          edge [->,bend left,blue]  node{\footnotesize $q_l$} (0)
          edge [->,bend left=15,blue]  node{\footnotesize $r_l$} (u1)
      (m1)edge [<-,bend left]  node{\footnotesize $\frac{3}{14}$}(0)
      (m) edge [<-,red] node{\footnotesize{$p_m$}} (m1)
      (u0)edge [<-,bend left=15,red] node{\footnotesize $r_m$}(m1)
      (u1)edge[<-] node{\footnotesize $\frac{2}{7}$} (0)   ;
\end{tikzpicture}
\caption{A first simplification in the dynamics of $(Y^{l,m}_n)_{n\geq0}$.}
\label{fig8}
\end{figure}

With this, we can easily compute the distribution of the first visited configuration in $\{-(m-2),u_{-1},u_0,l-2\}$ if the process starts from $0$. We obtain that the process will first visit $l-2$, $u_0$, $-(m-2)$, $u_{-1}$ with, respectively, probabilities $p_{l,m}$, $q_{l,m}$, $p_{m,l}$ and $q_{m,l}$, where
\begin{equation}\label{vel02}
    p_{l,m}=\frac{\frac{3}{14}p_l}{1-\frac{3}{14}(q_l+q_m)}, \quad q_{l,m}=\frac{\frac{2}{7}+\frac{3}{14}r_l}{1-\frac{3}{14}(q_l+q_m)}.
\end{equation}
Now, if we define
\begin{equation*}
   p_{l,\infty}:=\lim_{j\to\infty}p_{l,j}, \quad q_{l,\infty}:=\lim_{j\to\infty}q_{l,j}, \quad q_{\infty,m}:=\lim_{j\to\infty}q_{j,m},
\end{equation*}
this means that, in order to compute the probabilities in \eqref{38}, we can consider the simple seven-state chain whose jump probabilities are given in Figure~\ref{fig9}.

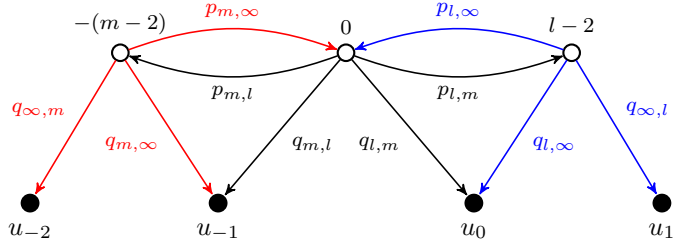
\begin{figure}[h!]
\centering
\tikzstyle{furado}=[circle,draw=black,thick,inner sep=0pt,minimum size=2.2mm]
\tikzstyle{chao}=[circle,fill=black,thick,inner sep=0pt,minimum size=2.5mm]
\begin{tikzpicture}[line cap=round,line join=round,>=stealth',semithick,x=1.0cm,y=1.0cm,bend angle=20,auto]
\clip(-5,-2.5) rectangle (5,1);
\node (0) at (0,0) [furado,label=above:\footnotesize $0$] {};
\node (m) at (-3,0) [furado,label=above:\footnotesize $-(m-2)$] {};
\node (l) at (3,0) [furado,label=above:\footnotesize $l-2$] {};
\node (u1) at (4.2,-2) [chao,label=below:$u_1$]{};
\node (u0) at (1.7,-2) [chao,label=below:$u_0$] {};
\node (um1) at (-1.7,-2) [chao,label=below:$u_{-1}$] {};
\node (um2) at (-4.2,-2) [chao,label=below:$u_{-2}$] {};
\path (0) edge [->,bend left] node{\footnotesize $p_{m,l}$} (m)
          edge [->] node{\footnotesize $q_{m,l}$} (um1)
          edge [<-,bend left,blue] node{\footnotesize $p_{l,\infty}$}(l)
      (u0)edge [<-] node{\footnotesize $q_{l,m}$} (0)
      (l) edge [<-,bend left] node{\footnotesize $p_{l,m}$}(0)
          edge [->,blue]node{\footnotesize $q_{l,\infty}$}(u0)
          edge[->,blue]node{\footnotesize $q_{\infty,l}$}(u1)
       (m)edge [->,bend left,red] node{\footnotesize $p_{m,\infty}$}(0)
       (um2)edge[<-,red] node{\footnotesize $q_{\infty,m}$}(m)
       (um1)edge[<-,red] node{\footnotesize $q_{m,\infty}$}(m);
\end{tikzpicture}
\caption{A further simplification in the dynamics of   $(Y^{l,m}_n)_{n\geq0}$.}
\label{fig9}
\end{figure}
From this point, it is easy to find that
\begin{eqnarray*}
  p^{l,m}(0,u_1)=\frac{p_{l,m}q_{\infty,l}}{1-(p_{l,m}p_{l,\infty}+p_{m,l}p_{m,\infty})}, &\!\!\!&\!\!\!\!\! p^{l,m}(l-2,u_1)=q_{\infty,l}+p_{l,\infty}p^{l,m}(0,u_1), \\
  p^{l,m}(0,u_{-1})=\frac{q_{m,l}+ p_{m,l}q_{m,\infty}}{1-(p_{l,m}p_{l,\infty}+p_{m,l}p_{m,\infty})}, &\!\!\!&\!\!\!\!\!p^{l,m}(l-2,u_{-1})=p_{l,\infty}p^{l,m}(0,u_{-1}),  \\
  p^{l,m}(0,u_{-2})=\frac{p_{m,l}q_{\infty,m}}{1-(p_{l,m}p_{l,\infty}+p_{m,l}p_{m,\infty})}, &\!\!\!&\!\!\!\!\!p^{l,m}(l-2,u_{-2})=p_{l,\infty}p^{l,m}(0,u_{-2}).
\end{eqnarray*}
To simplify the notations, in the next equations $\bs x$ and $\bs y$ will represent respectively $(1/3)^l$ and $(1/3)^m$. After many but elementary calculations we see from the last equations and the explicit expressions \eqref{vel01} and \eqref{vel02} that
\begin{eqnarray*}
  p^{l,m}(0,u_1)& = &\frac{3 \bs x}{1-9\bs x^2}+\mc O (\bs y^2),\\
  p^{l,m}(0,u_{-1})& = &\frac{1-27 \bs x^2}{2-18 \bs x^2}-(3+\mc O (\bs x))\bs y+\mc O (\bs y^2),\\
  p^{l,m}(0,u_{-2})& = &(3+ \mc O (\bs x))\bs y + \mc O (\bs y^2),
\end{eqnarray*}
and
\begin{eqnarray*}
p^{l,m}(l-2,u_1)& = &\frac{1-27 \bs x^2}{2-18 \bs x^2}+\mc O (\bs y^2)\\
p^{l,m}(l-2,u_{-1})& = &\frac{3 \bs x}{1-9\bs x^2}+\mc O (\bs x)\bs y+\mc O (\bs y^2),  \\
p^{l,m}(l-2,u_{-2})& = &\mc O(\bs x)\bs y + \mc O (\bs y^2),
\end{eqnarray*}

From now, many cancellations take place and we see from \eqref{37} that
$v(l,m)=[-3+\mc O (\bs x)]\bs y + \mc O (\bs y^2)$,
as desired.
\end{proof}

\section{Scaling limits for the trace process when $N\uparrow\infty$}\label{convBro}

Recall the definition of the random walk $\{X(t):t\geq0\}$ presented just before Theorem~\ref{ballistic}. 

\begin{theorem}\label{browdrift} Let $\theta_\beta$ be as in \eqref{timescale}. Assume that $\eta(0)=\omega_0$, $3\leq N_A<N_B\leq N_C$ and that $N\uparrow\infty$ as $\beta\uparrow\infty$ in such a way that \eqref{conddrift} holds and that
\begin{equation}\label{browdriftcond2}
    \lim_{\beta\to\infty}(N_C^54^{N_A}+N_C^6N_A\beta)e^{-\beta}=0.
\end{equation}
Then, as $\beta\uparrow\infty$, the process $\{X(t\theta_\beta)/N:t\geq0\}$ converges in the uniform topology to a Brownian motion with drift $\{\mu t + \sigma B_t: t\geq0\}$ on the circle $[-1,1]$. If $b=0$ in \eqref{conddrift}, we may replace the assumption $N_A<N_B$ by $N_A\leq N_B$.
\end{theorem}

In Theorem \ref{browdrift} the number of particles of type $A$ can go to infinity or be constant. The parameters $\mu$ and $\sigma$ can be explicitly computed. In the case where $N_A\uparrow\infty$ we have that $\mu=-3b/2$ and $\sigma=1$. If $N_A$ is constant, then there is a  multiplicative correction $(1+\mc O\left((1/3)^{N_A}\right)$ in these values. The restriction \eqref{browdriftcond2}, which imposes that $N$ can not increase too fast, is not optimal. It comes from our not very accurate estimate of $r_\beta(k)$ in Proposition \ref{tracoemOmega0}. A more careful analysis would need to take into account a much larger combinatorial complexity.

In this section we will prove Theorems \ref{ballistic} and \ref{browdrift}. The general strategy is the same for the two proofs, and so we start considering a general context.

For $\tilde\theta_\beta$, a function of $\beta$, define the process $\{Y^{\beta}(t): t\geq0\}$ by
\begin{equation*}
    Y^{\beta}(t)=\frac{X(\tilde\theta_\beta t)}{N}.
\end{equation*}
By \cite[Theorem 8.7.1]{d}, in order to prove that, as $\beta\to\infty$, the Markov chain $\{Y^{\beta}(t):t\geq0\}$ converges to a diffusion $\{\mu t + \sigma B_t:t\geq0\}$, it is enough to verify the convergence of the corresponding infinitesimal mean and covariance and a condition that rules out jumps in the limit. More precisely, we have to show that
\begin{equation}\label{c1}
\lim_{\beta\to\infty}\sum_{k\in \Lambda_N}\left(\frac{k}{N}\right)^2r_\beta(k)\tilde\theta_\beta=\sigma^2, \quad\lim_{\beta\to\infty}\sum_{k\in \Lambda_N}\frac{k}{N}r_\beta(k)\tilde\theta_\beta=\mu,
\end{equation}
and that, for every $\delta>0$
\begin{equation}\label{c3}
\lim_{\beta\to\infty}\sum_{|k|>N\delta}r_\beta(k)\tilde\theta_\beta=0.
\end{equation}
By Proposition \ref{tracoemOmega0}, if for $\psi(\beta)$ satisfying~\eqref{28,5}, we can prove that
\begin{equation}\label{cc0}
\lim_{\beta\to\infty}N\tilde\theta_\beta\psi(\beta)=0,
\end{equation}
then we may replace conditions \eqref{c1} and \eqref{c3} by the following:
\begin{equation}\label{cc1}
\lim_{\beta\to\infty}\sum_{\alpha:N_\alpha=M}\left(1+\frac{2}{3}\sum_{k\in\Lambda_N}k^2g^N_\alpha(k)\right)e^{-M\beta}\tilde\theta_\beta N^{-2}=\sigma^2,
\end{equation}
\begin{equation}\label{cc2}
\lim_{\beta\to\infty}\frac{2}{3}\left(\sum_{\alpha:N_\alpha=M}\sum_{k\in \Lambda_N}kg^N_\alpha(k)\right)e^{-{M\beta}}\tilde\theta_\beta N^{-1}=\mu,
\end{equation}
and, for every $\delta>0$
\begin{equation}\label{cc3}
\lim_{\beta\to\infty}\sum_{\alpha: N_\alpha=M}\sum_{|k|>N\delta}g^N_\alpha(k)e^{-M\beta}\tilde\theta_\beta=0.
\end{equation}
For the sake of clarity, from now on we look at each case separately. In fact, all the work has already been done.

\begin{proof}[Proof of Theorem \ref{ballistic}] In this case $\tilde\theta_\beta=Ne^{N_A\beta}$. Since $N_A<N_B$ are constants, we get $\eqref{cc0}$ from assumption~\eqref{ballcond}. Just observing~\eqref{41} we obtain~\eqref{cc3} and~\eqref{cc1} with $\sigma^2=0$. To conclude, observe that Lemma~\ref{lemavelocity} gives~\eqref{cc2} with $\mu=v(N_A,N_B)$.
\end{proof}

\begin{proof}[Proof of Theorem \ref{browdrift}] In this case $\tilde\theta_\beta=\theta_\beta$. Consider first the case $b>0$. In this setting, since $N_A<N_B\leq N_C$, \eqref{cc0} follows from assumption \eqref{browdriftcond2}. As before, \eqref{cc3} follows from \eqref{41}. For simplicity, let's suppose that we are in the case where $N_A\uparrow\infty$. By Lemma \ref{lemavelocity} the limit in \eqref{cc2} is equal to $\lim_{\beta\to\infty}(1/2)v(N_A,N_B)N$ and then, from Lemma \ref{lemmaprodrift} and assumption \eqref{conddrift} we get \eqref{cc2} with $\mu=-3b/2$. To conclude, it remains to verify \eqref{cc1}. By \eqref{41}, the limit in \eqref{cc1} is equal to
\begin{equation}\label{58}
\frac{1}{2}\left(1+\lim_{\beta\to\infty}\frac{2}{3}\sum_{k=-2}^1k^2g^N_A(k)\right).
\end{equation}
Now, \eqref{cc1} with $\sigma^2=1$ follows easily from Corollary~\ref{corolgnk}. If $N_A$ is constant, we could, for example, analyze the limit \eqref{58} in the same way that we have estimated $\sum_{k\in\Lambda_N}kg^N_\alpha(k)$ in Lemmas \ref{lemavelocity} and \ref{lemmaprodrift}, obtaining this way, after many (but elementary) calculations, that \eqref{cc1} holds with $\sigma^2=1+\mc O ((1/3)^{N_A})$.
Now, noting that, by symmetry, \eqref{cc2} is equal to $0$ if $N_A=N_B$ and recalling the definition of $d$, we see that the convergences still work for $N_A\leq N_B$ in the case $b=0$.
\end{proof}

\section{Proof of Lemma \ref{lemma3}} \label{poucofora}

\subsection{Some estimates for the invariant measure}\label{invmeasure}

Recall the notations introduced in the beginning of Section \ref{trace1}. Let $\mu_\beta$ be the invariant measure of the ABC process $\{\eta(t):t\geq0\}$. In our context the invariance of the measure $\mu_\beta$ is characterized by the fact that, for every $\omega\in \Omega^N$,
\begin{equation}\label{invchar}
    e^{-\beta}\sum_{\xi\in B(\omega)}\mu_\beta(\xi)+\sum_{\xi\in R(\omega)}\mu_\beta(\xi)=|B(\omega)|\mu_\beta(\omega)+e^{-\beta}|R(\omega)|\mu_\beta(\omega).
\end{equation}

\begin{lemma}\label{lemma9.1} For any $\beta>0$, $k\in\Lambda_N$ and $1\leq n \leq M$,
\begin{equation}\label{eqlemma9.1}
    \sum_{\omega\in\Delta^{n}_k}|B^*_k(\omega)|\mu_\beta(\omega)=e^{-\beta}\sum_{\xi\in\Delta^{n-1}_k}|R(\xi)|\mu_\beta(\xi).
\end{equation}
\end{lemma}

\begin{proof} Fix $k\in\Lambda_N$. We prove the result by induction in $n$. The case $n=1$ is just the relation \eqref{invchar} for $\omega=\omega_k$. Now suppose that \eqref{eqlemma9.1} holds for some $n\leq M-1$. Summing the relation~\eqref{invchar} over all $\omega\in\Delta^{n}_k$ we obtain
\begin{equation}\label{42}
\begin{split}
&e^{-\beta}\sum_{\omega\in\Delta^{n}_k}\sum_{\xi\in B(\omega)}\mu_\beta(\xi)+\sum_{\omega\in\Delta^{n}_k}\sum_{\xi\in R(\omega)}\mu_\beta(\xi) \\
&\qquad \:=\: \sum_{\omega\in\Delta^{n}_k}|B(\omega)|\mu_\beta(\omega)+e^{-\beta}\sum_{\omega\in\Delta^{n}_k}|R(\omega)|\mu_\beta(\omega)
\end{split}
\end{equation}
By \eqref{R+B-}, we note that in the first member on the left-hand side of~\eqref{42}, we are measuring configurations on $\Delta^{n-1}_k$. Moreover, note that each configuration $\xi\in\Delta^{n-1}_k$ is counted repeatedly $|R(\xi)|$ times. This is due to the simple fact that there are exactly $|R(\xi)|$ configurations $\omega\in\Delta^n_k$ such that $\xi\in B(\omega)$. So, we may rewrite the first member on the left-hand side of~\eqref{42} as $e^{-\beta}\sum_{\xi\in\Delta^{n-1}_k}|R(\xi)|\mu_\beta(\xi)$ and then, by the induction hypothesis it cancels with the first member on the right-hand side of the equality. Now note that in the second member of the left-hand side, we are measuring configurations in $\Delta^{n+1}_k$ and each configuration $\xi\in\Delta^{n+1}_k$ is counted repeatedly $|B^*_k(\xi)|$ times. Thus we get the relation \eqref{eqlemma9.1} with $n$ replaced by $n+1$, which concludes the proof.
\end{proof}

\begin{corollary}\label{corol9.2} For any $\beta>0$, $k\in\Lambda_N$ and $1\leq n \leq M$
\begin{equation*}
\sum_{\omega\in\Delta^n_k}|B^*_k(\omega)|\mu_{\beta}(\omega)\leq (4e^{-\beta})^n\mu_{\beta}(\omega_k),
\end{equation*}
and then, for $\beta$ large enough,
\begin{equation}\label{46}
\sum_{n=1}^M\sum_{\omega\in\Delta^n_k}|B^*_k(\omega)|\mu_{\beta}(\omega)\leq 5e^{-\beta}\mu_{\beta}(\omega_k).
\end{equation}
\end{corollary}

\begin{proof} The proof consists in iterations of \eqref{eqlemma9.1} using the inequality $|R(\xi)|\leq 4 |B^*_k(\xi)|$, which, by Lemma~\ref{lemma1}, holds for every $\xi\in V(\omega_k)$.
\end{proof}

For any set of configurations $\mc H \subseteq\Omega^N$, define inductively
\begin{equation}\label{47}
    B^0(\mc H)=\mc H, \quad B^n(\mc H)=\bigcup_{\omega\in B^{n-1}(\mc H)}B(\omega), ~n=1,2,\ldots,
\end{equation}
and then define $B^\infty(\mc H)=\bigcup_{n=0}^\infty B^n(\mc H)$.
In the same way, just changing $B$ by $R$ in \eqref{47}, define $R^n(\mc H)$, $n\geq1$. When $\mc H=\{\omega_k\}$ we write simply $R^n_k$ instead of $R^n(\{\omega_k\})$. We omit the index $n$ when $n=1$.

\begin{lemma}\label{lemma9.3} For any $\beta>0$, $k\in\Lambda_N$ and $n\geq1$,
\begin{equation}\label{eqlemma9.3}
\begin{split}
&\sum_{\omega\in R^{M+n}_k}|B^*_k(\omega)|\mu_\beta(\omega)+e^{-\beta}\sum_{i=0}^{n-1}\,\sum_{\omega\in R^{M+i}_k}\,\sum_{\zeta\in D^*_k(\omega)}\mu_\beta(\zeta) \\
&\qquad \:=\:e^{-\beta}\sum_{\xi\in R^{M+n-1}_k}|R(\xi)|\mu_\beta(\xi)+\sum_{i=0}^{n-1}\,\sum_{\xi\in R^{M+i}_k}|D^*_k(\xi)|\mu_\beta(\xi)
\end{split}
\end{equation}
\end{lemma}

\begin{proof} Fix $k\in\Lambda_N$. The proof by induction is very similar to the proof of Lemma~\ref{lemma9.1}. To pass from the case $n$ to the case $n+1$, sum the relation \eqref{invchar} over all $\omega\in R^{M+n}_k$ decomposing ${B(\omega)=B^*_k(\omega)}\cup D^*_k(\omega)$. The inductive argument is completed observing that
\begin{equation*}
\sum_{\omega\in R^{M+n}_k}\sum_{\xi\in B^*_k(\omega)}\mu_\beta(\xi)=\sum_{\xi\in R^{M+n-1}_k}|R(\xi)|\mu_\beta(\xi),
\end{equation*}
and
\begin{equation*}
\sum_{\omega\in R^{M+n}_k}\sum_{\xi\in R(\omega)}\mu_\beta(\xi)=\sum_{\xi\in R^{M+n+1}_k}|B^*_k(\xi)|\mu_\beta(\xi).
\end{equation*}
In the same way, we obtain the base case $n=1$ from the case $n=M$ of Lemma~\ref{lemma9.1}.
\end{proof}

Remind that we have defined $M^*=\max\{N_A,N_B,N_C\}$.

\begin{corollary}\label{corol9.4} There exists a constant $C_0$ such that, for any $\beta>0$,
\begin{eqnarray}
   \mu_\beta\left(\bigcup_{k\in\Lambda_N}\bigcup_{n=M}^{M^*}R^n_k\right) &\leq& C_0 4^{M^*}e^{-M\beta}\mu_{\beta}(\Omega^N_0),\label{49} \\
   \mu_\beta\left(\mc G^N\right)&\leq& C_0 4^{M^*}e^{-(M-1)\beta}\mu_{\beta}(\Omega^N_0), \label{50}\\
   \mu_\beta\left(R(\mc G^N)\right) &\leq& C_0 4^{M^*}e^{-M\beta}\mu_{\beta}(\Omega^N_0).\label{51}
\end{eqnarray}
\end{corollary}

\begin{proof} Fix $k\in\Lambda_N$. Let $Z_k(0)=\sum_{\omega\in\Delta^M_k}|B^*_k(\omega)|\mu_\beta(\omega)$ and, for $n\geq1$, let $Z_k(n)$ be the expression in \eqref{eqlemma9.3}. Obviously, for any $\omega\in V(\omega_k)\setminus\{\omega_k\}$ we have $|B^*_k(\omega)|\geq1$. Let's label the particles of the system in such a way that for the configuration $\omega_k$, the labels $A_1,A_2,\ldots, A_{N_A}, B_1, B_2, \ldots$ are placed clockwise and the particle $A_1$ is the particle of type $A$ that is adjacent to a particle of type $C$, which is $C_{N_C}$. With these labels we note that, when they exist, the blue edges of a configuration $\omega\in V(\omega_k)$ of the type whose transposition leads to configurations in $D^*_k(\omega)$ connect the adjacent particles $(\alpha-1)_{N_{(\alpha-1)}}$ and $(\alpha+1)_1$, for some $\alpha \in \{A,B,C\}$, see Figure~\ref{fig1}. So, for configurations in $V(\omega_k)$ there are at most three of such blue edges, that is $|D^*_k(\omega)|\leq3$. Moreover, from $\omega_k$ to achieve a configuration $\omega\in V(\omega_k)$ with $|D^*_k(\omega)|=3$, at least $N_A+N_B+N_C-3>M^*$ jumps are necessary. So, for configurations $\omega\in\bigcup_{n=M}^{M^*}R^n_k$, we have $|D^*_k(\omega)|\leq2$.  These inequalities and the one of Lemma~\ref{lemma1} applied to equality~\eqref{eqlemma9.3} allow us to conclude that, for $1\leq n\leq M^*-M+1$
\begin{equation*}
    Z_k(n)\leq (4e^{-\beta})Z_k(n-1)+2\sum_{i=0}^{n-1}Z_k(i).
\end{equation*}
Thus, just using that $4e^{-\beta}<1$ and Corollary \ref{corol9.2}, we get
\begin{equation}\label{48}
    Z_k(n)\leq 4^{n}Z_k(0)\leq 4^{n}4^Me^{-\beta}\mu_\beta(\omega_k).
\end{equation}
Now note that
\begin{equation*}
    \mc G^N\subseteq \bigcup_{k\in\Lambda_N}\,\,\bigcup_{m\in\{N_A,N_B,N_C\}}\,\,\bigcup_{\omega\in R^m_k}D^*_k(\omega).
\end{equation*}
So, by the definition of $Z_k(n)$ and \eqref{48},
\begin{equation*}
\mu_\beta\left(\bigcup_{k\in\Lambda_N}\bigcup_{n=M}^{M^*}R^n_k\right)+e^{-\beta}\mu_\beta(\mc G^N)\leq\sum_{k\in\Lambda_N} \sum_{n=0}^{M^*-M+1}Z_k(n)\leq C_0 4^{M*}e^{-M\beta}\mu_\beta(\Omega^N_0),
\end{equation*}
which proves \eqref{49} and \eqref{50}. To obtain the inequality \eqref{51}, just note that for each $\omega\in \mc G^N$, the invariance of $\mu_\beta$ says that $\sum_{v\in R(\omega)}\mu_\beta(v)=6e^{-\beta}\mu_\beta(\omega)$, and so $\mu_\beta(R(\mc G^N))\leq 6e^{-\beta}\mu_\beta(\mc G^N)$.
\end{proof}

\subsection{The set that the process never leaves}

The proof of Lemma~\ref{lemma3} consists in finding a set of configurations $\Xi^N$ that the process never leaves in the time scale $e^{M\beta}$. The set $\Xi^N$ needs to be sufficiently small so that $\mu_\beta(\Xi^N\setminus\Omega^N_0)/\mu_\beta(\Omega^N_0)$ vanishes as $\beta\uparrow\infty$, which will allow us to conclude that, in fact, the process stays almost always in $\Omega^N_0$. The analysis of the excursions between two consecutive visits to the set $\Omega^N_0$, which we have made in Section~\ref{trace1}, suggests a natural candidate. We define $\Xi^N$ as
\begin{equation}\label{Xi}
    \Xi^N=B^\infty\left(\bigcup_{k\in\Lambda_N}\bigcup_{n=0}^M R^n_k\right)\cup B^\infty\left(R(\mc G^N)\right).
\end{equation}
This choice is optimal in the sense that, starting from $\Omega^N_0$, any configuration in $\Xi^N$ can, in fact, be visited after a time of order $e^{M\beta}$. It is interesting to note that the configurations in $\Xi^N$ that are not in $\bigcup_{k\in\Lambda_N}\bigcup_{n=0}^M\Delta^n_k$ are very similar to the configurations in $\Omega^N_0$, differing by at most two particles that are detached from their corresponding blocks. This observation, which will be crucial for Section \ref{convcenterofmass}, justifies the inclusion $\Xi^N\subseteq\Gamma^N$ stated in Section \ref{mainresults}.

Let $\partial\Xi^N$ denote the boundary of $\Xi^N$, that is
\begin{equation*}
    \partial\Xi^N=\left\{\omega\in\Xi^N : \text{there exists $\xi\in\Omega^N\setminus\Xi^N$ and $i\in\Lambda_N$ such that $\xi=\sigma^{i,i+1}\omega$} \right\}.
\end{equation*}
Note that $\Xi^N$ was defined in such a way that
\begin{equation}\label{55}
    B(\xi)\subseteq\Xi^N, \text{ if }\xi \in\partial\Xi^N.
\end{equation}

\begin{lemma}\label{medidadeXi} There exists a constant $C_0$ such that, for any $\beta>0$,
\begin{equation}\label{52}
    \mu_\beta (\Xi^N\setminus\Omega^N_0)\leq C_04^{M^*}e^{-\beta}\mu_\beta(\Omega^N_0)
\end{equation}
and
\begin{equation}\label{53}
    \mu_\beta (\partial\Xi^N)\leq C_04^{M^*}e^{-M\beta}\mu_\beta(\Omega^N_0).
\end{equation}
\end{lemma}
\begin{proof} Note that
\begin{equation*}
    \Xi^N\setminus\Omega^N_0\subseteq\left(\bigcup_{k\in\Lambda_N}\bigcup_{n=1}^{M^*}R^n_k\right)\cup\mc G^N \cup R(\mc G^N)
\end{equation*}
and
\begin{equation*}
    \partial\Xi^N\subseteq\left(\bigcup_{k\in\Lambda_N}\bigcup_{n=M}^{M^*}R^n_k\right)\cup R(\mc G^N).
\end{equation*}
Now use Corollaries \ref{corol9.2} and \ref{corol9.4}.
\end{proof}

\begin{proof}[Proof of Lemma \ref{lemma3}]We follow the strategy presented in \cite{gl5}. Fix $k\in\Lambda_N$ and $t\geq0$. We first claim that
\begin{equation}\label{54}
    \lim_{\beta\to\infty}\Exp{\omega_k}{\beta}{\int_0^t\mb1 \{\eta(sN^2e^{M\beta})\notin \Xi^N\}ds}=0.
\end{equation}
For each $\omega\in\partial\Xi^N$, denote by $J(\omega, t)$ the number of jumps from $\omega$ to configurations in $\Omega^N\setminus\Xi^N$ in the time interval $[0,t]$ and let $R^\beta(\omega, \Omega^N\setminus\Xi^N)$ be the total jump rate from $\omega$ to $\Omega^N\setminus\Xi^N$. Note that the process $\{(\eta(t),J(\omega,t)) :t\geq0\}$ is a Markov chain. If $\tilde L_\beta$ stands for the generator of this chain and $f(\xi,n)=n$, it is easy to see that $\tilde L_\beta f(\xi,n)=\mb 1\{\xi=\omega\}R^\beta(\omega,\Omega^N_0\setminus\Xi^N)$, so that
\begin{equation*}
\{J(\omega, t)-\int_0^tR^\beta(\omega,\Omega^N_0\setminus\Xi^N)\mb 1\{\eta(s)=\omega\}ds: ~ t\geq0\}
\end{equation*}
is a martingale, and thus
\begin{equation*}
    \Exp{\omega_k}{\beta}{J(\omega,t)}=\Exp{\omega_k}{\beta}{\int_0^tR^\beta(\omega,\Omega^N_0\setminus\Xi^N)\mb 1\{\eta(s)=\omega\}ds}.
\end{equation*}
Therefore, if we define $J(t)=\sum_{\omega\in\partial\Xi^N}J(\omega,t)$, by observation~\eqref{55} we get that
\begin{equation*}
    \Prob{\omega_k}{\beta}{J(t)\geq1}\,\leq\,\Exp{\omega_k}{\beta}{J(t)}\,\leq\, C_0Ne^{-\beta}\Exp{\omega_k}{\beta}{\int_0^t\mb 1\{\eta(s)\in\partial\Xi^N\}ds}.
\end{equation*}
By symmetry,
\begin{equation}\label{56}
\begin{split}
& \Exp{\omega_k}{\beta}{\int_0^t\mb 1\{\eta(s)\in\partial\Xi^N\}ds}\,=\,\frac{1}{|\Omega^N_0|}\sum_{j\in\Lambda_N}\Exp{\omega_j}{\beta}{\int_0^t\mb 1\{\eta(s)\in\partial\Xi^N\}ds} \\
&\quad \,=\, \frac{1}{|\Omega^N_0|\mu_\beta(\omega_0)}\sum_{j\in\Lambda_N}\mu_\beta(\omega_j)\Exp{\omega_j}{\beta}{\int_0^t\mb 1\{\eta(s)\in\partial\Xi^N\}ds} \;.
\end{split}
\end{equation}
The sum is bounded above by
\begin{equation*}
    \Exp{\mu_\beta}{\beta}{\int_0^t\mb 1\{\eta(s)\in\partial\Xi^N\}ds}=t\mu_\beta(\partial\Xi^N)
\end{equation*}
and the denominator is equal to $\mu_\beta(\Omega^N_0)$, and so
\begin{equation*}
    \Prob{\omega_k}{\beta}{J(tN^2e^{M\beta})\geq1}\leq \frac{C_0 tN^3e^{-\beta} e^{M\beta}\mu_\beta(\partial\Xi^N)}{\mu_\beta(\Omega^N_0)}.
\end{equation*}
By \eqref{53} the above expression is less than or equal to $C_0tN^34^{M^*}e^{-\beta}$, which vanishes as $\beta\uparrow\infty$, in view of assumption \eqref{4}. Therefore, we have proved \eqref{naosaideXiintro},  that is, for any $t\geq0$, starting from $\omega_k$, with probability converging to~$1$, the process $\{\eta(s):s\geq0\}$ does not leave the set $\Xi^N$ in the time interval $[0,tN^2e^{M\beta}]$, which is a result stronger than $\eqref{54}$.
To conclude the proof of the lemma it remains to show that
\begin{equation*}
    \lim_{\beta\to\infty}\Exp{\omega_k}{\beta}{\int_0^t\mb 1\{\eta(sN^2e^{M\beta})\in\Xi^N\setminus\Omega^N_0\}ds}=0.
\end{equation*}
By repeating the arguments used in \eqref{56} we obtain that the previous expectation is bounded by
$t\mu_\beta(\Xi^N\setminus\Omega^N_0)/\mu_\beta(\Omega^N_0)$. By \eqref{52} this is bounded by $C_0t4^{M^*}e^{-\beta}$, which, in view of assumption \eqref{4}, vanishes as $\beta\uparrow\infty$.
\end{proof}

\begin{proof}[About Remark \ref{remarkpoucotempofora}] Now we prove that, for the equal densities case $N_A=N_B=N_C$, \eqref{medconcentra} holds without assumptions controlling the growth of $N$. In fact, this can be derived from the estimates of the partition function $Z_\beta$ presented in \cite{ekkm1,ekkm2}. However, for completeness we present a proof here. Subtracting a function of $M$ in the Hamiltonian~\eqref{1} (in fact, this function is $M^2$ but this not relevant) and incorporating this correction in the partition function~$Z_\beta$, we may suppose that the ground states $\omega_k$, $k\in\Lambda_N$, have energy zero. For each~$n$, the number of configurations with energy $n$ is bounded by $3^{3M}$, since this is a bound for the total number of configurations. So,
\begin{equation*}
    \mu_\beta(\{\omega: \bb H (\omega)>M/2\})\leq\frac{\sum_{n>M/2}3^{3M}e^{-n\beta}}{Z_\beta}\leq C_0\left(27e^{-\beta/2}\right)^M,
\end{equation*}
and this goes to zero when $\beta\uparrow\infty$ (for this term is even better if $M$ grows fast). For configurations with energy at most $M/2$ (which are configurations at distance at most $M/2$ from some ground state) we can use the estimate \eqref{46}. So, in the limit $\beta\uparrow\infty$, in the equal densities case, the invariant measure is concentrated in~$\Omega^N_0$, no matter how fast $N\uparrow\infty$.
\end{proof}

\section{Convergence of the center of mass}\label{convcenterofmass}

In this section we assume the hypothesis of Theorem \ref{mainresult}, under which we will show that, when $\beta\uparrow\infty$, the process $\{\mc C(\eta(tN^2e^{M\beta}):t\geq0)\}$ is close to the process $\{X(tN^2e^{M\beta})/N+r_A/2:t\geq0\}$ in the Skohorod space $D([0,\infty),[-1,1])$.

In the previous section we showed that under \eqref{4} we have \eqref{naosaideXiintro}, where $\Xi^N$ is the set defined in \eqref{Xi}. Later it will be useful to note that $\Xi^N$ can also be expressed as the union $\Xi^N=\bigcup_{k\in\Lambda_N}\Xi^N_k$ where
\begin{equation*}
     \Xi^N_k = \bigcup_{n=0}^M\Delta^n_k ~ ~\cup\bigcup_{\alpha\in\{A,B,C\}}B^\infty \left(R(\mc G^{N,k}_\alpha)\right).
     \end{equation*}

In order to compare the process $\{\mc C(\eta(tN^2e^{M\beta}):t\geq0)\}$ with the trace process $\{X(tN^2e^{M\beta})/N+r_A/2:t\geq0\}$ we will use the process, derived from $\{\eta(t):t\geq0\}$, that records the last visit to the set $\Omega^N_0$. Define
\begin{equation*}
\hat X(t) :=
\begin{cases}
\mb X(\eta(t)) & \text{if $\eta(t)\in\Omega^N_0$}, \\
\mb X(\eta(\sigma(t)^-)) & \text{if $\eta(t)\notin\Omega^N_0$}. \\
\end{cases}
\end{equation*}
where $\sigma(t)=\sup\{s\leq t: \eta(s)\in\Omega^N_0\}$. As we have done for $\{X(t):t\geq0\}$, we are omitting the dependence on $\beta$. The last visit process $\{\hat X(t):t\geq0\}$ has the advantage with respect to the trace process that it does not translate in time the original trajectory. 

With the same proof of \cite[Proposition 4.4]{bl2}, with obvious small modifications for our case, under \eqref{5} we have that
\begin{equation}\label{tracelast}
    \lim_{\beta\to\infty}\Exp{\omega_0}{\beta}{\rho(X^*,\hat X^*)}=0,
\end{equation}
where $X^*$ and $\hat X^*$ denote the speeded up processes $\{X(tN^2e^{M\beta})/N:t\geq0\}$ and $\{\hat X(tN^2e^{M\beta})/N:t\geq0\}$, respectively, and $\rho$ is a distance that generates the Skohorod topology in $D([0,\infty),[-1,1])$.

Now we prove that the last visit process is close to the center of mass in the uniform metric.

\begin{proposition}\label{lastmass} Assume the hypothesis of Theorem \ref{mainresult}. For any $\varepsilon>0$ and $t>0$,
\begin{equation*}
    \lim_{\beta\to\infty}\Prob{\omega_0}{\beta}{\sup_{0\leq s \leq tN^2e^{M\beta}}|\hat X(s)/N+r_A/2-\mc C(\eta(s))|>\varepsilon}=0.
\end{equation*}
\end{proposition}

\begin{proof}
Let $\hat N(t)$ be the number of jumps of the process $\{\hat X(s): s\geq0 \}$ during the time interval $[0,tN^2e^{M\beta}]$. We claim that there exists a constant $L$, depending on~$t$, such that
\begin{equation}\label{57}
\lim_{\beta\to\infty}\Prob{\omega_0}{\beta}{\frac{N^2}{L}<\hat N (t)< LN^2}=1.
\end{equation}
In fact, let $N(t)$ be the number of jumps of $\{X(s): s\geq0 \}$ during the time interval $[0,tN^2e^{M\beta}]$. Observing that
\begin{equation*}
N(t)=\hat N\left(t+\int_0^t \mb 1 \{\eta(sN^2e^{M\beta})\notin\Omega^N_0\}ds\right),
\end{equation*}
the claim will be proved if we prove \eqref{57} for $\hat N (t)$ changed by $N(t)$. Now note that
\begin{equation}\label{59}
\Prob{\omega_0}{\beta}{N(t)\geq L N^2}=P\left[T^\beta_1+T^\beta_2+\ldots+T^\beta_{LN^2}<tN^2e^{M\beta}\right],
\end{equation}
where, $T^\beta_i$, $i=1,2,\ldots$ are independent mean $1/\lambda_\beta$ exponential random variables, with $\lambda_\beta=\sum_{k\in\Lambda_N}r_\beta(k)$. By Proposition \ref{tracoemOmega0}, because of \eqref{condNebeta}, there exist constants $1<c_0<C_0<\infty$ such that $e^{M\beta}\lambda_\beta\in(c_0,C_0)$, for all $\beta>0$. Hence, the expression in \eqref{59} goes to zero as $\beta\uparrow\infty$, if $t/L<1/C_0$. In the same way, we show that $\lim_{\beta\uparrow\infty}\Prob{\omega_0}{\beta}{N(t)\leq N^2}=0$, and the claim is proved.

Let $\tau=H_{\Omega^N_0\setminus\{\omega_0\}}$. Using \eqref{naosaideXiintro}, \eqref{57}, the strong Markov property and translation invariance we see that the proposition will be proved if we prove that
\begin{equation}\label{60}
\lim_{\beta\to\infty}N^2\Prob{\omega_0}{\beta}{\sup_{0\leq s \leq \tau}|\mc C(\eta(s))-\mc C(\omega_0)|>\varepsilon, ~ \eta([0,\tau])\subseteq \Xi^N}=0.
\end{equation}
Observe that if $\omega\in\Xi^N_k$ then $|\mc C (\omega) - \mc C(\omega_k)|<1/N_A<\varepsilon/2$, for $\beta$ large enough. So, if $\nu$ denotes the hitting time of $\bigcup_{|k|> N\varepsilon/2}\Xi^N_k$, \eqref{60} will be proved if we prove that
\begin{equation}\label{61}
\lim_{\beta\to\infty}N^2\Prob{\omega_0}{\beta}{\nu\leq\tau,  ~ \eta([0,\tau])\subseteq \Xi^N}=0.
\end{equation}
For $k\in\Lambda_N$, let $\mc B_k=\{\omega\in V(\omega_k)\cap \Xi^N: ~ D^*_k(\omega)=\emptyset\}$. Note that 
\begin{equation*}
\bigcup_{n=0}^M\Delta^n_k\setminus\bigcup_{\alpha\in\{A,B,C\}}\mc F^{N,k}_\alpha\subseteq \mc B_k\subseteq \Xi^N_k
\end{equation*}
and the first inclusion is an equality in the special case of equal densities.
The set $\mc B_k$ is formed by the configurations in $\Xi^N$ from which the process is attracted to $\omega_k$. In the same way we proved Lemma \ref{vizinhosdexi}, we see that 
\begin{equation}\label{63}
\Prob{\omega}{\beta}{H_{\mc B^c_k}<H_{\omega_k}}\leq C_0Ne^{-\beta}, ~\text { if } \omega\in\mc B_k.
\end{equation}
Let $\tilde \nu$ denote the hitting time of $\bigcup_{0<|k|<N \varepsilon /2}\mc B_k$.
Using the strong Markov property and \eqref{63}, we see that $\Prob{\omega_0}{\beta}{\tilde \nu \leq \nu \leq \tau}\leq C_0Ne^{-\beta}$ and then, decomposing the event appearing in \eqref{61} in the partition $\{\tilde \nu \leq \nu\}\cup\{\nu < \tilde \nu\}$, we see that the proof will be completed once we show that 
\begin{equation}\label{64}
\lim_{\beta\to\infty}N^2\Prob{\omega_0}{\beta}{\nu<\tilde\nu,  ~ \eta([0,\nu])\subseteq \Xi^N}=0.
\end{equation}
Observing that, for each $k\in\Lambda_N$
\begin{equation*}
\Xi^N_k\setminus\!\bigcup_{j=k-1}^{k+1}\!\mc B_j=\!\! \bigcup_{\alpha\in\{A,B,C\}}\!\!\left(\mc G^{N,k}_\alpha \cup R(\mc G^{N,k}_\alpha) \cup \{\xi^{k+1}_{\alpha,2},\xi^{k-1}_{N_\alpha-2}\}\right) \cup \!\! \bigcup_{\alpha: N_\alpha=M}\!\!\{\zeta^k_{\alpha,0},\zeta^k_{\alpha,N_\alpha}\},
\end{equation*}
we note that the only possible paths from $\omega_0$ to $\bigcup_{|k|> N\varepsilon/2}\Xi^N_k$ contained in $\Xi^N$ that avoid the set $\bigcup_{0<|k|<N \varepsilon /2}\mc B_k$ are those passing through the intermediate metastates in $\mc G^N$.
Now we will use that, starting from $\mc G^N$ the trace of $\{\eta(t):t\geq0\}$ in $\Omega^N_1$ is well approximated by the ideal process $\{\widehat\eta_1(t):t\geq0\}$ whose jump probabilities are given in Figure \ref{fig6}. A small modification of Lemma \ref{lemmaideal} is needed to justify this approximation. To arrive in $\bigcup_{|k|> N\varepsilon/2}\Xi^N_k$ we have to pass first at some configuration in $\bigcup_{|k|= \lfloor N\varepsilon/4\rfloor}\Xi^N_k$. From this point we make the same coupling as in Lemma~\ref{lemmaideal}. Observe now that from $\bigcup_{|k|= \lfloor N\varepsilon/4\rfloor}\Xi^N_k$ the ideal process must make at least $\delta N^2\varepsilon$ jumps without absorption in $\Omega^N_0$, for some constant $\delta$. This gives the bound 
\begin{equation*}
\Prob{\omega_0}{\beta}{\nu<\tilde\nu, ~ \eta([0,\nu])\subseteq \Xi^N}\leq\left(\frac{3}{5}\right)^{\delta N^2\varepsilon}+C_0N^3\beta e^{-\beta}.
\end{equation*}
So, \eqref{64} follows by \eqref{condNebeta} which imposes that $N$ increases slowly with $\beta$. This completes the proof of the proposition.
\end{proof}

Now, Theorem \ref{mainresult} follows from \eqref{tracelast}, Proposition~\ref{lastmass} and Theorem~\ref{browdrift}.

\section{Appendix}\label{appendix}

\begin{lemma}\label{comblemma}
For any integers $M$ and $1\leq i \leq M-1$,
\begin{equation}\label{eqcomblemma}
\sum_{j=1}^i\binom{M-j-1}{i-j}\left(\frac{1}{2}\right)^{M-j}
+\sum_{r=1}^{M-i}\binom{M-r-1}{M-i-r}\left(\frac{1}{2}\right)^{M-r}=1
\end{equation}
\end{lemma}

\begin{proof} Fix $M$. For any $i\in\{1,\ldots,M-1\}$ define
\begin{equation*}
\phi(i)=\sum_{j=1}^i\binom{M-j-1}{i-j}2^{j-1}.
\end{equation*}
Multiplying \eqref{eqcomblemma} by $2^{M-1}$, the equality to be proved becomes
\begin{equation*}
\phi(i)+\phi(M-i)=2^{M-1}.
\end{equation*}
We claim that, for any $1\leq i \leq M-1$, $\phi(i)$ counts the number of subsets of $\{1,\ldots,M-1\}$ with at most $i-1$ elements, which implies the above equality. The proof of this claim relies on a suitable way of classifying the elements of
\begin{equation*}
\binom{[M-1]}{\leq i-1}:=\{E\subseteq\{1,\ldots,M-1\}: |E|\leq i-1\}.
\end{equation*}
We first observe that for any $E\in\binom{[M-1]}{\leq i-1}$ there exists $1\leq k_0 \leq i$ such that
\begin{equation*}
|E\cap\{k_0,\ldots,M-1\}|=i-k_0.
\end{equation*}
To see this, note that if we define $h$ on $\{1,\ldots,i\}$ as $h(k)=i-k-|E\cap\{k,\ldots,M-1\}|$,
then $h(k+1)\in\{h(k),h(k)-1\}$, $h(1)\geq0$, $h(i)\leq0$. So, there must exist some $k_0\in\{1,\ldots,i\}$ such that $h(k_0)=0$.
Therefore, we may decompose
\begin{equation}\label{33}
    \binom{[M-1]}{\leq i-1}=\bigcup_{j=1}^i\mc D^i_j
\end{equation}
into a disjoint union, where
\begin{equation*}
\mc D^i_j:=\left\{E\in\binom{[M-1]}{\leq i-1}: j=\max\{1\leq k\leq i : |E\cap\{k,\ldots,M-1\}|=i-k\}\right\}.
\end{equation*}
Now note that if $E\in \mc D^i_j$ then $j\notin E$. Another simple argument, using again the function $h$ defined above, shows that, in fact,
\begin{equation*}
\mc D^i_j=\left\{E\subseteq\{1,\ldots,M-1\}: j\notin E,~ |E\cap\{j+1,\ldots,M-1\}|=i-j\}\right\},
\end{equation*}
and so, $|\mc D^i_j|=\binom{M-j-1}{i-j}2^{j-1}$. As the union in \eqref{33} is disjoint, summing in~$j$, from $1$ to $i$, we get that $\phi(i)$ is the cardinality of $\binom{[M-1]}{\leq i-1}$, as claimed. 
\end{proof}

\medskip
\noindent{\bf Acknowledgments.} The author is greatly indebted to his advisor Claudio Landim for his inspiring guidance and collaboration throughout the preparation of this work and to Lorenzo Bertini for suggesting the problem. The author would also like to thank  Maur\'icio Collares Neto and Milton Jara for fruitful discussions.

\end{document}